\documentclass[a4paper,10pt]{article}
\usepackage[utf8]{inputenc}
\input{drlmacros}

\hypersetup{colorlinks=true,urlcolor=blue,citecolor=blue}
\usepackage{stackrel}
\renewcommand{\attains}[1]{\stackrel[#1\,]{}{\to}}
\usepackage[shortlabels]{enumitem}

\title{Quantitative convergence analysis of iterated expansive, set-valued mappings}

\author{D. Russell Luke\thanks{Institut f\"ur Numerische und Angewandte Mathematik,
							   Universit\"at G\"ottingen, 37083 G\"ottingen, Germany.
							   DRL was supported in part by German Israeli Foundation Grant G-1253-304.6 and Deutsche Forschungsgemeinschaft Collaborative Research Center SFB755 
							   E-Mail:~\href{mailto:r.luke@math.uni-goettingen.de}{r.luke@math.uni-goettingen.de}}
		 \and
		Nguyen H. Thao\thanks{Institut f\"ur Numerische und Angewandte Mathematik,
							  Universit\"at G\"ottingen, 37083 G\"ottingen, Germany.
							  NHT was supported by German Israeli Foundation Grant G-1253-304.6.
							  E-Mail:~\href{mailto:h.nguyen@math.uni-goettingen.de}{h.nguyen@math.uni-goettingen.de}}
		 \and 
		Matthew K. Tam\thanks{Institut f\"ur Numerische und Angewandte Mathematik,
							  Universit\"at G\"ottingen, 37083 G\"ottingen, Germany.
							  MKT was supported by Deutsche Forschungsgemeinschaft Research Training Grant 2088.
							  E-Mail:~\href{mailto:m.tam@math.uni-goettingen.de}{m.tam@math.uni-goettingen.de}}
		}

\begin{document}

\maketitle

\begin{abstract}

We develop a framework for quantitative convergence analysis of Picard iterations of expansive
set-valued fixed point mappings.   There are two key components of the analysis.  The first is a natural 
generalization of single-valued averaged mappings to expansive, set-valued mappings that 
characterizes a type of strong calmness of the fixed point mapping.  The second component 
to this analysis is an extension of the well-established notion of 
metric subregularity -- or inverse calmness -- of the mapping at fixed points.  Convergence
of expansive fixed point iterations is proved using these two properties, and quantitative 
estimates are a natural byproduct of the framework.    
To demonstrate the application of the theory, we prove for the first time a number
of results showing local linear convergence of nonconvex cyclic projections for inconsistent (and consistent) 
feasibility problems, local linear convergence of the forward-backward algorithm for 
structured optimization without convexity, strong or otherwise, and local linear 
convergence of the Douglas--Rachford algorithm for structured nonconvex minimization.   
This theory includes earlier approaches for known results, convex and nonconvex, as 
special cases.
\end{abstract}

{\small \noindent {\bfseries 2010 Mathematics Subject
Classification:} {Primary 49J53, 65K10
Secondary 49K40, 49M05, 49M27, 65K05, 90C26.\\
}}


\section{Introduction}
We present a program of analysis that enables one to quantify the rate of convergence of 
sequences generated by fixed point iterations of expansive, set-valued mappings.  
The framework presented here subsumes earlier approaches for analyzing fixed point iterations of relaxed  
nonexpansive mappings and opens up new results for {\em expansive} mappings.  Our 
approach has its roots in the pioneering work of Mann, Krasnoselski, Edelstein, Gurin\footnote{We learned
from Alex Kruger that Gurin's name was misprinted as Gubin in the translation of his work into English.}, 
Polyak and Raik who wrote seminal papers in the analysis of (firmly) 
nonexpansive and averaged mappings
\cite{mann1953mean, krasnoselski1955, edelstein1966, Gubin67} although the terminology ``averaged" wasn't coined 
until sometime later \cite{BaiBruRei78}.  Our strategy is also 
indebted to the developers of notions of stability, in particular metric regularity and 
its more recent refinements \cite{Penot89, Aze06, DontchevRockafellar14, Ioffe11, Ioffe13}.
We follow a pattern of proof used in \cite{HesseLuke13} and \cite{ACL15} for Picard iterations of set-valued mappings, 
though this approach was actually inspired by  the analysis of alternating projections in \cite{Gubin67}. 

The idea is to isolate two properties of the fixed point mapping.  The first property is a generalization of the averaging
property, what we call {\em almost averaging}.   When a self-mapping is averaged and fixed points 
exist, then the Picard iteration converges to a fixed point (weakly in the Hilbert space setting) 
without any additional assumptions. (See \cite[Theorem 3]{Opial67}.  See also \cite[3. Satz]{Schaefer57} 
for the statement under the assumption that the mapping is weakly continuous.)  
In order to {\em quantify} convergence, a second property is needed.  In their 
analysis of Krasnoselski-Mann relaxed cyclic projections for convex feasibility, Gurin, Polyak 
and Raik assume that the set-intersection has interior \cite[Theorem 1]{Gubin67}.  Interiority is an assumption about 
{\em stability} of the fixed points of the mapping, and this generalizes
considerably.  Even if rates of convergence are not the primary interest, if the averaging property 
is relaxed in any meaningful way, monotonicity of Picard iterations with respect to the 
set of fixed points is lost.  In order to  recover convergence in this case, we appeal 
to stability of the set of fixed points to overcome 
the lack of monotonicity of the fixed point mapping.  The second property we require of the mapping 
is a characterization of the needed stability at fixed points.  Metric {\em subregularity} of the 
mapping at fixed points is one well-established notion that fulfills this stability and provides 
quantitative estimates for the rate of convergence of the iterates.  This is closely 
related (actually synonymous) to the existence of {\em error bounds}.  The almost averaging 
and the stability properties are defined and quantified on local neighborhoods, but 
our approach is not asymptotic.  Indeed, when convexity or nonexpansivity is assumed, these local 
neighborhoods extend to the whole space and the corresponding results are global and recover the 
classical results.   

We take care to introduce the notions of almost averaging, stability and metric subregularity, 
and to present the most general abstract results  in Section \ref{s:gen thry}.  
Almost averaged mappings are developed first in Section \ref{s:av}, after which abstract convergence 
results are presented in Section \ref{s:abs}.  In Section \ref{s:mr} the notion of metric regularity and 
its variants is presented and applied to the abstract results of Section \ref{s:abs}. 
The rest of the paper, Section \ref{s:apps}, is a tutorial on the application of these 
ideas to quantitative convergence analysis of algorithms for, respectively, nonconvex and inconsistent 
feasibility (Section  \ref{s:feas}) and structured optimization (Section \ref{s:str opt}).   
We focus our attention on just a few simple algorithms, namely cyclic 
projections, projected gradients and Douglas--Rachford.

Among the new and recent concepts are:  almost 
nonexpansive/averaged mappings (Section \ref{s:av}), which 
are a generalization of averaged mappings \cite{BaiBruRei78} and satisfy a 
type of strong calmness of  set-valued mappings; submonotonicity of set-valued self-mappings
(Definition \ref{d:sub/hypomonotone}), which is equivalent to almost firm-nonexpansiveness of 
 their resolvents (Proposition~\ref{t:firmlynonexpansive}) generalizing Minty's classical identification 
of monotone mappings with firmly-nonexpansive
 resolvents \cite{Minty62, Reich77}; elementally subregular sets (Definition \ref{d:set regularity} from 
\cite[Definition 5]{KruLukNgu16});  
subtransversality of collections of sets at points of nonintersection 
(Definition \ref{d:(s)lf}); and 
 gauge metric subregularity (Definition \ref{d:(str)metric (sub)reg} from \cite{Ioffe11, Ioffe13}).   
These objects are applied to obtain a number of new results:  
local linear convergence of nonconvex cyclic projections for inconsistent 
feasibility problems (Theorem \ref{t:cp ncvx}) with 
 some surprising special cases like two nonintersecting circles (Example \ref{eg:circles}) and practical 
(inconsistent) phase retrieval (Example \ref{eg:pr cp});  global R-linear convergence of cyclic projections
onto convex sets (Corollary \ref{t:cp cvx}); 
local linear convergence of forward-backward-type algorithms without convexity or strong monotonicity
(Theorem \ref{t:linear convergence of fb}); local linear convergence of the Douglas--Rachford 
algorithm for structured nonconvex optimization (Theorem \ref{t:RAAR pr})  and a 
specialization to the relaxed averaged alternating reflections (RAAR) algorithm \cite{Luke05a, Luke08} 
for inconsistent phase retrieval (Example \ref{eg:RAAR pr}).  

The quantitative convergence results presented here 
focus on linear convergence, but this framework is appropriate for a wider range of behaviors, 
particularly sublinear convergence.
The emphasis on linear convergence is in part due to its simplicity, but also because it is surprisingly 
prevalent in first order algorithms for common problem structures (see the discussions of phase
retrieval in Examples \ref{eg:pr cp} and \ref{eg:RAAR pr}).  To be sure, there are constants 
 that would, if known, determine the  exact rate, and these  
are either hard or impossible to calculate.  But in many instances 
the {\em order} of convergence -- linear or sublinear -- can be determined a priori.  As such, {\em a posteriori
error bounds} can be estimated in some cases, with the usual epistemological caveats, from the observed behavior of the algorithm.   For 
problems where the {\em solution} to the underlying variational problem, as opposed to its {\em optimal 
value}, is the only meaningful result of the numerical algorithm, such error bounds are essential.  
One important example is image processing with statistical constraints studied in \cite{ACL15} and \cite{LukShe17}.  
Here the images are physical measurements and solutions to the variational image processing 
problems have a quantitative statistical interpretation in terms of the experimental data.  
In contrast, the more common analysis determining that an algorithm for computing these solutions 
merely converges, or even that the {\em objective value} 
converges at a given rate, leads unavoidably to vacuous assurances. 

\subsection{Basic definitions and notation}
The setting throughout this work is a finite dimensional Euclidean space $\Ebb$.  
The norm $\|\cdot\|$ denotes the Euclidean norm.  
The open unit ball and the unit sphere in a Euclidean space are denoted $\Ball$ and $\Sbb$, respectively.
$\Ball_\delta(x)$ stands for the open ball with radius $\delta>0$ and center $x$.
We denote the extended reals by 
$\extre\equiv \Rbb\cup\{+\infty\}$.  
The domain of a function 
$\map{f}{U}{\extre}$ is defined by $\dom f= \{u\in \Ebb: f(u)<+\infty\}$.
The 
{\em subdifferential} of $f$ at $\xbar\in\dom f$, for our purposes, can be defined by 
\begin{equation}\label{e:lsd}
\sd f(\xbar)\equiv \klam{v~:~\exists v^k\to v \mbox{ and } x^k\attains{f}\xbar
\mbox{ such that }f(x)\geq f(x^k) + \ip{v^k}{x-x^k} + o(\|x-x^k\|)}.
\end{equation}
Here the notation $x^k\attains{f}\xbar$ means that $x^k\to\xbar\in\dom f$ and $f(x^k)\to f(\xbar)$.  When $f$ is convex, \eqref{e:lsd} reduces to the 
usual convex subdifferential given by 
\begin{equation}%%\label{e:sd}
   \sd f(\xbar)\equiv\set{v\in U}{\left\langle v, x-\xbar\right\rangle\leq f(x)-f(\xbar), \mbox{ for all }x\in U}.
\end{equation}%
 When $\xbar\notin \dom f$ the subdifferential is defined to be empty.  Elements of the subdifferential 
are called {\em subgradients}. 

A set-valued mapping $T$ from $\Ebb$ to another Euclidean space $\Ybb$ is denoted 
$\mmap{T}{\Ebb}{\Ybb}$ and its \emph{inverse} is given by
\begin{equation}%%\label{e:inv set}
T^{-1}(y)\equiv \set{x\in \Ebb}{y\in T(x)}.
\end{equation}% 
The mapping $\mmap{T}{\Ebb}{\Ebb}$ is said to be {\em monotone} on $\Omega\subset\Ebb$ if
\begin{equation}%
 \langle T(x)-T(y), x-y \rangle \geq 0~~~~\forall~x,y\in\Omega.
\end{equation}%
$T$ is called {\em strongly monotone} on $\Omega$ if there exists a $\tau>0$ such that
\begin{equation}%
 \langle T(x)-T(y), x-y \rangle \geq \tau \|x-y\|^2~~~~\forall~x,y\in\Omega.
\end{equation}%
A {\em maximally monotone} mapping is one whose graph cannot be augmented by any more points without violating 
monotonicity.  The subdifferential of a proper, l.s.c., convex function, for example,  is a 
maximally monotone set-valued mapping \cite[Theorem~12.17]{VA}.    
We denote the {\em resolvent} of $T$ by
$J_{T}\equiv \left(\Id+T\right)^{-1}$ where $\Id$ denotes the identity mapping. The corresponding {\em reflector} 
is defined by  $R_{T}\equiv 2J_{T}-\Id$.  A basic and fundamental fact is that the resolvent of a monotone mapping 
is firmly nonexpansive and hence single-valued \cite{Minty62, BruRei77}.
Of particular interest are {\em polyhedral} (or {\em piecewise polyhedral} \cite{VA}) mappings, that is, 
mappings $\mmap{T}{\Ebb}{\Ybb}$ whose graph is the union of finitely many sets that 
 are polyhedral convex in $\Ebb\times \Ybb$ \cite{DontchevRockafellar14}.

Notions of {\em continuity} of set-valued mappings have been thoroughly developed over the last 
$40$ years. Readers are referred to the monographs 
\cite{AubinFrankowska90, VA, DontchevRockafellar14} for basic results.
A mapping $\mmap{T}{\Ebb}{\Ybb}$ is said to be {\em Lipschitz continuous} if it is closed-valued and
there exists a $\tau\geq 0$ such that, for all $u, u' \in \Ebb$,
 \begin{equation}%
  T(u')\subset T (u) + \tau \|u'-u\| \Bbb. 
 \end{equation}%
 Lipschitz continuity is, however, too strong a notion for set-valued mappings. 
We will mostly only require {\em calmness}, which is a pointwise version of 
Lipschitz continuity.  A mapping $\mmap{T}{\Ebb}{\Ybb}$ is said to be {\em calm} 
at $\ubar$ for $\vbar$ if $(\ubar,\vbar)\in\gph T$ and there is a constant $\kappa$ 
together with neighborhoods $U\times V$ of $(\ubar,\vbar)$ such that 
%\begin{equation}\label{e:set calm}
%   e\paren{T(u)\cap V, T(\ubar)}\leq \kappa\|u-\ubar\|\quad\forall~u\in U
%\end{equation}
%where $e(\Omega_1, \Omega_2)\equiv \sup_{v\in\Omega_1}\dist\paren{v,\Omega_2}$ is the 
%{\em excess} between the sets $\Omega_1$ and $\Omega_2$.  
%\todo[caption={},inline,color=blue!20]{
% I suggest re-writing the calmness definition like the Lipchitz ones.  Then we don't need to introduce the excess function. 
% Compare this also to \emph{``3D. Outer Liptschitz Continuity" in Dontchev and Rochafellar.}
% \begin{equation}\label{e:set calm}
%   T(u)\cap V \subset T(\ubar) + \kappa\|u-\ubar\|\quad\forall~u\in U
%\end{equation}
%}
 \begin{equation}%%\label{e:set calm}
   T(u)\cap V \subset T(\ubar) + \kappa\|u-\ubar\|\quad\forall~u\in U.
\end{equation}%
When $T$ is single-valued, calmness
is just pointwise Lipschitz continuity:
\begin{equation}%\label{e:func calm}
   \norm{T(u)- T(\ubar)}\leq \kappa\|u-\ubar\|\quad\forall~u\in U.
\end{equation}%

Closely related to calmness is  
{\em metric subregularity}, which 
can be understood as the property corresponding to a calmness of the inverse mapping.  
As the name suggests, it is a weaker property than {\em metric regularity} 
which, in the case of an $n\times m$ matrix for instance ($m\leq n$), is equivalent to surjectivity. 
% For linear maps, for instance,  modulus of metric
% regularity of the mapping is the reciprocal of its smallest singular value \cite{BorweinLewis}
Our definition follows the characterization of this property given in \cite{Ioffe11, Ioffe13}, and 
appropriates the terminology of  \cite{DontchevRockafellar14} with slight but significant variations.  
The {\em graphical derivative} of a mapping $\mmap{T}{\Ebb}{\Ybb}$
at a point $(x,y)\in\gph T$ is denoted $\mmap{DT(x|y)}{\Ebb}{\Ybb}$ and defined as the mapping 
whose graph is the tangent cone to $\gph T$ at $(x, y)$ (see \cite{AUB1980} where it is called the 
contingent derivative). That is,
\begin{equation}\label{e:gder}
   v\in DT(x|y)(u)\quad\iff\quad (u,v)\in \Tcal_{\gph T}(x,y)
\end{equation}
 where $\Tcal_{\Omega}$ is the tangent cone mapping associated with the set $\Omega$ defined by
\begin{equation}
\Tcal_{\Omega}(\xbar)\equiv \set{w}{\frac{(x^k-\xbar)}{\tau}\to w\quad\mbox{ for some }\quad x^k\attains{\Omega}\xbar,~
\tau\searrow 0}.
\end{equation}
Here the notation  $x^k\attains{\Omega}\xbar$ means that the sequence of points $\{x^k\}$
approaches $\xbar$ from within $\Omega$.

The distance to a set $\Omega\subset\Ebb$ with respect to the 
bivariate function $\dist(\cdot, \cdot)$
is defined by 
\begin{equation}%
\dist(\cdot, \Omega) \colon \Ebb\to\Rbb\colon x\mapsto \inf_{y\in\Omega}\dist(x,y)
\end{equation}%
and the set-valued mapping 
\begin{equation}%
\mmap{P_\Omega}{\Ebb}{\Ebb}\colon
x\mapsto \set{y\in \Omega}{\dist(x,\Omega)=\dist(x,y)}
\end{equation}%
is the corresponding \emph{projector}.  An element $y\in P_\Omega(x)$ is called a {\em projection}.  
Closely related to the projector is the  {\em prox} mapping \cite{Moreau62}
\[
 \prox_{\lambda,f}(x)\equiv \argmin_{y\in\Euclid}\klam{f(y)+\frac{1}{2\lambda}\norm{y-x}^2}.
\]
When $f(x)=\iota_\Omega$, then $\prox_{\lambda, \iota_\Omega}=P_\Omega$ for all $\lambda>0$.  The 
value function corresponding to the prox mapping is known as the {\em Moreau} envelope, which we denote 
by $e_{\lambda, f}(x)\equiv \inf_{y\in\Euclid}\klam{f(y)+\frac{1}{2\lambda}\norm{y-x}^2}$.  When $\lambda=1$ and $f=\iota_\Omega$
the Moreau envelope is just one-half the squared distance to the set $\Omega$:  
$e_{1,\iota_\Omega}(x) = \tfrac12\dist^2(x,\Omega)$.  
The {\em inverse} projector $P^{-1}_\Omega$ is defined  by
\begin{equation}%%\label{e:inv proj}
 P^{-1}_\Omega (y)\equiv\set{x\in\Ebb}{P_\Omega(x)\ni y}.
\end{equation}%
Throughout this note we will assume the distance corresponds to the Euclidean norm, though 
most of the statements are not limited to this.  When $\dist(x,y)=\|x-y\|$ then one has the 
following variational characterization of the projector: 
 $\zbar\in P^{-1}_{\Omega}\xbar$ if and only if
\begin{equation}%%\label{e:Euclidean Projector}
 \ip{\zbar-\xbar}{x-\xbar}\leq \frac{1}{2}\norm{x-\xbar}^2\quad\forall x\in\Omega.
\end{equation}%
Following \cite{BauLukePhanWang13a}, we use this object to define the various normal cone mappings, 
which in turn lead to  the 
{\em subdifferential} of the indicator function $\iota_\Omega$.  

%We use the notation  $x^k\attains{\Omega}\xbar$ to mean that the sequence of points $\{x^k\}$
%approaches $\xbar$ from within $\Omega$. 
The \emph{$\varepsilon$-normal cone} to $\Omega$ at $\xbar\in\Omega$ is defined
\begin{equation}%%\label{e:eps-normal}
\encone{\Omega}(\xbar) \equiv \set{v}{\limsup_{x\attains{\Omega}\xbar,~x\neq\xbar}\frac{\ip{v}{x-\xbar}}{\norm{x-\xbar}}\leq \varepsilon}.
\end{equation}%
The \emph{(limiting) normal cone} to $\Omega$ at $\xbar\in\Omega$, denoted $\ncone{\Omega}\paren{\xbar}$,  
is defined as the limsup of the $\varepsilon$-normal cones. That is, a vector $v\in \ncone{\Omega}\paren{\xbar}$ if 
there are sequences $x^k\attains{\Omega}\xbar$,
 $v^k\to v$ with $v^k\in {\widehat{N}}^{\varepsilon_k}_{\Omega}\paren{x^k}$ and $\varepsilon_k\searrow 0$.  
The {\em proximal normal cone}  to $\Omega$ at
$\xbar$ is the set 
\begin{equation}%%\label{e:pnX}
\pncone{\Omega}(\xbar)\equiv \cone\paren{P_\Omega^{-1}\xbar-\xbar}.
\end{equation}%
If $\xbar\notin\Omega $, then all normal cones are defined to be empty.

The proximal normal cone need not be closed. The limiting normal cone is, of
course, closed by definition.   See \cite[Definition~1.1]{Mord06} or \cite[Definition~6.3]{VA}
(where this is called the regular normal cone) for an in-depth treatment as well as 
\cite[page~141]{Mord06} for historical notes.  When the projection is with respect to the Euclidean norm, the limiting normal 
cone can be written as the limsup of proximal normals:
\begin{equation}%%\label{e:ncone limsup pncone}
   \ncone{\Omega}(\xbar)=\overline{\lim_{x\attains{\Omega}{\xbar}}} \, \pncone{\Omega}(x).
\end{equation}%

\section{General theory: Picard iterations}\label{s:gen thry}
\subsection{Almost averaged mappings}\label{s:av}
Our ultimate goal is a quantitative statement about convergence to fixed points for set-valued mappings.  
Preparatory to this, we first must be clear what is meant by a fixed point of a set-valued mapping.  
\begin{defn}[fixed points of set-valued mappings]\label{d:fixed points}
   The set of fixed points of a set-valued mapping $\mmap{T}{\Ebb}{\Ebb}$ is defined by 
\[
\Fix T\equiv \set{x\in \Ebb}{x\in T(x)}. 
\]
\end{defn}
In the set-valued setting, it is important to keep in mind a few things that can happen
that cannot happen when the mapping is single-valued. 
\begin{eg}[inhomogeneous fixed point sets]\label{eg:inhomogen fixed points} 
Let $T\equiv P_AP_B$ where
 \begin{align*}
  A&=\set{ (x_1,x_2)\in\Rtw}{ x_2\geq -2x_1 + 3}\cap \set{ (x_1,x_2)\in\Rtw}{ x_2\geq 1},\\
  B&=\Rtw\setminus\Rtw_{++}.
 \end{align*} % *}
Here $P_B(1,1)=\klam{(0,1), (1,0)}$ and the point $(1,1)$ is a fixed point of $T$ since $(1,1)\in P_A\klam{(0,1), (1,0)}$. 
However, the point $P_A(0,1)$ is also in $T(1,1)$, and this is not a fixed point of $T$.
\hfill$\Box$
\end{eg}

To help rule out inhomogeneous fixed point sets like the one in the previous example, we introduce
the following strong calmness of fixed point mappings that is an extension of conventional 
nonexpansiveness and firm nonexpansiveness.  What we call {\em almost nonexpansive 
mappings} below were called $(S,\epsilon)$-nonexpansive mappings in \cite[Definition 2.3]{HesseLuke13}, 
and almost averaged mappings are slight generalization of $(S,\epsilon)$-firmly nonexpansive mappings
also defined there.

\begin{defn}[almost nonexpansive/averaged mappings]\label{d:ane-aa}
Let $D$ be a nonempty subset of $\Ebb$ and let $T$ be a (set-valued) mapping from $D$ to $\Ebb$.
\begin{enumerate}[(i)]
   \item  $T$ is said to be  {\em pointwise almost nonexpansive on $D$ at $ y \in D$} if there exists
        a constant $\varepsilon\in[0,1)$ such that 
\begin{eqnarray}\label{e:epsqnonexp}
&&\norm{x^+- y^+}\leq\sqrt{1+\varepsilon}\norm{x- y}\\
&&~\forall~ y^+\in T y \mbox{ and } \forall~ x^+\in Tx \mbox{ whenever }x\in D. \nonumber
\end{eqnarray}

If \eqref{e:epsqnonexp} holds with $\varepsilon=0$ then $T$ is called \emph{pointwise nonexpansive} at $y$ on $D$.

If $T$ is pointwise (almost) nonexpansive at every point on a neighborhood
%$V\subset D$ 
of $y$ (with the same violation constant 
$\varepsilon$) on $D$, then $T$ is said to be 
\emph{(almost) nonexpansive at $y$ (with violation 
$\varepsilon$) on $D$}.

If $T$ is pointwise (almost) nonexpansive  on $D$ at every point $y\in D$
(with the same violation constant 
$\varepsilon$), then $T$ is said to be 
\emph{pointwise (almost) nonexpansive on $D$ (with violation 
$\varepsilon$)}.  If $D$ is open and $T$ is pointwise (almost) nonexpansive on $D$, then 
it is (almost) nonexpansive on $D$. 

%If $T$ is  pointwise almost nonexpansive at all $z$ on a neighborhood $V\subset D$ of $y$ with $\varepsilon=0$, then $T$ is 
%called \emph{nonexpansive} at $y$ on $D$.

\item $T$ is called \emph{pointwise almost averaged  on $D$ at $y$} if
there is an averaging constant $\alpha\in (0,1)$  and a violation constant 
$\varepsilon\in[0,1)$ such that the mapping 
$\widetilde{T}$ defined by 
\[
  T = \paren{1-\alpha}\Id +\alpha\widetilde{T}
\]
is pointwise almost nonexpansive at $y$ with violation $\varepsilon/\alpha$ on $D$.  

Likewise if $\widetilde{T}$ is (pointwise) (almost) nonexpansive  on $D$ (at $y$) (with violation 
$\varepsilon$), then $T$ is said to be 
\emph{(pointwise) (almost) averaged  on $D$ (at $y$) (with averaging constant $\alpha$ and 
violation $\alpha\varepsilon$)}.

If the averaging constant $\alpha=1/2$, then $T$ is said to be {\em (pointwise) (almost) 
firmly nonexpansive  on $D$ (with violation 
$\varepsilon$) (at $y$)}.  
\end{enumerate}
\end{defn}
\noindent Note that the mapping $T$ need not be a self-mapping from $D$ to itself.
In the special case where $T$ is (firmly) nonexpansive at all points $y\in \Fix T$, 
mappings satisfying \eqref{e:epsqnonexp} are also 
called {\em quasi-\-(firmly)\-nonexpansive}  \cite{BauschkeCombettes11}. 

The term ``almost nonexpansive'' has been used for different purposes by Nussbaum \cite{Nussbaum72} 
and Rouhani \cite{Rouhani90}.  Rouhani uses the term to indicate sequences, in the 
Hilbert space setting, that are asymptotically nonexpansive.  Nussbaum's definition is the closest in spirit and 
definition to ours, except that he defines $f$ to be locally almost nonexpansive when 
$\|f(y)-f(x)\|\leq \|y-x\|+\varepsilon$.  In this context, see also \cite{Reich73}.
At the risk of some confusion, we re-purpose 
the term here.  Our definition of pointwise almost nonexpansiveness of $T$ at $\xbar$ 
is stronger than {\em calmness} \cite[Chapter 8.F]{VA} with constant $\lambda=\sqrt{1+\varepsilon}$ since the 
inequality must hold for all 
pairs $x^+\in Tx$ and $y^+\in Ty$, while for calmness the inequality would hold only for points $x^+\in Tx$ and 
their {\em projections} onto $Ty$.  We have avoided the temptation to call this property ``strong calmness'' in 
order to make clearer the connection to the classical notions of (firm) nonexpansiveness.  A theory based 
only on calm mappings, what one might call ``{\em weakly} almost averaged/nonexpansive'' operators is 
possible and would yield statements about the existence of convergent {\em selections} from 
sequences of iterated set-valued mappings.  In light of the other requirement of 
the mapping $T$ that we will explore in Section~\ref{s:mr}, namely metric subregularity, this would illuminate an
aesthetically pleasing and fundamental symmetry between requirements on $T$ and its inverse.   
We leave this avenue of investigation open.   Our development of the properties of almost averaged 
operators parallels the treatment of averaged operators in \cite{BauschkeCombettes11}.

\begin{propn}[characterizations of almost averaged operators]\label{t:average char}
   Let $\mmap{T}{\Ebb}{\Ebb}$, $U\subset\Ebb$ and let $\alpha\in (0,1)$.  The following are equivalent. 
\begin{enumerate}[(i)]
   \item\label{t:average char i} $T$ is pointwise almost averaged at $ y$ on $U$ with violation $\varepsilon$ and 
averaging constant $\alpha$.
   \item\label{t:average char ii} $\paren{1-\frac1\alpha}\Id + \frac1\alpha T$ is pointwise almost nonexpansive at $ y$ 
	on $U\subset\Ebb$ with violation $\varepsilon/\alpha$. 
   \item\label{t:average char iii} For all $x\in U,\, x^+\in T(x)$ and $y^+\in T( y)$ it holds that
    \begin{equation}%%\label{e:average char iii}
    \norm{x^+-  y^+}^2\leq \paren{1+\varepsilon}\norm{x- y}^2 - \frac{1-\alpha}{\alpha}\norm{\paren{x-x^+}-\paren{ y- y^+}}^2 .  
    \end{equation}%
%\begin{eqnarray}
%\label{e:average char iii}
%\norm{x^+-  y^+}^2&\leq& \paren{1+\varepsilon}\norm{x- y}^2 - \frac{1-\alpha}{\alpha}\norm{\paren{x-x^+}-\paren{ y- y^+}}^2    \\
%&&\forall x^+\in T(x)\quad \forall  \mbox{ whenever } x\in U.\nonumber
%\end{eqnarray}
\end{enumerate}
Consequently, if $T$ is pointwise almost averaged at $y$ on $U$ with violation $\varepsilon$ and 
averaging constant $\alpha$ then $T$ is pointwise almost nonexpansive at $y$ on $U$ with 
violation at most $\varepsilon$.
\end{propn}
\begin{proof}
   This is a slight extension of \cite[Proposition 4.25]{BauschkeCombettes11}.  
% \ref{t:average char i}$\iff$\ref{t:average char ii} is 
% immediate from the definition.  
\end{proof}
\begin{eg}[alternating projections]\label{eg:Pie}
Let $T\equiv P_AP_B$ for the closed sets $A$ and $B$ defined below. 
\begin{enumerate}[(i)]
\item If $A$ and $B$ are convex, then $T$ is nonexpansive and averaged (i.e. pointwise everywhere, no violation).
\item\label{eg:Pie2} Packman eating a piece of pizza:
 \begin{eqnarray*} % *}
  A&=&\set{ (x_1,x_2)\in\Rtw}{ x_1^2+ x_2^2\leq 1, ~-1/2x_1\leq x_2\leq x_1, x_1\geq 0}\subset\Rtw\\
  B&=&\set{ (x_1,x_2)\in\Rtw}{ x_1^2+ x_2^2\leq 1, ~x_1\leq |x_2|}\subset\Rtw.\\
\xbar&=&(0,0).
 \end{eqnarray*} % *}
% \item\label{eg:Pie1} Let 
%  \begin{eqnarray*} % *}
%   A&=&\set{ (x_1,x_2)\in\Rtw}{  x_1^2+ x_2^2\leq 1, ~2|x_2|\leq x_1, x_1\geq 0}\subset\Rtw\\
%   B&=&\set{ (x_1,x_2)\in\Rtw}{ x_1^2+ x_2^2\leq 1, ~x_1\leq |x_2|}\subset\Rtw.
%  \end{eqnarray*} % *}
The mapping $T$ is not almost nonexpansive on any neighborhood 
for any finite violation at $y=(0,0)\in \Fix T$, but it 
is {\em pointwise} nonexpansive (no violation) at $y=(0,0)$ and nonexpansive 
at all $y\in (A\cap B)\setminus\{0\}$ on small enough neighborhoods of these points.
\item\label{eg:Pie3} $T$ is pointwise averaged at $(1,1)$ when
 \begin{eqnarray*} % *}
  A&=&\set{ (x_1,x_2)\in\Rtw}{ x_2\leq 2x_1 - 1}\cap \set{ (x_1,x_2)\in\Rtw}{x_2\geq\tfrac{1}{2}x_1+\tfrac{1}{2}}\\
  B&=&\Rtw\setminus\Rtw_{++}.
 \end{eqnarray*} % *}
This illustrates that whether or not $A$ and $B$ have points in common is not relevant to the property. 
\item\label{eg:Pie4} $T$ is not pointwise almost averaged at $(1,1)$ for any $\varepsilon>0$ when
 \begin{eqnarray*} % *}
  A&=&\set{ (x_1,x_2)\in\Rtw}{ x_2\geq -2x_1 + 3}\cap \set{ (x_1,x_2)\in\Rtw}{ x_2\geq 1}\\
  B&=&\Rtw\setminus\Rtw_{++}
 \end{eqnarray*} % *}
In light of Example~\ref{eg:inhomogen fixed points}, this shows that 
the pointwise almost averaged property is incompatible with  
inhomogeneous fixed points (see Proposition~\ref{t:single-valued paa}). 
% \item\label{eg:Borwein0} A circle and a line:
%  \begin{eqnarray*} % *}
%   A&=&\set{ (x_1,x_2)\in\Rtw}{ x_2=\sqrt{2}/2}\subset\Rtw\\
%   B&=&\set{ (x_1,x_2)\in\Rtw}{ x_1^2+x_2^2=1}\\
% \xbar&=&(0,\sqrt{2}/2).
%  \end{eqnarray*} % *}
% $T$ is pointwise almost averaged at $(0, \sqrt{2}/2)$ with violation $1$ on $A$. 
% Closer inspection, however, reveals that the point $(0, \sqrt{2}/2)$ is not {\em stable} since any perturbation
% in the $x_1$ direction will cause the images under $T$ to move {\em away} from $(0, \sqrt{2}/2)$.  
% The pointwise almost averaged property does not distinguish stable from unstable fixed points.  
\endproof
\end{enumerate}
\end{eg}
%\todo[inline,caption={},color=magenta!50]{
%\begin{propn}[single-valuedness of almost averaged mappings]\label{t:single-valued paa}
%If $\mmap{T}{\Ebb}{\Ebb}$ is pointwise almost averaged at $\xbar \in 
%\Fix T$ on a neighborhood $U$ of $\xbar$, then $T$ is single-valued there.
%\end{propn}
%\begin{proof}
% Using the characterization of Proposition~\ref{t:average char}\ref{t:average char iii}, at  
% $\xbar\in\Fix T$ it holds that 
%\[
%    \norm{x^+-  \xbar^+}^2\leq \paren{1+\varepsilon}\norm{x-\xbar}^2 - \frac{1-\alpha}{\alpha}\norm{\paren{x-x^+}-
%    \paren{ \xbar- \xbar^+}}^2 
%\]
%for all $x\in U,\, x^+\in T(x)$ and $\xbar^+\in T( \xbar)$.  In particular, it must hold for $x=\xbar$ and $x^+=\xbar$, which 
%yields
%\[
%    0\leq \norm{x^+-  \xbar^+}^2=\norm{\xbar-  \xbar^+}^2\leq  - \frac{1-\alpha}{\alpha}\norm{\xbar-\xbar^+}^2 \leq 0.
%\]
%The rightmost inequality is strict for all $\xbar^+\in T\xbar$ with $\xbar^+\neq\xbar$, which leads
%to a contradiction.  Hence it must hold that $T\xbar=\{\xbar\}$.  
%\end{proof}
%}
%\todo[inline,caption={}]{
% version 160826 has proof for almost averaged mappings	
\begin{propn}[pointwise single-valuedness of pointwise almost nonexpansive mappings]\label{t:single-valued paa}
	If $\mmap{T}{\Ebb}{\Ebb}$ is pointwise almost nonexpansive on $D\subseteq\Ebb$ at $\xbar\in D$ with 
violation $\varepsilon\geq 0$, then $T$ is single-valued at $\bar{x}$. In particular, if $\xbar\in\Fix T$ 
(that is $\bar{x}\in T\bar{x}$) then $T\xbar=\{\xbar\}$.
\end{propn}
\begin{proof}
	By the definition of pointwise nonexpansive on $D$ at $\xbar$, it holds that
	\[
	\norm{x^+-  \xbar^+}\leq \sqrt{1+\varepsilon}\norm{x-\xbar}
	\]
	for all $x\in D,\, x^+\in T(x)$ and $\xbar^+\in T( \xbar)$.  In particular, setting $x=\xbar$ gives
	yields
	\[
	\norm{x^+-  \xbar^+}\leq \sqrt{1+\varepsilon}\norm{\xbar-\xbar}=0.
	\]
	That is, $x^+=\xbar^+$ and hence we conclude that $T$ is single-valued at $\xbar$. 
\end{proof}
%}
\begin{eg}[pointwise almost nonexpansive mappings not single-valued on neighborhoods]
\label{eg:cross1}
Although a pointwise almost nonexpansive mapping is single-valued at the reference point, it need not be 
single-valued on neighborhoods of the reference points.  
Consider, for example, the coordinate axes in $\Rtw$, 
 \[
    A=\mathbb{R}\times\{0\}\cup \{0\}\times\mathbb{R}.
 \]
The metric projector $P_A$ is single-valued and 
even pointwise nonexpansive (no almost) at every point in 
$A$, but multivalued on $L\equiv\set{(x,y)\in\Rtw\setminus\{0\}}{|x|=|y|}$.
\end{eg}

Almost firmly nonexpansive mappings have particularly convenient characterizations.
In our development below and thereafter we use the set $S$ to denote the collection 
of points at which the property holds.  This is useful for distinguishing points where the 
regularity holds from other distinguished points, like fixed points.  In Section \ref{s:mr}
the set $S$ is used to isolate a subset of fixed points.  The idea here is that 
the properties needed to quantify convergence need not hold on the space where 
a problem is formulated, but may only hold on a subset of this space where 
the iterates of a particular algorithm may be, naturally, confined.  This is used in 
\cite{ACL15} to achieve linear convergence results for the alternating directions method of 
multipliers algorithm.  Alternatively, $S$ can also include points that are not fixed points
of constituent operators in an algorithm, but 
are closely related to fixed points.  One example of this is {\em local best approximation points}, 
that is, points in one set that are locally nearest to another.  In section \ref{s:feas} we will need to 
quantify the violation of the averaging property for a projector onto a nonconvex set $A$ at 
points in another set, say $B$, that are locally {\em nearest points} to $A$.  This will 
allow us to tackle {\em inconsistent feasibility} where the alternating projections iteration 
converges not to the intersection, but to local best approximation points.

\begin{propn}[almost firmly nonexpansive mappings]\label{t:firmlynonexpansive}
Let $S\subset U\subset\Ebb$ be nonempty and $\mmap{T}{U}{\Ebb}$. The following are equivalent.
\begin{enumerate}[(i)]
 \item\label{t:averaged1}
 $T$ is pointwise almost firmly nonexpansive  on $U$ at all $y\in S$ with violation $\varepsilon$. 
 \item\label{t:averaged2}
The mapping $\mmap{\Ttilde}{U}{\Ebb}$ given by
\begin{equation}%
\Ttilde x\equiv(2Tx-x)\quad \forall x\in U 
\end{equation}%
 is pointwise almost nonexpansive  on $U$ at all $y\in S$ with violation $2\varepsilon$,
that is, $T$ can be written as
 \begin{equation}%%\label{e:charquasifirm}
 T x =\frac{1}{2}\left(x +\Ttilde x\right)\quad \forall x\in U.
\end{equation}%
\item\label{t:firmlynonexpansive angle} 
$\norm{x^+-y^+}^2\leq\frac{\veps}{2}\norm{x-y}^2+\langle x^+-y^+,x-y\rangle$ for all $x^+\in Tx$, and 
all $y^+\in Ty$ at each $y\in S$ whenever $x\in U$.
\item\label{t:submonotone} Let $\mmap{F}{\Ebb}{\Ebb}$ be a mapping whose resolvent is $T$, 
{\em i.e.,} $T=\paren{\Id + F}^{-1}$.  At 
each $x\in U$, for all $u\in Tx$, $y\in S$ and $v\in Ty$, the points $(u,z)$ and $(v,w)$ are in $\gph F$ 
where $z=x-u$ and $w=y-v$, and satisfy 
\begin{equation}%%\label{e:submonotone'}
 -\frac{\varepsilon}{2}\norm{(u+z)-(v+w)}^{2}\leq \ip{z-w}{u-v}.
\end{equation}%
\end{enumerate}
\end{propn}

\begin{proof}
\ref{t:averaged1}$\iff$\ref{t:averaged2}: Follows from Proposition~\ref{t:average char} when $\alpha=1/2$.

\ref{t:averaged2}$\implies$\ref{t:firmlynonexpansive angle}:
Note first that at each $x\in U$ and $y\in S$
\begin{subequations}
\label{e:KEXP}
\begin{eqnarray}
 \norm{\left(2 x^+- x\right)-\left(2 y^+- y\right)}^2
 =4\norm{x^+ - y^+}^2-4\ip{x^+-y^+}{x-y} +\norm{x-y}^2\label{e:fund}
\end{eqnarray}
for all $x^+\in Tx$ and $y^+\in Ty$.  
Repeating the definition of pointwise almost nonexpansiveness of $2T-\Id$ at $y\in S$ with 
violation $2\varepsilon$ on $U$, 
\begin{equation}
 \norm{\left(2 x^+- x\right)-\left(2 y^+- y\right)}^2\leq\paren{1+2\veps}\norm{x-y}^2.\label{e:drive}
\end{equation}
\end{subequations}
Together \eqref{e:KEXP} yields 
\[
\norm{x^+-y^+}^2\leq\frac{\veps}{2}\norm{x-y}^2+\langle x^+-y^+,x-y\rangle
\]
as claimed. 

\ref{t:firmlynonexpansive angle}$\implies $\ref{t:averaged2}:  Use \eqref{e:fund}
to replace $\ip{x^+-y^+}{x-y}$ in \ref{t:firmlynonexpansive angle} and rearrange the resulting 
inequality to conclude that $2T-\Id$ is pointwise almost nonexpansive at $y\in S$ with 
violation $2\varepsilon$ on $U$.  

\ref{t:submonotone}$\iff$\ref{t:firmlynonexpansive angle}:  First, note that 
$(u,z)\in\gph F$ if and only if  $\paren{u+z, u}\in \gph\paren{\Id+F}^{-1}$.  From this it follows that 
for  $u\in Tx$ and $v\in Ty$, the points $(u,z)$ and $(v,w)$ with $z=x-u$ and $w=y-v$,
are in $\gph F$.  So starting with 
\ref{t:firmlynonexpansive angle}, at each $x\in U$ and $ y\in S$, 
\begin{eqnarray}
 \label{e:fne angle} 
 \norm{u-v}^2&\leq&\frac{\veps}{2}\norm{x-y}^2+\langle u-v,x-y\rangle \\
 \label{e:submon1}&=& \frac{\veps}{2}\norm{(u+z)-(v+w)}^2+\langle u-v, (u+z)-(v+w)\rangle
\end{eqnarray}
for all $u\in Tx$ and $v\in Ty$.  
Separating out $ \norm{u-v}^2$ from the inner product on the left hand side of \eqref{e:submon1}
yields the result.  
\end{proof}

Property~\ref{t:submonotone} of Proposition~\ref{t:firmlynonexpansive}
is a type of {\em submonotonicity} of the mapping $F$ on $D$ with respect to $S$.
We use this descriptor to distinguish this notion from another well-established 
property known as {\em hypomonotonicity} \cite{PolRockThib00}.
\begin{defn}[(sub/hypo)monotone mappings]\label{d:sub/hypomonotone}$~$
\begin{enumerate}[(a)]
 \item\label{d:submonotone} A mapping $\mmap{F}{\Ebb}{\Ebb}$ is {\em pointwise submonotone at $\vbar$} 
if there is a constant 
$\tau$ together with a neighborhood $U$ of $\vbar$ such that
\begin{equation}\label{e:submonotone}
 -\tau\norm{(u+z)-(\vbar+w)}^{2}\leq \ip{z-w}{u-\vbar}\quad \forall z\in Fu, ~\forall u\in U, ~\forall w\in F\vbar.
\end{equation}
The mapping $F$ is said to be {\em submonotone on $U$} if 
\eqref{e:submonotone} holds for all $\vbar$ on $U$.  
\item\label{d:p-hypomonotone} The mapping $\mmap{F}{\Ebb}{\Ebb}$ is said to be {\em pointwise 
hypomonotone at $\vbar$
with
% violation
constant $\tau$ on $U$} if 
\begin{equation}
 \label{e:p-hypomonotone} -\tau\norm{u-\vbar}^2\leq \ip{z-w}{u-\vbar}\quad \forall~ z\in Fu, ~\forall u\in U, 
 ~\forall w\in F\vbar.  
\end{equation}
If \eqref{e:p-hypomonotone} holds for all $\vbar\in U$ then $F$ is said to be hypomonotone 
with %violation
constant $\tau$ on $U$.
\end{enumerate}
\end{defn}
In the event that $T$ is in fact firmly nonexpansive
(that is, $S=D$ and $\tau=0$) then Proposition~\ref{t:firmlynonexpansive}\ref{t:submonotone}
just establishes the well known equivalence between monotonicity of a  mapping and 
firm nonexpansiveness of its resolvent \cite{Minty62}.  
 Moreover, if a single-valued mapping $\map{f}{\Ebb}{\Ebb}$ is calm at $\vbar$ with calmness modulus $L$, then it is 
pointwise hypomonotone at $\vbar$ with violation at most $L$. Indeed,
\begin{equation}%
\ip{u - \vbar}{f\left(u\right) - f\left(\vbar\right)} \geq -\norm{u - \vbar}\norm{f\left(u\right) - f\left(\vbar\right)}
\ge -L\norm{u - \vbar}^{2}.
\end{equation}%
This also points to a relationship to 
{\em cohypomonotonicity} developed in \cite{CombettesPennanen04}.  More recently the notion of 
pointwise quadratically supportable functions was introduced \cite[Definition 2.1]{LukShe17};  for 
smooth functions, this class -- which is not limited to convex functions -- was shown to include functions 
whose gradients are pointwise strongly monotone (pointwise hypomonotone with constant $\tau<0$)
\cite[Proposition 2.2]{LukShe17}.  A deeper investigation of the relationships between these different 
notions is postponed to future work. 

The next result shows the inheritance of the averaging property under compositions and averages 
of averaged mappings.

\begin{proposition}[compositions and averages of relatively averaged operators]\label{t:av-comp av}
Let $\mmap{T_j}{\Ebb}{\Ebb}$ for $j=1,2,\dots,m$  be pointwise almost averaged on $U_j$ at all 
$ y_{j}\in S_j\subset\Ebb$ 
with violation $\varepsilon_j$ and 
averaging constant $\alpha_j\in (0, 1)$ where 
$U_j\supset S_j$ for $j=1,2,\dots,m$.
\begin{enumerate}[(i)]
\item\label{t:av-comp av i} If $U\equiv U_1=U_2=\dots=U_m$ and $S\equiv S_1=S_2=\cdots=S_m$ then 
the weighted mapping 
$T:=\sum_{j=1}^m w_jT_j$ with weights $w_j\in[0,1]$, $\sum_{j=1}^mw_j=1$, is pointwise almost averaged at 
all $ y\in S$ 
with violation $\varepsilon=\sum_{j=1}^mw_j\varepsilon_j$ and averaging constant $\alpha=\max_{j=1, 2, \dots, m}\klam{\alpha_j}$ 
on $U$.
\item\label{t:av-comp av ii}  If 
$T_{j}U_{j}\subseteq U_{j-1}$ and $T_{j}S_{j}\subseteq S_{j-1}$ for $j=2,3,\dots,m$, 
 then the composite mapping 
$T:=T_1\circ T_{2}\circ \cdots \circ  T_m$ is pointwise almost nonexpansive at all $y\in S_m$ on $U_m$ with violation 
at most
\begin{equation}\label{e:composit violation}
 \varepsilon = \prod_{j=1}^m\paren{1+\varepsilon_j}-1.
\end{equation}
\item\label{t:av-comp av iii}  If 
$T_{j}U_{j}\subseteq U_{j-1}$ and $T_{j}S_{j}\subseteq S_{j-1}$ for $j=2,3,\dots,m$, 
then the composite mapping $T:=T_1\circ T_{2}\circ \cdots \circ  T_m$ is pointwise almost averaged at all 
$y\in S_m$ on $U_m$ with violation at most $\varepsilon$ given by \eqref{e:composit violation}
%\begin{equation}\label{e:av comp violation}
%\varepsilon= \prod_{j=1}^m\paren{1+\varepsilon_j}-1
%% \varepsilon_m +  \sum_{i=0}^{m-2}\paren{\varepsilon_{m-1-i}\paren{\prod_{j=m-i}^m \paren{1+\varepsilon_j}}}
%\end{equation}
and averaging constant at least
\begin{equation}%%\label{e:alpha comp}
 \alpha = \frac{m}{m-1 + \frac{1}{\max_{j=1,2,\dots,m}\klam{\alpha_j}}}.
\end{equation}%
%on $U_1$.  
\end{enumerate}
\end{proposition}
\begin{proof}  
Statement~\ref{t:av-comp av i} is a formal generalization of 
\cite[Proposition~4.30]{BauschkeCombettes11} and follows directly from convexity of the squared 
norm and Proposition \ref{t:average char}\ref{t:average char iii}.

Statement~\ref{t:av-comp av ii} follows from applying the definition of almost nonexpansivity to each of 
the operators $T_j$ inductively, from $j=1$ to $j=m$.  

Statement~\ref{t:av-comp av iii} is formal generalization of \cite[Proposition~4.32]{BauschkeCombettes11} 
and follows from more or less the same pattern of proof.  Since it requires a little more care, the proof is given here.  
Define $\kappa_j:=\alpha_j/(1-\alpha_j)$ and set $\kappa=\max_{j}\klam{\kappa_j}$.
Identify 
$y_{j-1}$ with any $ y^+_{j}\in T_{j}y_{j}\subseteq S_{j-1}$ 
for $j=2,3,\dots,m$ and choose any $y_m\in S_m$.  Likewise, identify  
$x_{j-1}$ with any $x^+_{j}\in T_{j}x_{j}\subseteq U_{j-1}$ 
for $j=2,3,\dots,m$ and choose any $x_m\in U_m$.
Denote 
$u^+ \in T_1\circ T_{2}\circ \cdots\circ  T_m u$ for $u\equiv x_m$
and $ v^+\in T_1\circ T_{2}\circ \cdots\circ  T_mv$ for $v\equiv y_m$.
By convexity of the squared norm 
and Proposition~\ref{t:average char}\ref{t:average char iii} one has
\begin{eqnarray*}
	&&\frac{1}{m} \norm{\paren{u-u^+}- \paren{v- v^+}}^2 \nonumber\\%= \norm{\paren{x_1-u^+}-\paren{ y_1- v^+}}^2\nonumber\\
	&&\qquad\leq  \norm{\paren{x_{1}-u^+}-\paren{y_{1}- v^+}}^2 + 
	\norm{\paren{x_{2} - x_{1}}-\paren{ y_{2} -  y_{1}}}^2\nonumber\\
	&&	\qquad \qquad + \cdots+
	\norm{\paren{x_{m} - x_{m-1}}-\paren{ y_{m} -  y_{m-1}}}^2\nonumber\\
	&&\qquad\leq \kappa_1\paren{\paren{1+\varepsilon_1}\norm{x_{1}- y_{1}}^2-\norm{u^+- v^+}^2}\nonumber\\
	&&\qquad\qquad+\kappa_{2}\paren{\paren{1+\varepsilon_{2}}\norm{x_{2}- y_{2}}^2-
		\norm{x_{1}- y_{1}}^2}+ \cdots\nonumber\\
	&&	\qquad \qquad +
	\kappa_{m}\paren{\paren{1+\varepsilon_{m}}\norm{u- v}^2-\norm{x_{m-1}- y_{m-1}}^2}.
\end{eqnarray*}
Replacing $\kappa_j$ by $\kappa$ yields  
\begin{equation}
\frac{1}{m} \norm{\paren{u-u^+}-\paren{ v- v^+}}^2\leq 
\kappa\left(\paren{1+\varepsilon_m}\norm{u- v}^2-\norm{u^+- v^+}^2
% \right.\nonumber\\
% &&\qquad\qquad \left.
+\sum_{i=1}^{m-1}\varepsilon_{i}\norm{x_{i}- y_{i}}^2\right),
\label{e:average char interim1}
\end{equation}
From part \ref{t:average char ii} one has
\[
\norm{x_{i}- y_{i}}^2=\norm{x^+_{i+1}- y^+_{i+1}}^2\leq \paren{\prod_{j=i+1}^{m} \paren{1+\varepsilon_j}}\norm{u- v}^2,\; i=1,2,\ldots,m-1
\]
so that 
\begin{equation}\label{e:average char interim2}
 \sum_{i=1}^{m-1}\varepsilon_{i}\norm{x_{i}- y_{i}}^2\leq 
 \paren{\sum_{i=1}^{m-1}\varepsilon_{i}\paren{\prod_{j=i+1}^{m}\paren{1+\varepsilon_j}}}\norm{u- v}^2.
\end{equation}
Putting \eqref{e:average char interim1} and  \eqref{e:average char interim2} together yields  
\begin{eqnarray}
&&\frac{1}{m} \norm{\paren{u- u^+}-\paren{ v -  v^+}}^2\leq\nonumber\\
&& \kappa\left(\paren{1+\varepsilon_m +  \sum_{i=1}^{m-1}\varepsilon_{i}
\paren{\prod_{j=i+1}^{m} \paren{1+\varepsilon_j}}}\norm{u- v}^2-\norm{u^+- v^+}^2
\right).
% \right.\nonumber\\
% &&\qquad\qquad \left.
%\label{e:average char interim3}
\end{eqnarray}
The composition $T$ is therefore almost averaged with violation 
\[
\varepsilon=\varepsilon_m +  
\sum_{i=1}^{m-1}\varepsilon_{i}\paren{\prod_{j=i+1}^{m} \paren{1+\varepsilon_j}}
\]
 and averaging 
constant $\alpha = m/(m+1/\kappa)$. 
Finally, an induction argument shows that 
\[
\varepsilon_m +  
\sum_{i=1}^{m-1}\varepsilon_{i}\paren{\prod_{j=i+1}^{m} \paren{1+\varepsilon_j}}
 = \prod_{j=1}^m\paren{1+\varepsilon_j}-1,
\]
which is the claimed violation.  
\end{proof}

\begin{remark}
	We remark that Proposition~\ref{t:av-comp av}\ref{t:av-comp av ii} holds 
	in the case when $T_j$ ($j=1,2,\dots,m$) are merely pointwise almost nonexpansive. 
	The counterpart for $T_j$ ($j=1,\dots,m$) pointwise almost nonexpansive to 
	Proposition~\ref{t:av-comp av}\ref{t:av-comp av i} is given by allowing $\alpha=0$.
	% in \ref{t:av-comp av i}.
\end{remark}

\begin{corollary}[Krasnoselski--Mann relaxations]\label{t:KM ane}
	Let $\lambda\in[0,1]$ and define  $T_\lambda:=\paren{1-\lambda}\Id + \lambda T$ for $T$ pointwise almost averaged at 
	$ y$ with violation $\varepsilon$
	and averaging constant $\alpha$ on $U$.  Then $T_\lambda$ is pointwise almost averaged at $ y$ with violation 
	$\lambda\varepsilon$ and averaging constant $\alpha$ on $U$.   In particular,  when $\lambda=1/2$ the 
	mapping $T_{1/2}$ is pointwise almost firmly nonexpansive at $ y$ with violation $\varepsilon/2$ on $U$. 
\end{corollary}
\begin{proof}
	Noting that $\Id$ is averaged everywhere on $\Ebb$
	with zero violation and all averaging constants $\alpha\in(0,1)$, the statement is 
	an immediate specialization of Proposition \ref{t:av-comp av}\ref{t:av-comp av i}.
\end{proof}

\noindent A particularly attractive consequence of Corollary \ref{t:KM ane} is that the violation of almost averaged 
mappings can be mitigated by taking smaller steps via Krasnoselski-Mann relaxation.

To conclude this section we prove the following lemma, a special case of which will be required in Section~\ref{s:cp}, 
which relates the fixed point set of the composition of pointwise almost averaged operators to the corresponding 
\emph{difference vector}.   
\begin{defn}[difference vectors of composite mappings]
 	For a collection of operators $T_j:\Ebb\setto\Ebb$  ($j=1,2,\dots,m$) and  $T\equiv T_1\circ T_2\circ\dots\circ T_m$ 
        the set of \emph{difference vectors of $T$ at $u$} is given by the mapping $\mathcal{Z}:\Ebb\setto \Ebb^m$ defined by
	\begin{align}%\label{e:diff vector mapping}
		\Zcal(u) &\equiv \set{\zeta\equiv z-\Pi z}{ z\in W_0\subset\Ebb^m, ~ z_1=u}, 
	\end{align}
where $\Pi:z=(z_1, z_2, \dots, z_m)\mapsto (z_2, \dots, z_m, z_1)$ is the permutation mapping on the 
product space $\Ebb^m$ for $z_j\in \Ebb$ $(j=1,2,\dots, m)$ and  
	\begin{align*}
		W_0 &\equiv\set{x=(x_1,\dots,x_m)\in\Ebb^m}%
		%{x_1\in\range T_1, ~ x_{j-1}\in T_{j-1}(x_j)~j=2,3,\dots,m
		%	\und x_m\in T_m x_1}.
		{x_m\in T_m x_1,\, x_{j}\in T_{j}(x_{j+1}),~j=1,2,\dots,m-1}.
	\end{align*}

\end{defn}

\begin{lemma}[difference vectors of averaged compositions]\label{l:difference vector averaged operators}
	Given a collection of operators $T_j:\Ebb\setto\Ebb$  ($j=1,2,\dots,m$), set $T\equiv T_1\circ T_2\circ\dots\circ T_m$.
	Let $S_0\subset\Fix T$, let $U_0$ be a neighborhood of $S_0$ and define 
	$U\equiv \set{z = (z_1,z_2,\ldots,z_m)\in W_0}{z_1\in U_0}$.  Fix
	$\bar{u}\in S_0$ and  the difference vector $\zetabar\in\mathcal{Z}(\bar{u})$ with $\zetabar=\zbar-\Pi\zbar$ 
	for the point $\zbar = (\zbar_1,\zbar_2,\ldots,\zbar_m)\in W_0$ having $\zbar_1=\bar{u}$.	
	Let $T_j$ be pointwise almost averaged at $\zbar_j$ with violation $\varepsilon_j$ and averaging 
	constant $\alpha_j$ on $U_j\equiv p_j(U)$ where $p_j:\Ebb^m\to\Ebb$ denotes the $j$th coordinate 
	projection operator  ($j=1,2,\dots,m$).
	Then, for $u\in S_0$ and $\zeta\in\mathcal{Z}(u)$ with 
	$\zeta=z-\Pi z$ for $z= (z_1,z_2,\ldots,z_m)\in W_0$ having $z_1=u$,
	\begin{equation} \label{eq:displacement vector lemma}
	\frac{1-\alpha}{\alpha}\|\zetabar-\zeta\|^2 \leq \sum_{j=1}^m\varepsilon_j \|\zbar_j-z_j\|^2
	\text{ where }\alpha = \max_{j=1,2,\dots,m}\alpha_j.
	\end{equation}
	If the mapping $T_j$ is in fact pointwise averaged at $\zbar_j$ on $U_j$ ($j=1,2,\dots,m$), 
	then the set of difference vectors of $T$ is a singleton and independent of the initial point; that is, there exists 
	$\zetabar\in\Ebb^m$ such that $\mathcal{Z}(u)=\{\zetabar\}$ for all $u\in S_0$.
\end{lemma}
\begin{proof} 
First observe that, since $\zetabar\in\mathcal{Z}(\bar{u})$, there exists 
$\zbar = (\zbar_1,\zbar_2,\ldots,\zbar_m)\in W_0$ with $\zbar_1=\bar{u}$ such 
that $\zetabar = \zbar-\Pi\zbar$, hence $U$, and thus 
$U_j=p_j(U)$, is nonempty since it at least contains $\zbar$ 
(and $\zbar_j\in U_j$ for $j=1,2,\ldots, m$). 
Consider a second point $u\in S_0$ and let 
$\zeta\in\mathcal{Z}(u)$. Similarly, there exists $z= (z_1,z_2,\ldots,z_m)\in 
W_0$ such that $z_1=u$ and $\zeta = z-\Pi z\in U$. For each $j=1,2,\dots,m$, we 
therefore have that
	\begin{equation}%
		\|(\zbar_j-\zbar_{j-1})-(z_j-z_{j-1})\| = \|\zetabar_j-\zeta_j\|,
	\end{equation}%
	and, since $T_j$ is pointwise almost averaged at $\zbar_j$ with constant $\alpha_j$ 
	and violation $\varepsilon_j$ on $U_j$, 
	\begin{equation}%
		\|\zbar_j-z_j\|^2 + \frac{1-\alpha_j}{\alpha_j}\|\zetabar_j-\zeta_j\|^2 \leq (1+\varepsilon_j) \|\zbar_{j-1}-z_{j-1}\|^2 ,
	\end{equation}%
	where $\zbar_{0}:=\zbar_m$ and $z_{0}=z_m$. Altogether this yields
	\begin{align*}
		\frac{1-\alpha}{\alpha}\|\zetabar-\zeta\|^2 
		\leq \sum_{j=1}^m \frac{1-\alpha_j}{\alpha_j}\|\zetabar_j-\zeta_j\|^2 
		\leq \sum_{j=1}^m\left((1+\varepsilon_j) \|\zbar_{j-1}-z_{j-1}\|^2 - \|\zbar_j-z_j\|^2\right) 
		= \sum_{j=1}^m\varepsilon_j \|\zbar_j-z_j\|^2,
	\end{align*}
	which proves \eqref{eq:displacement vector lemma}. 
	If in addition, for all $j=1,2,\ldots, m$, the mappings $T_j$ are  pointwise averaged, then 
	$\varepsilon_1=\varepsilon_2=\dots=\varepsilon_m=0$, and the proof is complete.
\end{proof}

\subsection{Convergence of Picard iterations}\label{s:abs}
The next theorem serves as the basic template for the quantitative convergence analysis of fixed point iterations and generalizes 
\cite[Lemma~3.1]{HesseLuke13}.  By the notation $\mmap{T}{\Lambda}{\Lambda}$ where $\Lambda$ is a subset 
or an affine subspace of $\Ebb$, we mean that $\mmap{T}{\Ebb}{\Ebb}$ and $T(x)\subset \Lambda$ for all 
$x\in \Lambda$.  This simplification of notation should not lead to any confusion if one keeps in mind 
that there may exist fixed points of $T$ that are not in $\Lambda$.
For the importance of the use of $\Lambda$ in isolating the desirable fixed point, we refer the 
reader to \cite[Example 1.8]{ACL15}.

\begin{thm}\label{t:Tconv}
   Let $\mmap{T}{\Lambda}{\Lambda}$ for $\Lambda\subset\Ebb$ and let $S\subset\reli \Lambda$ be closed and nonempty with 
   $Ty\subset\Fix T\cap S$ for all 
   $y\in S$.  
   Let $\Ocal$ be a neighborhood of $S$ such that 
   $\Ocal\cap \Lambda\subset \reli \Lambda$.  
   Suppose
\begin{enumerate}[(a)]
   \item\label{t:Tconv a} T is pointwise almost averaged at all points $y\in S$ 
   with violation $\varepsilon$ and averaging constant $\alpha\in (0,1)$ on $\Ocal\cap \Lambda$, and 
\item\label{t:Tconv b} there exists a neighborhood $\Vcal$ of $\Fix T\cap S$ and a $\kappa>0$, 
such that for all $y^+\in Ty, ~y\in S,$ and all  $x^+\in Tx$  the estimate
\begin{equation}\label{e:doughnut mreg}
\dist(x,S)\leq \kappa    \|\paren{x-x^+}-\paren{y-y^+}\| 
\end{equation}
holds 
whenever  $x\in \paren{\Ocal\cap \Lambda}\setminus \paren{\Vcal\cap \Lambda}$.
% when $S_\deltabar\cap\paren{\Fix T+\delta\Ball}\neq S_\deltabar$.
\end{enumerate}
Then for all $x^+\in Tx$ 
\begin{equation}\label{e:contraction}
   \dist\paren{x^+,\Fix T\cap S}\leq\sqrt{1+\varepsilon - \frac{1-\alpha}{\kappa^2\alpha}} \dist(x,S) 
\end{equation}
whenever  $x\in \paren{\Ocal\cap \Lambda}\setminus \paren{\Vcal\cap \Lambda}$.

In particular, if $ \kappa<\sqrt{\frac{1-\alpha}{\varepsilon\alpha}}$,  then 
for all $x^0\in \Ocal\cap \Lambda$  the iteration 
$x^{j+1}\in Tx^j$ satisfies 
\begin{equation}\label{e:contraction b}
   \dist\paren{x^{j+1},\Fix T\cap S}\leq c^j \dist(x^0,S) 
\end{equation}
with $c\equiv \paren{1+\varepsilon-\frac{1-\alpha}{\alpha\kappa^2}}^{1/2}<1$
for all $j$ such that $x^{i}\in \paren{\Ocal\cap \Lambda}\setminus \paren{\Vcal\cap \Lambda}$ for $i=1,2,\dots,j$. 
 \end{thm}

 Before presenting the proof, some remarks will help clarify the technicalities.  The role of assumption 
 \ref{t:Tconv a} is clear in the two-property scheme we have set up.  The second assumption \ref{t:Tconv b} is a 
characterization of the required stability of the fixed points and their preimages.  It is helpful to consider a 
specialization of this assumption which simplifies things considerably.  
 First, by Proposition \ref{t:single-valued paa}, since $T$ is almost averaged at all points in 
 $S$, then it is single-valued  there and one can simply write $Ty$ for all $y\in S$ instead of $y^+\in Ty$.  The 
 real simplification comes when one considers the case $S=\Fix T$.  In this case $Ty=y$ for all 
 $y\in S$ and condition \eqref{e:doughnut mreg} simplifies to 
\begin{equation}%%\label{e:doughnut mreg2}
\dist(x,\Fix T)\leq \kappa   \dist(0, x-Tx) \quad\iff\quad \dist(x,\Phi^{-1}(0))\leq \kappa   \dist(0, \Phi(x))
\end{equation}%
for all $x\in \paren{\Ocal\cap \Lambda}\setminus \paren{\Vcal\cap \Lambda}$ 
where $\Phi\equiv T-\Id$.  The statement on annular
regions $\paren{\Ocal\cap \Lambda}\setminus \paren{\Vcal\cap \Lambda}$ can be viewed as 
an assumption about the existence of an {\em error bound} on that region.  
For earlier manifestations 
of this and connections to previous work on error bounds see \cite{LuoTseng93}.
In the present context, this condition will be identified in Section \ref{s:mr} with {\em metric subregularity} of 
$\Phi$, though, of course error bounds and metric subregularity are related.  

The assumptions lead to the conclusion that the iterates approach the set of fixed 
points at some rate that can be bounded below by a linear characterization on the region 
$\paren{\Ocal\cap \Lambda}\setminus \paren{\Vcal\cap \Lambda}$.  This will 
lead to {\em convergence} in Corollary \ref{t:subfirm convergence} where on all such 
annular regions there is some lower linear convergence bound.  

The possibility to 
have $S\subset \Fix T$ and not $S=\Fix T$ allows one to sidestep complications arising from the 
not-so-exotic occurrence of fixed point 
mappings that are almost nonexpansive at some points in $\Fix T$ and not at others 
(see Example \ref{eg:Pie}\ref{eg:Pie2}). 
It would be too restrictive in the statement of the theorem, however, to have $S\subseteq\Fix T$, 
since this does not allow one to tackle 
inconsistent feasibility, studied in depth in Section \ref{s:feas}.  In particular, we have in mind 
the situation where sets $A$ and $B$ do not intersect, but still the alternating projections 
mapping $T_{AP}\equiv P_AP_B$ has nice properties at points in $B$ that, while not fixed points, 
at least locally are nearest to $A$.     The full richness of the structure is used in 
Theorem \ref{t:cp ncvx} were we establish, for the first time, sufficient conditions for local linear convergence
of the method of cyclic projections for nonconvex inconsistent feasibility.

\textit{Proof of Theorem \ref{t:Tconv}}
  % The first part of this proof is left as an exercise.
If $\Ocal\cap\Vcal= \Ocal$ there is nothing to prove.  
Assume, then, that there is some 
$x\in \paren{\Ocal\cap \Lambda}\setminus \paren{\Vcal\cap \Lambda}$.  
Choose any $x^+\in Tx$ and define 
$\xbar^+\in T\xbar$ for $\xbar\in P_Sx$.
%  Note that $T S\subset \Vcal\cap S \subset \Vcal\cap \Lambda$, so $\xbar^+\in \Vcal\cap \Lambda$.  
Inequality \eqref{e:doughnut mreg} implies 
  \begin{equation} % *}
\frac{1-\alpha}{\kappa^2\alpha} \norm{x-\xbar}^{2}\leq
\frac{1-\alpha}{\alpha}\norm{\paren{x -x^+}-\paren{\xbar-\xbar^+}}^2.
 \label{e:lcpg1}
 \end{equation} % *}
Assumption~\ref{t:Tconv a} and Proposition~\ref{t:average char}\ref{t:average char iii}, 
together with \eqref{e:lcpg1} then yield
% \begin{equation}\label{e:lcpg2}
%    \norm{T(u)-\xbar}^2 + \lambda^2\frac{1-\alpha}{\alpha}\norm{u-\xbar}^2\leq   
%  \paren{1+\varepsilon}\norm{u-\xbar}^2.
% \end{equation}
% This immediately yields
\begin{equation}%%\label{e:key}
   \norm{x^+-\xbar^+}^2\leq\paren{1+\varepsilon-\frac{1-\alpha}{\alpha\kappa^2}}\norm{x-\xbar}^2.
\end{equation}%
Note in particular that  $0\leq 1+\varepsilon-\frac{1-\alpha}{\alpha\kappa^2}$.  
Since $\xbar^+\in T(\xbar)\subset \Fix T\cap S$, this proves the first statement. 

If, in addition, $\kappa<\sqrt{\frac{1-\alpha}{\varepsilon\alpha}}$ then $c\equiv \paren{1+\varepsilon-\frac{1-\alpha}{\alpha\kappa^2}}^{1/2}<1$. 
Since clearly $S\supset \Fix T\cap S$, \eqref{e:contraction} yields 
\begin{align*}
\dist(x^{1}, S) \leq   \dist(x^{1},\Fix T\cap S)\leq c\dist(x^{0},S).
\end{align*}
%\begin{eqnarray*}
%\dist(x^{1}, S)&\leq&   \dist(x^{1},\Fix T\cap S)\\
%&\leq& c   \dist(x^{0},S).
%% &\leq&  c\dist(x^{0},\Fix T\cap S).
%\end{eqnarray*}
If  $x^{1}\in \Ocal \setminus \Vcal$ 
%\cap \Lambda$, 
then the first part of this theorem yields 
\begin{align*}
\dist(x^{2}, S) \leq   \dist(x^{2},\Fix T\cap S)
\leq c\dist(x^{1},S) \leq  c^2\dist(x^{0},S).
\end{align*}
%\begin{eqnarray*}
%\dist(x^{2}, S)&\leq&   \dist(x^{2},\Fix T\cap S)\\
%&\leq& c   \dist(x^{1},S)\\
%&\leq&  c^2\dist(x^{0}, S).
%\end{eqnarray*}
Proceeding inductively then, the relation 
$ \dist(x^{j},\Fix T\cap S)\leq c^j\dist(x^{0}, S)$ holds until the first time 
$x^{j-1}\notin\Ocal\setminus \Vcal$.
\hfill $\Box$
% \end{proof}

The inequality \eqref{e:contraction} by itself says nothing about convergence of the iteration $x^{j+1}=Tx^j$, but it 
does clearly indicate what needs to hold in order for the iterates to move closer to a fixed point of $T$.  This is 
stated explicitly in the next corollary. 
\begin{cor}[convergence]\label{t:subfirm convergence}
   Let $\mmap{T}{\Lambda}{\Lambda}$ for $\Lambda\subset\Ebb$ and let $S\subset\reli \Lambda$ be closed and nonempty 
   with $T\xbar\subset\Fix T\cap S$ for all 
   $\xbar\in S$.  
Define $\Ocal_\delta\equiv S+\delta\Ball$ and $\Vcal_\delta\equiv\Fix T\cap S+\delta\Ball$.  
Suppose that for 
$\gamma\in(0,1)$ fixed and for all $\deltabar>0$ small enough,  
there is a triplet 
$(\varepsilon, \delta,\alpha)\in\Rp\times(0,\gamma\deltabar]\times (0,1)$ such that 
\begin{enumerate}[(a)]
   \item\label{t:subfirm convergence a} $T$ is pointwise almost averaged at all $y\in S$ with violation $\varepsilon$ 
   and averaging constant $\alpha$ on $\Ocal_\deltabar\cap \Lambda$, and  
   \item\label{t:subfirm convergence b}  at each $y^+\in Ty$ for all $y\in S$ there exists a 
$ \kappa\in \left[0, \sqrt{\frac{1-\alpha}{\varepsilon\alpha}}\right) $ such that 
% \begin{equation}\label{e:doughnut mreg}
\[ 
\dist(x,S)\leq  \kappa \|\paren{x-x^+}-\paren{y-y^+}\| 
\]
at each $x^+\in Tx$ for all $x\in \paren{\Ocal_\deltabar\cap \Lambda}\setminus\paren{\Vcal_\delta\cap \Lambda}$.
%    \end{equation}
%    $\|x-x^+\|\geq \lambda\dist(x,S)\quad\forall x^+\in Tx, ~\forall x\in S_\deltabar\setminus 
% \paren{S_\deltabar\cap \Vcal_\delta}$.
\end{enumerate}
Then for any $x^0$ close enough to $S$ the iterates $x^{i+1}\in Tx^i$ satisfy 
$\dist(x^i,\Fix T\cap S)\to 0$ as $i\toinf$. 
\end{cor}
\begin{proof}
   Let $\Delta>0$ be such that for all $\deltabar\in(0,\Delta]$ there is a triplet 
$(\varepsilon, \delta,\alpha)\in\Rp\times(0,\gamma\deltabar]\times(0,1)$ for which
\ref{t:subfirm convergence a} and \ref{t:subfirm convergence b} hold.  
Choose any $x^0\in \Ocal_\Delta
%\paren{S+\Delta\Ball}
\cap \Lambda$ and define $\deltabar_0\equiv\dist(x^0,S)$ so that  
\ref{t:subfirm convergence a} and \ref{t:subfirm convergence b} are satisfied  for the 
parameter values $(\varepsilon_0, \delta_0, \alpha_0)\in \Rp\times(0,\gamma\deltabar_0]\times(0,1)$.
Define $x^{(0,j)}\in Tx^{(0,j-1)}$ for $j=1,2,\dots$ with $x^{(0,0)}\equiv x^0$.  
At $j=1$, there are two possible cases: either $x^{(0,1)}\in \Vcal_{\delta_0}\cap \Ocal_{\deltabar_0}$ or 
$x^{(0,1)}\notin \Vcal_{\delta_0}\cap \Ocal_{\deltabar_0}$.
In the former case, 
\[
 \dist\paren{x^{(0,1)}, \Fix T\cap S}\leq \delta_0\leq \gamma\deltabar_0<\deltabar_0,
\]
so for $J_0=1$ it holds that 
\[
  \dist\paren{x^{(0,J_0)}, \Fix T\cap S}\leq \delta_0\leq \gamma\deltabar_0<\deltabar_0.
\]
In the latter case, since $x^{(0,0)}\in \Ocal_{\deltabar_0}\cap \Lambda$, Theorem \ref{t:Tconv} shows that 
\[
\dist\paren{x^{(0,1)}, \Fix T\cap S}\leq c_0\dist\paren{x^{(0,0)}, S}   
\]
for $c_0\equiv \sqrt{1+\varepsilon_0-\frac{1-\alpha_0}{\kappa_0^2\alpha_0}}<1$.  Moreover, clearly
$\dist\paren{x^{(0,1)}, S}\leq \dist\paren{x^{(0,1)}, \Fix T\cap S}$, so in either case  $x^{(0,1)}\in \Ocal_{\deltabar_0}$, 
and the alternative reduces to either $x^{(0,1)}\in \Vcal_{\delta_0}$ or  $x^{(0,1)}\notin \Vcal_{\delta_0}$.
Proceeding by induction for some $j\geq 1$ it holds that 
$x^{(0,\nu)}\in \paren{\Ocal_{\deltabar_0}\cap \Lambda}\setminus\paren{\Vcal_{\delta_0}\cap \Lambda}$ 
for all $\nu=0,1,2\dots j-1$ and $x^{(0,j)}\in \Ocal_{\deltabar_0}\cap \Lambda$ 
with either $x^{(0,j)}\notin \Vcal_{\delta_0}$ 
or $x^{(0,j)}\in \Vcal_{\delta_0}$.  If  $x^{(0,j)}\notin \Vcal_{\delta_0}$, 
then since $x^{(0,j)}\in\Ocal_{\deltabar_0}\cap \Lambda$, by Theorem \ref{t:Tconv},
\[
 \dist\paren{x^{(0,j+1)}, \Fix T\cap S}\leq c_0\dist\paren{x^{(0,j)}, S}.
\]

Iterating this process, there must eventually be a $J_0\in \Nbb$ such  
that 
\begin{equation}\label{conv inter1}
\dist\paren{x^{(0,J_0)}, \Fix T\cap S}\leq \delta_0\leq \gamma\deltabar_0<\deltabar_0.
\end{equation}
To see this, suppose that there is no such $J_0$.  Then 
$x^{(0,j)}\in \paren{\Ocal_{\deltabar_0}\cap \Lambda}\setminus\paren{\Vcal_{\delta_0}\cap \Lambda}$ and 
\[
 \dist\paren{x^{(0,j+1)}, \Fix T\cap S}\leq c_0\dist\paren{x^{(0,j)}, S}\leq c_0^j\dist\paren{x^{(0,0)}, S}
\]
 for all $j\geq 1$.  Since, by assumption $c_0<1$, it holds that $\dist\paren{x^{(0,j)}, \Fix T\cap S}\to 0$ at least linearly with constant $c_0$, 
in contradiction with the assumption that $x^{(0,j)}\notin\Vcal_{\delta_0}$ for all $j$. 

So, with $J_0$ being the first iteration where \eqref{conv inter1} occurs, we update the region 
$\deltabar_1\equiv \dist\paren{x^{(0,J_0)}, S}\leq \dist\paren{x^{(0,J_0)}, \Fix T\cap S}\leq \delta_0 \le \gamma\deltabar_0$, and 
set $x^{1}\equiv x^{(0,J_0)}$ and $x^{(1,0)}\equiv x^{1}$.  By assumption there is a triplet 
$(\varepsilon_1, \delta_1,\alpha_1)\in\Rp\times(0,\gamma\deltabar_1]\times(0,1)$ for which
\ref{t:subfirm convergence a} and \ref{t:subfirm convergence b} hold.  

Proceeding inductively, this  generates the sequence $\paren{x^i}_{i\in\Nbb}$ with 
\[
   x^i\equiv x^{(i-1, J_{i-1})}, \quad \deltabar_i\equiv\dist\paren{x^{i}, S}\leq \dist\paren{x^{i}, \Fix T\cap S}\leq  \gamma^{i}\deltabar_0.
\]
So $\dist\paren{x^i, \Fix T\cap S}\to 0$ as $i\to \infty$.  
As this is just a reindexing of the Picard iteration, this completes the proof.
\end{proof}

An interesting avenue of investigation would be to see to what 
extent the proof mining techniques of \cite{KohLopNic} could be applied to quantify convergence in the present setting. 

\subsection{Metric regularity}\label{s:mr}

The key insight into condition \ref{t:Tconv b} of Theorem \ref{t:Tconv} is the connection 
to {\em metric regularity} of set-valued 
mappings ({\em cf.,} \cite{VA,DontchevRockafellar14}).  
This approach to the study of algorithms has been 
advanced by several authors \cite{Pen02, Ius03, Ara05, AragonGeoffroy07, KlatteKummer09}.
We modify the concept of \textit{metric regularity with functional modulus on a set} suggested in 
\cite[Definition 2.1 (b)] {Ioffe11} and 
\cite[Definition 1 (b)]{Ioffe13} so that the property is relativized to appropriate sets for iterative methods.
Recall that $\mu:[0,\infty) \to [0,\infty)$ is a \textit{gauge function} if $\mu$ is continuous strictly increasing 
with $\mu(0)=0$ and $\lim_{t\to \infty}\mu(t)=\infty$. 

\begin{defn}[metric regularity on a set]\label{d:(str)metric (sub)reg}
$~$ Let $\mmap{\Phi}{\Ebb}{\Ybb}$, $U\subset \Ebb$, $V\subset \Ybb$.
The mapping $\Phi$ is called \emph{metrically regular with gauge $\mu$ on $U\times V$ relative to $\Lambda\subset\Ebb$} if
\begin{equation}\label{e:metricregularity}
\dist\paren{x, \Phi^{-1}(y)\cap \Lambda}\leq \mu\paren{\dist\paren{y, \Phi(x)}}
\end{equation}
holds for all $x\in U\cap \Lambda$ and $y\in V$ with $0<\mu\paren{\dist\paren{y,\Phi(x)}}$.
When the set $V$ consists of a single point, $V=\{\ybar\}$, then $\Phi$ is said to be 
\emph{metrically subregular for $\ybar$ on $U$ with gauge $\mu$ relative to $\Lambda\subset\Ebb$}.

When $\mu$ is a linear function (that is, $\mu(t)=\kappa t,\, \forall t\in [0,\infty)$), one says ``with constant $\kappa$'' instead of 
``with gauge $\mu(t) =\kappa t$''. When $\Lambda=\Ebb$, the quantifier ``relative to'' is dropped.  
When $\mu$ is linear, the smallest constant  
$\kappa$ for which \eqref{e:metricregularity} holds is called the {\em modulus} of metric regularity.
\end{defn}

The conventional concept of \textit{metric regularity} \cite{Aze06, DontchevRockafellar14, VA} 
(and \textit{metric regularity of order $\omega$, respectively} \cite{KrugerNguyen15}) at a point 
$\xbar\in \Ebb$ for $\ybar\in \Phi(\xbar)$ corresponds to the setting in Definition \ref{d:(str)metric (sub)reg} 
where $\Lambda=\Ebb$, $U$ and $V$ are {\em neighborhoods} of $\xbar$ and $\ybar$, respectively, and 
the gauge function $\mu(t)=\kappa t$ ($\mu(t)=\kappa t^{\omega}$ for metric regularity of order $\omega<1$) 
for all $t\in [0,\infty)$, with $\kappa>0$.

Relaxing the requirements on the sets $U$ and $V$ from neighborhoods to the more ambiguous sets
in Definition \ref{d:(str)metric (sub)reg} allows the same definition and terminology to 
unambiguously cover well-known 
relaxations of metric regularity such as \textit{metric subregularity} ($U$ is a neighborhood of 
$\xbar$ and $V=\{\ybar\}$, \cite{DontchevRockafellar14}) and \textit{metric hemi/semiregularity} 
($U=\{\xbar\}$ and $V$ is a neighborhood of $\ybar$ \cite[Definition 1.47]{Mord06}).  For our purposes, we will use 
the flexibility of choosing $U$ and $V$ in Definition \ref{d:(str)metric (sub)reg} to {\em exclude} the 
reference point $\xbar$ and to {\em isolate} the image point $\ybar$.  
This is reminiscent of the Kurdyka-{\L}ojasiewicz (KL) property \cite{BolDan2010} for functions 
which requires that the subdifferential posses a sharpness property near (but not at) critical points of the function.
However, since the restriction of $V$ to a point features prominently in our development, we retain the terminology 
{\em metric subregularity} to ease the technicality of the presentation.  The reader is cautioned, however, that 
our usage of metric subregularity does not precisely correspond to the usual definition 
(see \cite{DontchevRockafellar14}) since we do not require the domain $U$ to be a neighborhood. 

\begin{thm}[(sub)linear convergence with metric regularity]\label{t:metric subreg convergence}
   Let $\mmap{T}{\Lambda}{\Lambda}$ for $\Lambda\subset\Ebb$, $\Phi\equiv T-\Id$ and let 
   $S\subset \reli \Lambda$ be closed and nonempty with $TS\subset\Fix T\cap S$.
Denote $\paren{S+\delta\Ball}\cap \Lambda$ by $S_\delta$ for a nonnegative real $\delta$.
Suppose that, for all $\deltabar>0$ small enough, there are  $\gamma\in (0,1)$, a nonnegative 
sequence of scalars $\paren{\varepsilon_i}_{i\in \Nbb}$  and a sequence 
   of positive constants $\alpha_i$ bounded above by $\overline{\alpha}<1$, 
   such that, for each $i\in \Nbb$,  
\begin{enumerate}[(a)]
   \item\label{t:metric subreg convergence a} $T$ is pointwise almost averaged at all $y\in S$ with 
      averaging constant $\alpha_i$ and violation 
      $\varepsilon_i$ on  $S_{\gamma^i\deltabar}$, and
   \item\label{t:metric subreg convergence b} for
\[
R_i\equiv S_{\gamma^i\deltabar}\setminus\paren{\Fix T\cap S + \gamma^{i+1}\deltabar\Ball},
\]   
   \begin{enumerate}[(i)]
      \item\label{t:metric subreg convergence bi} 
      $\dist\paren{x, S} \le \dist\paren{x, \Phi^{-1}(\ybar)\cap \Lambda}$ for all
      $x\in R_i$ and $\ybar\in \Phi(P_S(x))\setminus \Phi(x)$,
      \item\label{t:metric subreg convergence bii} $\Phi$ is metrically regular with gauge $\mu_i$ 
      relative to $\Lambda$ on $R_i\,\times\, \Phi(P_S(R_i))$, where $\mu_i$ satisfies
	\begin{equation}\label{e:kappa-epsilon}
\sup_{x\in R_i,\ybar \in \Phi(P_S(R_i)), \ybar\notin \Phi(x)} 
\frac{\mu_i\paren{\dist\paren{\ybar,\Phi(x)}}}{\dist\paren{\ybar,\Phi(x)}}
\le\kappa_i <\sqrt{\frac{1-\alpha_i}{\varepsilon_i\alpha_i}}.
	\end{equation}
	\end{enumerate}
\end{enumerate}
Then, for any $x^0\in \Lambda$ close enough to $S$, 
the iterates $x^{j+1}\in Tx^j$ satisfy $\dist\paren{x^j,\Fix T\cap S}\to 0$ and
\begin{equation}\label{e:metric subreg linear conv}
\dist\paren{x^{j+1}, \Fix T\cap S}
% \leq \dist\paren{x^{j+1}, \Fix T\cap S}
\leq c_i\dist\paren{x^j, S}\quad\forall ~x^j\in R_i,
\end{equation}
where $c_i\equiv \sqrt{1+\varepsilon_i-\paren{\tfrac{1-\alpha_i}{\kappa_i^2\alpha_i}}}<1$.

In particular, if $\varepsilon_i$ is bounded above by $\varepsilonbar$ and 
$\kappa_i\leq \kappabar<\sqrt{\frac{1-\alphabar}{\alphabar\,\varepsilonbar}}$ for all $i$ large enough, 
then convergence is eventually at least linear with rate at most 
$\cbar\equiv\sqrt{1+\varepsilonbar-\paren{\tfrac{1-\alphabar}{\kappabar^2\alphabar}}}<1$.
\end{thm}

The first 
inequality in \eqref{e:kappa-epsilon} is a 
condition on the gauge function $\mu_i$ and would not be needed if the statement were limited to linearly 
metrically regular mappings.  Essentially, it says that the gauge function characterizing metric regularity of 
$\Phi$ can be bounded above by a linear function.    
The second inequality states that the constant of metric regularity $\kappa_i$ 
is small enough relative to the violation of the averaging property $\varepsilon_i$ to guarantee 
a linear progression of the iterates through the region $R_i$. 
\medskip

\noindent {\em Proof of Theorem \ref{t:metric subreg convergence}.}
To begin, note that by assumption \ref{t:metric subreg convergence b}, 
for any $x\in R_i$, $\xbar\in P_S(x)$, and $\ybar\in \Phi(\xbar)$ with $\ybar\notin \Phi(x)$,
\begin{equation}%%\label{e:last last kappa_k ineq}
\dist(x,S)\leq \dist(x,\Phi^{-1}(\ybar)\cap \Lambda) \le \mu_i\paren{\dist\paren{\ybar, \Phi(x)}}
\le \kappa_i\dist\paren{\ybar, \Phi(x)}.
\end{equation}%
Let $\ybar=\xbar^+-\xbar$ for $\xbar^+\in T\xbar$. The above statement yields 
\begin{equation}%
      \dist(x,S)\leq \kappa_i\norm{\paren{x^+-x}-\paren{\xbar^+-\xbar}} \quad \forall ~x\in R_i, \forall ~x^+\in Tx, \forall~ 
      \xbar\in P_S(x), \forall~ \xbar^+\in T\xbar.
\end{equation}%
The convergence of the sequence $\dist\paren{x^j, \Fix T\cap S}\to 0$ then follows from Corollary \ref{t:subfirm convergence} 
with the sequence of triplets $\paren{\varepsilon_i, \gamma^{i+1}\deltabar, \alpha_i}_{i\in\Nbb}$.  
By Theorem \ref{t:Tconv} the rate of convergence on $R_i$
% $S_{\gamma^i\deltabar}\setminus \paren{S_{\gamma^i\deltabar}\cap \paren{\Fix T+\gamma^{i+1}\deltabar\Ball}}$ 
is characterized by
\begin{equation}%%\label{e:contraction2}
   \dist\paren{x^+,\Fix T\cap S}\leq\sqrt{1+\varepsilon_i - \frac{1-\alpha_i}{\kappa_i^2\alpha_i}} \dist(x,S) \quad\forall~ x^+\in Tx,
\end{equation}%
whence \eqref{e:metric subreg linear conv} holds with constant $c_i<1$ given by \eqref{e:kappa-epsilon}.  

The final claim of the theorem follows immediately. 
\endproof

When $S=\Fix T\cap \Lambda$ in Theorem \ref{t:metric subreg convergence}, the condition 
\ref{t:metric subreg convergence b} \ref{t:metric subreg convergence bi} can be dropped from the assumptions, 
as the next corollary shows. 

\begin{cor}\label{t:str metric subreg convergence} Let $\mmap{T}{\Lambda}{\Lambda}$ for 
$\Lambda\subset\Ebb$ with $\Fix T$ nonempty and closed, $\Phi\equiv T-\Id$.
Denote $\paren{\Fix T+\delta\Ball}\cap \Lambda$ by $S_\delta$ for a nonnegative real $\delta$.
Suppose that, for all $\deltabar>0$ small enough, there are  $\gamma\in (0,1)$, a nonnegative 
sequence of scalars $\paren{\varepsilon_i}_{i\in \Nbb}$  and a sequence 
   of positive constants $\alpha_i$ bounded above by $\overline{\alpha}<1$, 
   such that, for each $i\in \Nbb$,  
\begin{enumerate}[(a)]
   \item $T$ is pointwise almost averaged at all $y\in \Fix T\cap \Lambda$ with 
      averaging constant $\alpha_i$ and violation 
      $\varepsilon_i$ on  $S_{\gamma^i\deltabar}$, and
   \item for
\[
R_i\equiv S_{\gamma^i\deltabar}\setminus\paren{\Fix T + \gamma^{i+1}\deltabar\Ball},
\]   
$\Phi$ is metrically subregular for $0$  on $R_i$ (metrically regular on $R_i\times\{0\}$) with gauge $\mu_i$ relative to $\Lambda$, 
where $\mu_i$ satisfies
	\begin{equation}\label{e:kappa-epsilon'}
\sup_{x\in R_i} 
\frac{\mu_i\paren{\dist\paren{0,\Phi(x)}}}{\dist\paren{0,\Phi(x)}}
\le\kappa_i <\sqrt{\frac{1-\alpha_i}{\varepsilon_i\alpha_i}}.
	\end{equation}
\end{enumerate}
Then, for any $x^0\in \Lambda$ close enough to $\Fix T\cap \Lambda$, 
the iterates $x^{j+1}\in Tx^j$ satisfy $\dist\paren{x^j,\Fix T\cap \Lambda}\to 0$ and
\begin{equation}%
\dist\paren{x^{j+1}, \Fix T\cap \Lambda}
\leq c_i\dist\paren{x^j, \Fix T\cap \Lambda}\quad\forall ~x^j\in R_i,
\end{equation}%
where $c_i\equiv \sqrt{1+\varepsilon_i-\paren{\tfrac{1-\alpha_i}{\kappa_i^2\alpha_i}}}<1$.

In particular, if $\varepsilon_i$ is bounded above by $\varepsilonbar$ and 
$\kappa_i\leq \kappabar<\sqrt{\frac{1-\alphabar}{\alphabar\,\varepsilonbar}}$ for all $i$ large enough, 
then convergence is eventually at least linear with rate at most 
$\cbar\equiv\sqrt{1+\varepsilonbar-\paren{\tfrac{1-\alphabar}{\kappabar^2\alphabar}}}<1$.
\end{cor}

\begin{proof}
To deduce Corollary \ref{t:str metric subreg convergence} from Theorem \ref{t:metric subreg convergence}, 
it suffices to check that when $S=\Fix T\cap \Lambda$, condition \eqref{e:kappa-epsilon} becomes \eqref{e:kappa-epsilon'}, 
and condition \ref{t:metric subreg convergence bi} is always satisfied.
This follows immediately from the fact that $\Phi(P_{\Fix T\cap \Lambda}(\Ebb))=\{0\}$ and $\Phi^{-1}(0)=\Fix T$.
\end{proof}

The following example explains why gauge metric regularity on a set (Definition \ref{d:(str)metric (sub)reg}) fits well 
in the framework of Theorem \ref{t:metric subreg convergence}, whereas the conventional metric (sub)regularity does not.
\begin{eg}[a line tangent to a circle]\label{eg:4.4} In $\mathbb{R}^2$, consider the two sets
\begin{align*}
A &:= \{(u,-1)\in \Rbb^2: u\in \Rbb\},\\
B &:= \{(u,v)\in \Rbb^2: u^2+v^2=1\},
\end{align*}
and the point $\xbar = (0,-1)$.
It is well known that the alternating projection algorithm $T:=P_AP_B$ does not converge linearly to
$\xbar$ unless with the starting points on $\{(0,v)\in \Rbb^2: v\in \Rbb\}$ (in this special case, the method reaches 
$\xbar$ in one step).
Note that $T$ behaves the same if $B$ is replaced by the closed unit ball (the case of two closed convex sets). 
In particular, $T$ is averaged with constant $\alpha=2/3$ by Proposition \ref{t:av-comp av}\ref{t:av-comp av iii}.
Hence, the absence of linear convergence of $T$ here can be explained as the lack of regularity of the fixed point set 
$A\cap B=\{\xbar\}$. In fact, the mapping $\Phi:=T-\Id$ is not (linearly) metrically subregular at $\xbar$ for $0$ on any set 
$\Ball_{\delta}(\xbar)$, for any $\delta>0$.
However, $T$ does converge sublinearly to $\xbar$.  This can be characterized in two different ways.
\begin{itemize}
\item Using %Theorem \ref{t:metric subreg convergence}
Corollary \ref{t:str metric subreg convergence}, we characterize sublinear convergence in this example as linear 
convergence 
on annular sets. To proceed, we set
\[
R_i:= \Ball_{2^{-i}}(\xbar)\setminus \Ball_{2^{-(i+1)}}(\xbar),\quad (i=0,1,\ldots).
\]
This corresponds to setting $\deltabar=1$ and $\gamma=1/2$ in Corollary \ref{t:str metric subreg convergence}.
The task that remains is to estimate the constant of metric subregularity, $\kappa_i$, of $\Phi$ on each $R_i$. Indeed, we have
\begin{align*}
\inf_{x\in R_i\cap A}\frac{\|x-Tx\|}{\|x-\xbar\|} = \frac{\|x^*-Tx^*\|}{\|x^*-\xbar\|}
= 1-\frac{1}{\sqrt{2^{-2(i+1)}+1}}:=\kappa_i>0, \quad (i=0,1,\ldots),
\end{align*}
where $x^* = (2^{-(i+1)},-1)$.

Hence, on each ring $R_i$, $T$ converges linearly to a point in $\Ball_{2^{-(i+1)}}(\xbar)$ with rate
$c_i$ not worse than $\sqrt{1-1/(2\kappa_i^2)}<1$ by Corollary \ref{t:str metric subreg convergence}.

\item The discussion above uses the linear gauge functions $\mu_i(t):=\frac{t}{\kappa_i}$ on annular 
regions, and hence a piecewise linear gauge function for the characterization of metric subregularity.
Alternatively, we can construct a smooth gauge function $\mu$ that works on neighborhoods of the 
fixed point.  For analyzing convergence of $P_AP_B$, we must have $\Phi$  
metrically subregular at $0$ with gauge $\mu$ on $\Rbb^2$ \textit{relative to $A$}.
But we have
\begin{align}\label{f_est}
\dist\paren{0,\Phi(x)} = \|x-x^+\| = f\paren{\|x-\xbar\|} = f\paren{\dist\paren{x,\Phi^{-1}(0)}},\quad  \forall x\in A,
\end{align}
where $f:[0,\infty)\to [0,\infty)$ is given by $f(t):= t\paren{1-1/\sqrt{t^2+1}}$.
The function $f$ is continuous strictly increasing and satisfies $f(0)=0$ and $\lim_{t\to \infty}f(t) =\infty$. 
Hence, $f$ is a gauge function.
%  and so is the function $\mu:[0,\infty) \to [0,\infty)$ given by $\mu(t):= f^{-1}(t)$.
% This gauge function $\mu$ is the one we desired to construct.

We can now characterize sublinear convergence of $P_AP_B$ explicitly without resorting to annular sets.
Note first that since $f(t)<t$ for all $t\in (0,\infty)$ the function $g:[0,\infty)\to [0,\infty)$ given by
\[
g(t):= \sqrt{t^2-\frac{1}{2}(f(t))^2}
\]
is a gauge function and satisfies $g(t)< t$ for all $t\in (0,\infty)$.
Note next that $T\equiv P_AP_B$ is (for all points in $A$) averaged with constant $2/3$ together with 
\eqref{f_est}, we get for any $x\in A$
\begin{align*}
\norm{x^+-\xbar}^2 &\le  \norm{x-\xbar}^2 - (1/2)\norm{x-x^+}^2\\
&=  \norm{x-\xbar}^2 - (1/2)\paren{f\paren{\|x-\xbar\|}}^2.
\end{align*} 
This implies
\begin{align*}
\dist(x^+,S) &= \norm{x^+-\xbar} \le \sqrt{\norm{x-\xbar}^2 - (1/2)\paren{f\paren{\|x-\xbar\|}}^2}\\
&= g\paren{\norm{x-\xbar}}
= g\paren{\dist(x,S)},\quad \forall x\in A.
\end{align*} 
\end{itemize}
\hfill$\triangle$
\end{eg}
\bigskip
\begin{remark}[global (sub)linear convergence of pointwise averaged mappings]\label{r:metric subregu convergence}
As Example \ref{eg:4.4} illustrates, 
Theorem~\ref{t:metric subreg convergence} is not an asymptotic result and does not gainsay the possibility that 
the required properties hold with neighborhood $U=\Ebb$, which would then lead to a global quantification of convergence.   
First order methods for convex problems lead generically to {\em globally} averaged fixed point mappings $T$. 
Convergence for convex problems can be determined from the averaging property of $T$ and existence of fixed points. 
Hence in order to {\em quantify} convergence the only thing to be determined is the gauge of metric regularity at 
the fixed points of $T$. In this context, see\cite{borwein2015convergence}.
Example \ref{eg:4.4} illustrates how this can be done.  This instance will be revisited in Example \ref{eg:circles}. 
\end{remark}
%\begin{remark}
%Strong met. reg. im plies assumption b(i).....
%\end{remark}

The following proposition, taken from \cite{DontchevRockafellar14},  characterizes 
	metric subregularity in terms of the graphical derivative defined by \eqref{e:gder}.
	\begin{propn}[characterization of metric regularity {\cite[Theorems~4B.1 and 4C.2]{DontchevRockafellar14}}]\label{t:strmsr diff}
		Let $\mmap{T}{\Rbb^{n}}{\Rbb^{n}}$ have locally closed graph at 
		$(\xbar, \ybar)\in\gph T$, $\Phi\equiv T-\Id$, and $\zbar:=\ybar-\xbar$.
		Then $\Phi$ is metrically subregular for $0$ on $U$ (metrically regular on $U\times \{\zbar\}$) with constant 
		$\kappa$ for $U$ some neighborhood of 
		$\xbar$ satisfying $U\cap \Phi^{-1}(\zbar)=\{\xbar\}$ if and only if the graphical derivative satisfies 
		\begin{equation}\label{e:gd strmr}
		   D\Phi(\xbar|\zbar)^{-1}(0)=\{0\}.
		\end{equation}
		If, in addition,  $T$ is single-valued and continuously differentiable on 
		$U$, then the two conditions hold if and only if $\nabla \Phi$ has rank $n$ at $\xbar $ with 
 		$\norm{\ecklam{\ecklam{\nabla
 		%^{\ast} 
 		\Phi\left(x\right)}^\intercal}^{-1}} \leq \kappa$
 		for all $x$ on $U$.
	\end{propn}
While the characterization \eqref{e:gd strmr} appears daunting, the property  
comes almost for free for polyhedral mappings.  

\begin{propn}[polyhedrality implies metric subregularity]\label{t:polyhedrality-strong msr}
Let $\Lambda\subset \Ebb$ be an affine subspace and $\mmap{T}{\Lambda}{\Lambda}$.  If $T$ is polyhedral and 
$\Fix T\cap \Lambda$ is an isolated point, $\{\xbar\}$, 
then $\Phi\equiv T-\Id$ is metrically subregular for $0$ on $U$ (metrically regular on $U\times \{0\}$) relative to 
$\Lambda$ with some constant $\kappa$ for some neighborhood $U$ of $\xbar$. In particular, $U\cap \Phi^{-1}(0)=\{\xbar\}$.     
\end{propn}
\begin{proof}
If $T$ is polyhedral, so is 
$\Phi^{-1}\equiv (T-\Id)^{-1}$. The statement now follows from 
\cite[Propositions  3I.1 and 3I.2]{DontchevRockafellar14}, since 
$\Phi^{-1}$ is polyhedral and $\xbar$ is an isolated point of $\Phi^{-1}(0)\cap\Lambda$.
\end{proof}

\begin{propn}[local linear convergence: polyhedral fixed point iterations]\label{t:polyhedral convergence}
 Let $\Lambda\subset \Ebb$ be an affine subspace and $\mmap{T}{\Lambda}{\Lambda}$ be pointwise almost averaged 
at $\{\xbar\}= \Fix T\cap \Lambda$ on $\Lambda$ with violation constant $\varepsilon$ and averaging constant $\alpha$.
If  $T$ is polyhedral, then there
is a neighborhood $U$ of $\xbar$ such that
\[
\norm{x^+-\xbar}\leq c\norm{x-\xbar}\quad\forall  x\in U\cap \Lambda,~ x^+\in Tx,   
\]
where $c=\sqrt{1+\varepsilon - \frac{1-\alpha}{\kappa^2\alpha}}$ and $\kappa$ is the modulus of metric 
subregularity of $\Phi\equiv T-\Id$ for $0$ on $U$ relative to $\Lambda$. If, in addition 
$\kappa< \sqrt{(1-\alpha)/(\alpha\varepsilon)}$, then the fixed point iteration 
$x^{j+1}\in Tx^j$ converges linearly to $\xbar$ with rate 
$c<1$ for all $x^0\in U\cap \Lambda$. 
\end{propn}
\begin{proof}
The result follows immediately from Proposition \ref{t:polyhedrality-strong msr} and Corollary \ref{t:str metric subreg convergence}.
\end{proof}

% relate to other notions:  demiclosedness of Fix T, asymptotic regularity ...
\section{Applications}\label{s:apps}
The idea of the previous section is simple.  Formulated as Picard iterations of a fixed point mapping $T$, 
in order to establish the quantitative convergence of an algorithm, one must establish two properties of this mapping:  
first that $T$ is almost averaged and second that 
it is metrically subregular at fixed points relative to an appropriate subset.  
  This section serves as a tutorial for how to do this for fundamental first-order algorithms.  
Each of the problems 
studied below represents a distinct region on the map of numerical analysis, each with its own dialect.  
Part of our goal is to show that the phenomena that these different dialects  
describe sort into one of the two more general properties of fixed point mappings established above.  
While the technicalities can become quite dense, particularly for feasibility, the two principles above
offer a reliable guide through the details.    

\subsection{Feasibility}\label{s:feas}

The feasibility problem is to find $\xbar\in\cap_{j=1}^m\Omega_j$.  If the intersection is empty, 
the problem is called {\em inconsistent}, but 
a meaningful solution still can be found in the sense of best approximation in the case of just 
two sets, or in some other appropriate sense when there are three or more sets.  The most
prevalent algorithms for solving these problems are built on projectors onto the individual 
sets (indeed, we are aware of no other approach to the problem). The regularity 
of the fixed point mapping $T$ that encapsulates a particular algorithm 
(in particular, pointwise almost averaging and coercivity at the fixed point set)  stems 
from the regularity of the underlying projectors and the way the projectors are put together 
to construct $T$.  Our first task is to show in what way the regularity of the underlying projectors is inherited from 
the regularity of the sets $\Omega_j$.

\subsubsection{Elemental set regularity}
The following definition of what we call {\em elemental regularity} was first presented in 
\cite[Definition 5]{KruLukNgu16}.  This places under one schema the many different 
kinds of set regularity appearing in \cite{LewisLukeMalick09, BauLukePhanWang13a, 
BauLukePhanWang13b, HesseLuke13, BauLukePhanWang14, NolRon16}.  
\begin{definition}[elemental regularity of sets]\label{d:set regularity}
Let $\Omega\subset\Euclid$ be nonempty and let $\paren{\ybar,\vbar}\in\gph\paren{\ncone{\Omega}}$.
\begin{enumerate}[(i)]
   \item\label{d:geom subreg} $\Omega$ is 
   {\em elementally subregular of order $\sigma$ relative to $\Lambda$ 
at $\xbar$ for $\paren{\ybar,\vbar}$ with constant $\varepsilon$}
if there exists a neighborhood $U$ of $\xbar$
% (depending on $\varepsilon$) 
such that 
\begin{equation}\label{e:geom subreg}
\ip{\vbar - \paren{x-x^+}}{x^+ - \ybar}\leq
\varepsilon\norm{\vbar - \paren{x-x^+}}^{1+\sigma}\norm{x^+ - \ybar},\quad \forall x\in \Lambda \cap U,~ x^+\in P_{\Omega}(x).
\end{equation}

   \item\label{d:uni geom subreg} The set $\Omega$ 
is said to be {\em uniformly} elementally subregular 
of order $\sigma$ relative to $\Lambda$ at $\xbar$ for 
$\paren{\ybar,\vbar}$ if for any $\varepsilon>0$ 
there is a neighborhood $U$ (depending on $\varepsilon$) of $\xbar$  such that \eqref{e:geom subreg} holds.

   \item\label{d:geom reg}
The set $\Omega$ is said to be {\em elementally regular of order $\sigma$ 
at $\xbar$ for $\paren{\ybar,\vbar}$ with constant $\varepsilon$} if it is elementally 
subregular of order $\sigma$ relative to $\Lambda=\Omega$ at $\xbar$ for all $\paren{\ybar,v}$ with 
constant $\varepsilon$ where 
$v\in\ncone{\Omega}(\ybar)\cap V$ for some neighborhood $V$ of $\vbar$.

 \item
The set $\Omega$ is said to be {\em uniformly} elementally regular of order $\sigma$ 
at $\xbar$ for $\paren{\ybar,\vbar}$ if it is uniformly elementally subregular of order 
$\sigma$ relative to $\Lambda=\Omega$ at $\xbar$ for all $\paren{\ybar,v}$ where 
$v\in\ncone{\Omega}(\ybar)\cap V$ for some neighborhood $V$ of $\vbar$. 
\end{enumerate}
If $\Lambda=\{\xbar\}$ in \ref{d:geom subreg} or \ref{d:uni geom subreg}, then the respective qualifier
``relative to'' is dropped.
If $\sigma=0$, then the respective qualifier ``of order'' 
is dropped in the description of the properties.
The {\em modulus of elemental (sub)regularity} is the infimum over all
$\varepsilon$ for which \eqref{e:geom subreg} holds.  
\end{definition}

In all properties in Definition \ref{d:set regularity},
$\xbar$ need not be in $\Lambda$ and $\ybar$ need not be in either $U$ or $\Lambda$.
In case of order $\sigma=0$, the properties are trivial for any constant $\varepsilon\ge 1$.
When saying a set is not elementally (sub)regular but without specifying a constant, 
it is meant for any constant $\varepsilon<1$.

\begin{eg}\label{eg:regularity examples}$~$
\begin{enumerate}[(a)]
 \item\label{eg:cross}  (cross)  
Recall the set in Example \ref{eg:cross1},
 \[
    A=\mathbb{R}\times\{0\}\cup \{0\}\times\mathbb{R}.
 \]
This example is of particular interest for the study of {\em sparsity constrained optimization.}
$ A$ is elementally regular at any $\xbar\neq 0$, say $\|\xbar\|>\delta>0$, for all
$(a,v)\in \gph{\ncone{A}}$ where $a\in \Ball_{\delta}(\xbar)$
with constant $\varepsilon=0$ and neighborhood $\Ball_\delta(\xbar)$.
The set $ A$ is not elementally regular at $\xbar=0$ for any $(0,v)\in \gph{\ncone{A}}$
since $\ncone{A}(0)= A$.
However, $ A$ is elementally subregular at $\xbar=0$ for all $(a,v)\in\gph{\ncone{A}}$
with constant $\varepsilon=0$ and neighborhood $\Ebb$ since all vectors $a\in A$ are orthogonal to $\ncone{A}(a)$. \\

\item\label{eg:circle} (circle)
The humble circle is central to the phase retrieval problem, 
\[
  A=\set{ (x_1,x_2)\in\Rtw}{ x_1^2+x_2^2=1}.
 \]
The set $ A$ is uniformly elementally regular at any $\xbar\in  A$ for all $(\xbar,v)\in \gph{\ncone{A}}$.
Indeed, note first that for any $\xbar \in  A$, $\ncone{A}(\xbar)$ consists of the line passing through the origin and $\xbar$.
Now, for any $\varepsilon \in (0,1)$, we choose $\delta = \varepsilon$.
Then for any $x\in  A\cap \Ball_{\delta}(\xbar)$, it holds $\cos\angle(-\xbar,x-\xbar)\le \delta\le \varepsilon$.
Hence, for all $x\in  A\cap \Ball_{\delta}(\xbar)$ and $v\in \ncone{A}(\xbar)$,
 \[
   \ip{v}{x-\xbar} = \cos\angle(v,x-\xbar)\|v\|\|x-\xbar\| \leq \cos\angle(-\xbar,x-\xbar)\|v\|\|x-\xbar\| \le \varepsilon\|v\|\|x-\xbar\|.
 \]

\item\label{eg:Pie1r} (Packman eating a piece of pizza)
Consider again the sets 
 \begin{eqnarray*} % *}
  A&=&\set{ (x_1,x_2)\in\Rtw}{ x_1^2+ x_2^2\leq 1, ~-1/2x_1\leq x_2\leq x_1, x_1\geq 0}\subset\Rtw\\
  B&=&\set{ (x_1,x_2)\in\Rtw}{ x_1^2+ x_2^2\leq 1, ~x_1\leq |x_2|}\subset\Rtw.\\
\xbar&=&(0,0)
 \end{eqnarray*} % *}
from Example \ref{eg:Pie}\ref{eg:Pie2}.
The set $ B$ is elementally subregular relative to $A$ at $\xbar = 0$ for all
$(b,v)\in \gph{\paren{\ncone{B}\cap A}}$ with constant $\varepsilon=0$ and neighborhood
$\Ebb$ since for all $a\in A$, $a_B\in P_{B}(a)$ and $v\in \ncone{B}(b)\cap A$, there holds
\[
 \ip{v-(a-a_B)}{a_B-b} = \ip{v}{a_B-b}-\ip{a-a_B}{a_B-b} = 0.
\]
The set $B$, however, is not elementally regular at $\xbar=0$ for any $(0,v)\in\gph{\ncone{B}}$
because by choosing $x=tv \in  B$ (where $0\neq v\in  B\cap \ncone{B}(0)$, $t\downarrow 0$),
we get $\ip{v}{x}=\|v\|\|x\|>0$.
\end{enumerate}
\end{eg}

To see how the language of elemental regularity unifies the existing terminology, we list the 
following equivalences first established in \cite[Proposition 4]{KruLukNgu16}. 
\begin{propn}[equivalences of elemental (sub)regularity]\label{t:set regularity equivalences}
Let $A$, $A'$ and $B$ be closed nonempty subsets of $\Ebb$.
\begin{enumerate}[(i)]
\item\label{t:Hoelder char}
Let $ A\cap B\neq \emptyset$ and
suppose that there is a neighborhood $W$ of
$\xbar\in A\cap B$ and a constant $\varepsilon>0$ such that
for each
\begin{equation}\label{e:Hoelder normals}
(a,v)\in V\equiv \set{(b_A,u)\in \gph{\pncone{A}}}
{u=b-b_A,
       \begin{array}{c}\mbox{ for }b \in  B\cap W\\
		      \mbox{ and }b_A \in P_A(b)\cap W
       \end{array}},
\end{equation}
it holds that
\begin{equation}\label{e:xbar neighborhood}
\xbar\in\intr U(a,v) \mbox{ where } U(a,v)\equiv\Ball_{(1+\varepsilon^2)\|v\|}(a+v).
\end{equation}
Then, $A$ is $\sigma$-H\"older regular relative to $B$ at $\xbar$
in the sense of \cite[Definition 2]{NolRon16}  
%-- Definition \ref{d:set regularities} \eqref{d:Hoelder} --
with constant $c=\varepsilon^2$ and neighborhood $W$ of $\xbar$ 
if and only if $ A$ is elementally subregular of order $\sigma$ relative to
$ A\cap P^{-1}_B\paren{a+v}$ at $\xbar$
for each $(a,v)\in V$ with constant $\varepsilon=\sqrt{c}$ and the respective neighborhood
$U(a,v)$.

\item\label{t:L-eps-del-subreg char}
Let $B\subset A$.  The set $A$ is $(\varepsilon,\delta)$-subregular relative to $B$ 
at $\xbar\in  A$ in the sense of \cite[Definition 2.9]{HesseLuke13} if and only if $ A$ is
elementally subregular relative to $B$ at $\xbar$ for all
$(a,v)\in \gph{\pncone{ A}}$ where $a\in \Ball_{\delta}(\xbar)$
with constant $\varepsilon$ and neighborhood $\Ball_{\delta}(\xbar)$.
Consequently, $(\varepsilon,\delta)$-subregularity implies 0-H\"older regularity.

\item\label{t:L-eps-del-regular equivalence}
If the set $ A$ is $( \Ebb,\varepsilon,\delta)$-regular at $\xbar$ in the sense of 
\cite[Definition 8.1]{BauLukePhanWang13a}, then $ A$ is elementally regular at $\xbar$
for all $(\xbar,v)$ with constant $\varepsilon$, where $0\neq v\in \pncone{ A}(\xbar)$.
Consequently, $( \Ebb,\varepsilon,\delta)$-regularity implies $(\varepsilon,\delta)$-subregularity. 

\item\label{t:Clarke char}
The set $ A$ is {\em Clarke regular} at $\xbar\in  A$ \cite[Definition 6.4]{VA} if and only if $ A$ is
uniformly elementally regular at $\xbar$ for all $(\xbar,v)$ with $v\in\ncone{A}(\xbar)$.
Consequently, Clarke regularity implies $(\varepsilon,\delta)$-regularity.

\item\label{t:super-regular equivalence}
The set $ A$ is super-regular at $\xbar\in  A$ \cite[Definition 4.3]{LewisLukeMalick09}
% Definition \ref{d:set regularities}\eqref{d:super-reg}
if and only if for any $\varepsilon>0$, there is a $\delta>0$
such that $A$ is elementally regular at $\xbar$ for all
$(a,v)\in \gph{\ncone{A}}$ where
$a\in \Ball_{\delta}(\xbar)$
with constant $\varepsilon$ and neighborhood $\Ball_{\delta}(\xbar)$.
Consequently, super-regularity implies Clarke regularity.

\item\label{t:prox-regular equivalence} If $ A$ is prox-regular at $\xbar$ \cite[Definition 1.1]{PolRockThib00}, then
there exist positive constants $\overline{\varepsilon}$ and $\deltabar$
such that, for any $\varepsilon>0$ and $\delta:=\frac{\varepsilon\deltabar}{\overline{\varepsilon}}$ defined correspondingly,
$ A$ is elementally regular at $\xbar$ for all
$(a,v)\in\gph{\ncone{A}}$ where $a\in \Ball_\delta(\xbar)$ with constant $\varepsilon$ and neighborhood $\Ball_\delta(\xbar)$.
Consequently, prox-regularity implies super-regularity.

\item\label{t: convex set regularity} If $ A$ is convex then it is
% RL 2/5: nonempty closed used in the proof.
% Th 2/9: We assume nonempty and closed at the beginning for all items.
elementally regular at all $x \in A$ for all
$(a,v)\in \gph{\ncone{A}}$ with constant $\varepsilon=0$
and the neighborhood $\Ebb$ for both $x$ and $v$.
\end{enumerate}
\end{propn}

The following relations 
reveal a similarity to almost firm-nonexpansiveness of the projector onto elementally subregular sets 
on the one hand, and almost nonexpansiveness of the same projector on the other.  
\begin{propn}[characterizations of elemental subregularity]\label{t:geom subreg equivalences}
$~$
\begin{enumerate}[(i)]
\item\label{t:geom subreg equivalences i} 
A nonempty set $\Omega\subset\Euclid$ is elementally subregular at $\xbar$ relative to $\Lambda$ for $(y,v)\in\gph\paren{\pncone{\Omega}}$
where $y\in P_\Omega(y+v)$ if and only if there is a neighborhood $U$ of $\xbar$ together with a constant $\varepsilon\ge 0$ such that
\begin{equation}\label{e:geom subreg2}
\norm{x-y}^2\leq\varepsilon\norm{\paren{y'-y} - \paren{x'-x}}\norm{x-y}+
\ip{x'-y'}{x-y}
\end{equation}
holds with $y'=y+v$ whenever $x'\in   U\cap \Lambda$ and $x\in P_\Omega x'$.
\item\label{t:geom subreg equivalences ii}
Let the nonempty set $\Omega\subset\Euclid$ be elementally subregular at $\xbar$ relative to $\Lambda$ 
for $(y,v)\in\gph\paren{\pncone{\Omega}}$
where $y\in P_\Omega(y+v)$ with the constant $\varepsilon\ge 0$ for the neighborhood $U$ of $\xbar$. 
% Denote 
% \begin{eqnarray}
% \norm{v}&=&\norm{\paren{x-x^+}-\paren{y-y^+}}\nonumber\\
% &=&\paren{\norm{y-x}^2-2\cos(\theta)\norm{y-x}\norm{y^+-x^+}+\norm{y^+-x^+}^2}^{1/2} .
% \end{eqnarray}
Then
\begin{equation}%%\label{e:gsr 3}
\norm{x-y}\leq\varepsilon\norm{\paren{y'-y}-\paren{x'-x}}+\norm{x'-y'}
\end{equation}%
holds with $y'=y+v$ whenever $x'\in   U\cap \Lambda$ and $x\in P_\Omega x'$.
\end{enumerate}
\end{propn}
\begin{proof}
 \ref{t:geom subreg equivalences i}:  This is just a rearrangement of the inequality in \eqref{e:geom subreg}. 
 \ref{t:geom subreg equivalences ii}:  Follows by applying the Cauchy-Schwarz inequality to the inner product on the right hand side of \eqref{e:geom subreg2}.
\end{proof}

The next theorem is an update of \cite[Theorem 2.14]{HesseLuke13} to the current terminology.  It establishes
the connection between elemental subregularity of a set and almost nonexpansiveness/averaging of 
the projector onto that set.  Since the cyclic projections algorithm applied to inconsistent feasibility problems involves the 
properties of the projectors at points that are {\em outside} the sets, we show the how the properties depend on whether the reference points 
are inside or outside of the sets.  
The theorem uses the symbol $\Lambda$ to indicate subsets of the sets and 
the symbol $\Lambda'$ to indicate points on some neighborhood whose projection lies in $\Lambda$.  Later, the sets  $\Lambda'$
will be specialized in the context of cyclic projections to sets of points $S_j$ whose projections lie in $\Omega_j$.  One thing to note 
in the theorem below is that the almost nonexpansive/averaging property degrades rapidly as the reference points 
move away from the sets.  Our estimate is severe and could be sharpened somewhat, but it serves our purposes. 
\begin{thm}[projectors and reflectors onto elementally subregular sets]
\label{t:subreg proj-ref}
Let $\Omega\subset\Ebb$ be nonempty closed, and let $U$ be a neighborhood of $\xbar\in \Omega$.  
Let $\Lambda\subset\Omega\cap U$ and $\Lambda'\equiv P_{\Omega}^{-1}(\Lambda)\cap U$.  
% so that  $\xbar\in \Lambda$ 
If $\Omega$ is elementally sub\-reg\-ul\-ar at $\xbar$ 
relative to $\Lambda'$ for each 
\[
(x,v)\in V\equiv\set{(z, w)\in\gph\pncone{\Omega}}{ z+w\in U\und z\in P_\Omega(z+w)} 
\]
 with constant $\varepsilon$ on the neighborhood $U$, then the following hold.
\begin{enumerate}[(i)]
\item\label{t:subreg proj-ref1}
The projector $P_\Omega$ is pointwise almost nonexpansive at each $y\in \Lambda$ on 
$U$ 
with violation $\varepsilon'\equiv 2\varepsilon+\varepsilon^2$. That is, at each $y\in \Lambda$
\[
\norm{x -y}\leq \sqrt{1+\varepsilon'}\norm{x'-y}\quad\forall x'\in U,\, x\in P_\Omega x'.
% \underbrace{\sqrt{1+\varepsilontilde_1}}_{=\left(1+\varepsilon\right)} \norm{x-\xbar}, 
\]
%for all $x'\in U$.

\item\label{t:subreg proj-ref1b}
Let $\varepsilon\in[0,1)$.  The projector $P_\Omega$ is pointwise almost nonexpansive at each 
$y'\in \Lambda'$ with violation $\varepsilontilde$ on 
$U$
%for $\varepsilontilde\equiv\paren{\frac{1+\varepsilon}{1-\varepsilon}}^2-1$, that is,
for $\varepsilontilde\equiv
%\paren{\paren{1+\varepsilon}/\paren{1-\varepsilon}}^2 -1 = 
4\varepsilon/\paren{1-\varepsilon}^2$. That is,
at each $y'\in \Lambda'$
\[
\norm{x -y}\leq \frac{1+\varepsilon}{1-\varepsilon}\norm{x'-y'}\quad\forall x'\in U,\, x\in P_\Omega x',y\in P_\Omega y'.
\]
%for all $x'\in U$. 
\item\label{t:subreg proj-ref2} 
The projector $P_\Omega$ is pointwise almost firmly nonexpansive at each $y\in \Lambda$ 
with violation $\varepsilon'_2\equiv2\varepsilon+2\varepsilon^2$ on $U$. That is,
at each $y\in \Lambda$
% for all  $x\in  U,~ x^+\in P_\Omega x$ and $\xbar\in S$,
\[
\norm{x -y}^2+\norm{x'-x }^2\leq (1+\varepsilon'_2)\norm{x'-y}^2 \quad\forall x'\in U,\, x\in P_\Omega x'.  
\]
%for all  $x'\in U$.
\item\label{t:subreg proj-ref2b}
Let $\varepsilon\in[0,1)$.  The projector $P_\Omega$ is pointwise almost firmly nonexpansive at 
each $y'\in \Lambda'$ 
with violation 
%$\varepsilontilde_2\equiv2\varepsilon\paren{\paren{\frac{1+\varepsilon}{1-\varepsilon}}^2+%
%\paren{\frac{1+\varepsilon}{1-\varepsilon}}}$
$\varepsilontilde_2 \equiv 4\varepsilon\paren{1+\varepsilon}/\paren{1-\varepsilon}^2$
  on $U$. That is,
at each $y'\in \Lambda'$
\[
\norm{x -y}^2+\norm{(x'-x)-(y'-y) }^2\leq \paren{1+\varepsilontilde_2}\norm{x'-y'}^2\quad\forall x'\in U,\, x\in P_\Omega x',y\in P_\Omega y'.
\]
%for all $x'\in U$.
\item\label{t:subreg proj-ref3} 
The reflector $R_\Omega$ is pointwise almost nonexpansive at each $y\in \Lambda$ 
(respectively, $y'\in\Lambda'$) 
with violation $\varepsilon'_3\equiv 4\varepsilon+4\varepsilon^2$ 
(respectively, 
$\varepsilontilde_3\equiv 
%4\varepsilon\paren{\paren{\frac{1+\varepsilon}{1-\varepsilon}}^2+%
%\paren{\frac{1+\varepsilon}{1-\varepsilon}}}$
8\varepsilon\paren{1+\varepsilon}/\paren{1-\varepsilon}^2$
) on $U$; that is,
for all $y\in \Lambda$ (respectively, $y'\in\Lambda'$)
% for all  $x\in  U,~ x^+\in R_\Omega x$ and $\xbar\in S$,
\begin{align*}
&\norm{x-y}\leq\sqrt{1+\varepsilon'_3} \norm{x'-y}
\quad\forall x'\in U,\, x\in P_\Omega x'
\\
(\mbox{respectively},\; &\norm{x-y}\leq\sqrt{1+\varepsilontilde_3} \norm{x'-y'}
\quad\forall x'\in U,\, x\in P_\Omega x',\, y\in P_\Omega y'.)
\end{align*}
%whenever $x\in U$.
\end{enumerate}
\end{thm}

\begin{proof}
First, some general observations about the assumptions.  
The projector is nonempty since $\Omega$ is closed.  Note also that, since $\Lambda\subset\Omega$, 
$P_\Omega y=y$ for all $y\in \Lambda$.  
Since $\Lambda\subset\Lambda'$, $\Omega$ is elementally sub\-reg\-ul\-ar at $\xbar$ 
relative to $\Lambda$ for each 
$
(x,v)\in V
$
 with constant $\varepsilon$ on the neighborhood $U$ of $\xbar$, though the constant $\varepsilon$ may not be optimal
 for $\Lambda$ even if it is optimal for $\Lambda'$. 
Finally, for all $x'\in U$ it holds that $(x, x'-x)\in V$ for all 
$x\in P_\Omega(x')$.  To see this, take any $x'\in U$ and any $x\in P_\Omega(x')$.  Then 
$v=x'-x\in \pncone{\Omega}(x)$ and, by definition $(x,v)\in V$. 

\ref{t:subreg proj-ref1}:   By the 
Cauchy-Schwarz inequality
 \begin{eqnarray} % *}
\norm{x -y}^2
&=& \langle x'- y,x - y\rangle+\langle x -x',x - y\rangle\nonumber\\
&\leq& \norm{x'- y}\norm{x - y}+\langle x' -x,y-x\rangle.\label{e:inter1}
\end{eqnarray}
Now with $v=x'-x\in\pncone{\Omega}(x)$ such that $x'=x+v\in U$ one has $(x,v)\in V$ and by 
the  definition of elemental subregularity of $\Omega$ at $\xbar$ relative to  $\Lambda\subset\Omega$ 
for each 
$
(x,v)\in V
$
 with constant  $\varepsilon$ on the neighborhood $U$ of $\xbar$, the inequality 
% ($\forall x'\in  P^{-1}_\Omega(x)\cap U$)~($\forall  x \in P_\Omega x'$)~
$\langle x'-x,~y-x\rangle\leq \varepsilon \norm{x'-x }\norm{y-x}$ holds
for all $y\in \Lambda = U\cap \Lambda$.
But $\norm{x'-x }\leq \norm{x'-y}$ since $x\in P_\Omega(x')$ and $y\in\Omega$, so 
in fact the inequality 
$\langle x' -x,~y-x\rangle\leq \varepsilon \norm{x'-y }\norm{y-x}$ holds
whenever $y\in \Lambda$.
Combining this with \eqref{e:inter1}  yields, 
for all
$(x,x'-x)\in V$ and $y\in \Lambda$,
\begin{align}\label{e:inter2}
\norm{x - y} \leq (1+\varepsilon)\norm{x'- y}
%\nonumber\\ %= \sqrt{(1+\varepsilon)^2}\norm{x- y} 
= \sqrt{1+(2\varepsilon+\varepsilon^2)}\norm{x'- y}.
\end{align} % *}
%at each $y\in \Lambda$.  
Equivalently, since for all $x'\in U$ it holds that $(x, x'-x)\in V$ for all $x\in P_\Omega(x')$,  
\eqref{e:inter2} holds at each $y\in \Lambda$ 
for all $x\in P_\Omega(x')$ whenever $x'\in U$, that is, $P_\Omega$ is almost nonexpansive at 
each $y\in\Lambda\subset U$ with violation $(2\varepsilon+\varepsilon^2)$ on $U$
as claimed. \hfill$\triangle$

\ref{t:subreg proj-ref1b}:  Since any point $(x,x'-x)\in V$ satisfies $x'-x\in\pncone{\Omega}(x)$ and 
$x\in P_\Omega(x')$, Proposition \ref{t:geom subreg equivalences} \ref{t:geom subreg equivalences ii} applies with $\Lambda$ replaced by $\Lambda'$, 
namely  
\[
 \norm{y-x}\leq \varepsilon\norm{(y'-y)-(x'-x)}+\norm{x'-y'}
\]
for all $y'\in U\cap \Lambda'$ and for every $y\in P_\Omega(y')$.
The triangle inequality applied to $\norm{(y'-y)-(x'-x)}$ then establishes the result.  
\hfill$\triangle$

\ref{t:subreg proj-ref2}:
Expanding and rearranging the norm yields, for all
 $y\in U\cap\Lambda$,
\begin{eqnarray}
&&\!\!\!\!\!\!\norm{x - y}^2+\norm{x'-x }^2\nonumber\\
&&=~\norm{x - y}^2+\norm{x'- y+ y -x }^2\nonumber\\
&&=~\norm{x - y}^2+\norm{x'- y}^2+2\langle x'- y,~ y-x \rangle+\norm{x - y}^2\nonumber\\
&&=2\norm{x - y}^2+\norm{x'- y}^2+ 2\langle x - y, y-x \rangle +2\langle x'-x , y-x  \rangle\nonumber\\
&&\leq~ \norm{x'- y}^2+2\varepsilon\norm{x'-x }\norm{x-y}\label{eq:P:prf1}
\end{eqnarray}
for each $(x,x'-x)\in V$
where the last inequality follows from the definition of elemental subregularity of $\Omega$
at $\xbar$ relative to $\Lambda$ for $(x,x'-x)\in V$. %\eqref{e:geom subreg} 
As in Part~\ref{t:subreg proj-ref1}, since $y\in\Omega$ and $x\in P_\Omega(x')$ it holds that 
$\norm{x'-x }\leq \norm{x'-y}$. 
 Combining \eqref{eq:P:prf1} and part \ref{t:subreg proj-ref1} yields, at each $y\in \Lambda$
\begin{equation}\label{e:inter3} % *}
\norm{x - y}^2+\norm{x'-x}^2
\leq
\left(1+2\varepsilon\left(1+\varepsilon\right)\right)\norm{x'- y}^2
% & \forall x\in  U,~\forall x^+ \in P_\Omega x,~\forall y\in S.
\end{equation} % *}
for all $(x,x-x')\in V$.  Again, since for all $x'\in U$ it holds that $(x, x'-x)\in V$ for all $x\in P_\Omega(x')$,  
\eqref{e:inter3} holds at each $y\in \Lambda$ 
for all $x\in P_\Omega(x')$ whenever $x'\in U$.  By Proposition \ref{t:average char}\ref{t:average char iii} 
with $\alpha=1/2$ and $y^+=P_\Omega(y)=y$ it follows that $P_\Omega$ is almost firmly nonexpansive at 
each $y\in\Lambda\subset U$ with violation $(2\varepsilon+2\varepsilon^2)$ on $U$
as claimed. 
\hfill$\triangle$

\ref{t:subreg proj-ref2b}:
As in part \ref{t:subreg proj-ref1b}, Proposition \ref{t:geom subreg equivalences} 
applies with $\Lambda$ replaced by $\Lambda'$.  Proceeding as in part \ref{t:subreg proj-ref2}
\begin{align}
\norm{x - y}^2+\norm{(x'-x)-(y'-y) }^2 &= 2\norm{x - y}^2+\norm{x'- y'}^2+
  2\langle x'- y',~ y-x \rangle\nonumber\\
&\leq \norm{x'- y'}^2+2\varepsilon \norm{(x'-x)-(y'-y) }\norm{x-y}, %\label{e:2b 1}
\end{align}
where, by elemental subregularity of $\Omega$ at $\xbar$ relative to $\Lambda'$ for $(x,x'-x)\in V$ and 
\eqref{e:geom subreg2} of Proposition \ref{t:geom subreg equivalences}, 
the  last inequality holds for each $(x,x'-x)\in V$ for every 
$y'\in U\cap\Lambda'$ for all $y\in P_\Omega(y')$.  This together with the triangle inequality yields
\begin{align}
%&&\!\!\!\!\!\!
\norm{x - y}^2+\norm{(x'-x)-(y'-y) }^2
\leq \norm{x'- y'}^2+2\varepsilon \norm{x-y}\paren{\norm{x'-y'}+\norm{x-y}}. 
\label{e:2b 2}
\end{align}
Part \ref{t:subreg proj-ref1b} and \eqref{e:2b 2} then give
\begin{align}
%&&\!\!\!\!\!\!
\norm{x - y}^2+\norm{(x'-x)-(y'-y) }^2
%\nonumber\\
\leq \paren{1+4\varepsilon\frac{1+\varepsilon}{\paren{1-\varepsilon}^2}}\norm{x'-y'}^2 
\label{e:2b 3}
\end{align}
for all $(x,x'-x)\in V$ and for all $y\in P_\Omega(y')$ at each 
$y'\in U\cap\Lambda'$ .  
Again, since for all $x'\in U$ it holds that $(x, x'-x)\in V$ for all $x\in P_\Omega(x')$,  
\eqref{e:2b 3} holds at each $y'\in \Lambda'=U\cap \Lambda'$ 
for all $x\in P_\Omega(x')$ whenever $x'\in U$.  By Proposition \ref{t:average char}\ref{t:average char iii} 
with $\alpha=1/2$ and $y^+$ replaced by $y\in P_\Omega(y')$ it follows that $P_\Omega$ is almost firmly nonexpansive at 
each $y'\in\Lambda'\subset U$ with violation
$4\varepsilon(1+\varepsilon)/\paren{1-\varepsilon}^2$
% $2\varepsilon \paren{\frac{1+\varepsilon}{1-\varepsilon}}\paren{1+
%\paren{\frac{1+\varepsilon}{1-\varepsilon}}}$ 
on $U$ as claimed. 
\hfill$\triangle$

\ref{t:subreg proj-ref3}: By part \ref{t:subreg proj-ref2} (respectively, part \ref{t:subreg proj-ref2b}) 
the projector is pointwise almost firmly nonexpansive
at each $y\in \Lambda$ (respectively, $y'\in\Lambda'$) with violation 
$(2\varepsilon+2\varepsilon^2)$ (respectively, $4\varepsilon(1+\varepsilon)/\paren{1-\varepsilon}^2$) on $U$,
and so, by Proposition \ref{t:firmlynonexpansive}\ref{t:averaged2}, $R_\Omega=2 P_\Omega -\Id$ is 
pointwise almost nonexpansive at each $y\in\Lambda$ (respectively, $y'\in\Lambda'$) 
with violation  $(4\varepsilon+4\varepsilon^2)$ 
(respectively, $8\varepsilon(1+\varepsilon)/\paren{1-\varepsilon}^2$)
on $U$.
This completes the proof.
\end{proof}

\subsubsection{Subtransversal collections of sets}
Elemental regularity of sets has been shown to be the source of the almost averaging 
property of the corresponding projectors.  We show in this section that metric subregularity 
of the composite/averaged fixed point mapping is a consequence of how the 
individual sets align with each other.  This impinges on a literature rich in terminology and 
competing notions of stability that have been energetically promoted recently in the context of 
consistent feasibility (see \cite{KruLukNgu16} and references therein).  Our placement of 
metric subregularity as the central organizing principle allows us to extend these notions beyond 
consistent feasibility to {\em inconsistent feasibility}.   Before we can translate the dialect of 
set feasibility into the language of metric subregularity, we need to first extend one of the 
main concepts describing the regularity of collections of sets to collections that 
don't necessarily intersect.  The idea behind the following definition stems from the equivalence 
between metric subregularity of an appropriate set-valued mapping on the product space and subtransversality 
of sets at common points \cite[Theorem 3]{KruLukNgu16}.  The trick to extending this to points 
that do not belong to all the sets is to define the correct set-valued mapping.   
\begin{defn}[subtransversal collections of sets]\label{d:(s)lf}
Let $\klam{\Omega_1, \Omega_2, \dots,\Omega_m}$ be a collection of nonempty closed subsets of $\Ebb$ 
and define $\mmap{\Psi}{\Ebb^m}{\Ebb^m}$ by 
$\Psi(x)\equiv P_{\Omega}\paren{\Pi x}-\Pi x$ where 
$\Omega\equiv\Omega_1\times\Omega_2\times\cdots\times\Omega_m$, 
the projection $P_\Omega$ is with respect to the Euclidean norm on $\Ebb^m$ and 
$\Pi:x=(x_1, x_2, \dots, x_m)\mapsto (x_2, \dots, x_m, x_1)$ is the permutation mapping on the 
product space $\Ebb^m$ for $x_j\in \Ebb$ $(j=1,2,\dots, m)$.
Let $\xbar = (\xbar_1, \xbar_2, \dots, \xbar_m)\in \Ebb^m$ and
 $\ybar\in \Psi(\xbar)$.
\begin{enumerate}[(i)]
 \item\label{d:lf} The collection of sets is said to be {\em subtransversal with gauge $\mu$ relative to $\Lambda\subset \Ebb^m$
% of order $w$ 
at $\xbar$ for $\ybar$}
if $\Psi$ is metrically subregular at $\xbar$ for $\ybar$ on some neighborhood $U$ of $\xbar$ 
(metrically regular on $U\times \{\ybar\}$) with gauge $\mu$ relative to 
$\Lambda$.
% , that is, where there is 
%  a neighborhood $U$ of $\xbar$ together with constants $\kappa>0$ and $w\in (0,1]$ 
% such that 
% \begin{equation}\label{e:linearly focusing}
%    \dist\paren{x, \Psi^{-1}(\ybar)\cap W}\leq \kappa\dist^w\paren{\ybar,\Psi(x)}\quad\forall ~x\in U\cap W.
% \end{equation}
%\item\label{d:slf} The collection is said to be {\em strongly  subtransversal of  order $w$ 
%at the collection of points $\klam{\xbar_1,\xbar_2,\dots,\xbar_m}$ 
% for $\ybar$ relative to $W$} when  $\Psi$ is strongly metrically subregular  of order $w$ at 
%$\xbar\equiv \paren{\xbar_1, \xbar_2, \dots, \xbar_m}$ for $\ybar$ relative to $W$.
%  , that is, when 
% there is a neighborhood $U$ of $\xbar$ together with constants $\kappa>0$ and $w\in (0,1]$ 
% such that 
% \begin{equation}\label{e:str linearly focusing}
%    \norm{x- \xbar}\leq \kappa\dist^w\paren{\ybar,\Psi(x)}\quad\forall ~x\in U\cap W.
% \end{equation}
\item\label{d:transversality} The collection of sets is said to be {\em transversal with gauge $\mu$ relative to $\Lambda\subset \Ebb^m$
at $\xbar$ for $\ybar$}
if $\Psi$ is metrically regular with gauge $\mu$ relative to $\Lambda$
on $U\times V$, for some neighborhoods $U$ of $\xbar$ and $V$ of $\ybar$.
\end{enumerate}
As in Definition \ref{d:(str)metric (sub)reg}, when
$\mu(t)=\kappa t,\,\forall t\in [0,\infty)$,
one
 says ``constant $\kappa$'' instead of ``gauge $\mu(t) =\kappa t$''. When $\Lambda=\Ebb$, the quantifier ``relative to'' is dropped.
\end{defn}
Consistent with the terminology of metric regularity and subregularity, the prefix ``sub'' is meant to indicate the 
pointwise version of the more classical, though restrictive, idea of transversality.  
When the point $\xbar=\paren{\ubar, \cdots,\ubar}$ for $\ubar\in\cap_{j=1}^m\Omega_j$ the following 
characterization of substransversality holds.

\begin{propn}[subtransversality at common points]\label{t:metric characterization}
Let $\Ebb^m$ be endowed with $2$-norm, that is, 
$\norm{(x_1,x_2,\ldots,x_m)}_2= \paren{\sum_{j=1}^m\norm{x_j}^2_{\Ebb}}^{1/2}$. 
A collection $\{\Omega_1, \Omega_2, \dots,\Omega_m\}$ of nonempty closed subsets of $\Ebb$ 
is subtransversal relative to 
$\Lambda\equiv\set{x=(u, u, \dots, u)\in\Ebb^m}{u\in\Ebb}$ at
% the collection of points 
$\xbar = \paren{\ubar, \cdots,\ubar}$ with $\ubar\in \cap_{j=1}^m\Omega_j$ for 
$\ybar=0$ with gauge $\mu$
if there exists a neighborhood 
$U'$ of $\ubar$ together with a gauge $\mu'$ satisfying $\sqrt{m}\mu'\le \mu$ such that 
\begin{equation}\label{eq:locallinear}
\dist\paren{u, \cap_{j=1}^m \Omega_j}\leq\mu'\paren{\max_{j=1,\dots,m}\dist\paren{u, \Omega_i}}, 
\quad\forall ~u\in  U'. 
\end{equation}
Conversely, if $\{\Omega_1, \Omega_2, \dots,\Omega_m\}$ is subtransversal relative to 
$\Lambda$ at $\xbar$ for $\ybar=0$ with gauge $\mu$, then
\eqref{eq:locallinear} is satisfied with any gauge $\mu'$ for which $\mu(\sqrt{m}t)\le \sqrt{m}\mu'(t)$ for all $t\in [0,\infty)$.
\end{propn}
\begin{proof}
Let $x=(u,u,\dots, u)\in\Ebb^m$ with 
$u\in U'$ where $U'$ denotes a neighborhood of $\ubar$.  
Note that $\Pi x=x$ for all $x\in \Lambda$.
% when $\Pi$ is the permutation on the product space. 
Moreover, for $\Psi(x)\equiv P_{\Omega}\paren{\Pi x}- \Pi x$ with 
$\Omega\equiv\Omega_1\times\Omega_2\times\cdots\times\Omega_m$, it holds that 
\begin{equation}\label{e:inverse Psi}
   \paren{\cap_{j=1}^m\Omega_j,\cap_{j=1}^m\Omega_j, \dots, \cap_{j=1}^m\Omega_j}\cap \Lambda = \Psi^{-1}(0)\cap \Lambda.
\end{equation}
To see this, note that any element $z\in\Psi^{-1}(0)\cap \Lambda$ satisfies $z\in \Lambda$ and 
$0\in P_{\Omega}\paren{\Pi z}- \Pi z$, 
which means that $z_i=z_j$ and $z_j\in\Omega_j$ for $i,j=1,2,\dots,m$.  In other words, $z_i\in\cap_{j=1}^m\Omega_j$
and  $z_i=z_j$ for $i,j=1,2,\dots,m$, which is just \eqref{e:inverse Psi}.

Denote $\Omega\equiv\Omega_1\times\Omega_2\times\cdots\times\Omega_m$. For the first implication,
%Then, by norm equivalence, the collection $\klam{\Omega_1, \Omega_2, \dots, \Omega_m}$ satisfies 
if \eqref{eq:locallinear} is satisfied,
% at $\ubar$ 
%with constant $\kappa'$ on the neighborhood $U'$ 
%if and only if, for some finite $\kappa>0$, 
it holds that
\begin{eqnarray}
   \dist\paren{x,\paren{\cap_{j=1}^m\Omega_j, \cap_{j=1}^m\Omega_j, \dots, \cap_{j=1}^m\Omega_j}\cap \Lambda}&=&
 \sqrt{m} \dist\paren{u,\cap_{j=1}^m\Omega_j}\nonumber\\
&\leq&\sqrt{m}\,\mu'\paren{\max_{j=1,\dots, m}\klam{\dist\paren{u,\Omega_j}}}\nonumber\\
&\leq&\sqrt{m}\,\mu'\paren{\dist\paren{x,\Omega}}\nonumber\\
&\leq& \mu\paren{\dist\paren{x,\Omega}}\nonumber\\
   &=&\mu\paren{\dist\paren{\Pi x, P_{\Omega}(\Pi x)}}\nonumber\\
   &=&\mu\paren{\dist\paren{0,\Psi(x)}}\label{e:llr norm}
   \end{eqnarray}
whenever $x\in U\cap \Lambda=\set{x=(u,u,\dots,u)}{u\in U'}$.  
By \eqref{e:inverse Psi}, the inequality \eqref{e:llr norm} is equivalent to 
\begin{equation}\label{e:stvl}
   \dist\paren{x, \Psi^{-1}(0)\cap \Lambda}\leq \mu\paren{\dist\paren{0,\Psi(x)}}, \quad \forall x\in U\cap \Lambda,
\end{equation}
which is the definition of subtransverality of $\klam{\Omega_1, \Omega_2, \dots, \Omega_m}$ relative to $\Lambda$ at $\xbar$ for $0$ with gauge $\mu$.

For the reverse implication,
%Then, by norm equivalence, the collection $\klam{\Omega_1, \Omega_2, \dots, \Omega_m}$ satisfies 
if \eqref{e:stvl} is satisfied, (using \eqref{e:inverse Psi})
% at $\ubar$ 
%with constant $\kappa'$ on the neighborhood $U'$ 
%if and only if, for some finite $\kappa>0$, 
it holds that
\begin{eqnarray}
   \dist\paren{u,\cap_{j=1}^m\Omega_j}&=&
 \frac{1}{\sqrt{m}}\, \dist\paren{x,\paren{\cap_{j=1}^m\Omega_j, \cap_{j=1}^m\Omega_j, \dots, \cap_{j=1}^m\Omega_j}\cap \Lambda}\nonumber\\
&=&\frac{1}{\sqrt{m}}\, \dist\paren{x,\Psi^{-1}(0)\cap \Lambda}\nonumber\\
&\leq&\frac{1}{\sqrt{m}}\, \mu\paren{\dist\paren{0,\Psi(x)}}\nonumber\\
&=&\frac{1}{\sqrt{m}}\,\mu\paren{\dist\paren{x,\Omega}}\nonumber\\
&\leq&\frac{1}{\sqrt{m}}\,\mu\paren{\sqrt{m}\max_{j=1,\dots, m}\klam{\dist\paren{u,\Omega_j}}}\nonumber\\
&\leq& \mu'\paren{\max_{j=1,\dots, m}\klam{\dist\paren{u,\Omega_j}}},
\quad\forall u\in U'.
\nonumber
\end{eqnarray}
This is \eqref{eq:locallinear}.
\end{proof}

Note that if one endows $\Ebb^m$ with the maximum norm, ($\norm{(x_1,x_2,\ldots,x_m)}_{\Ebb^m}:= \max_{1\le j\le m}\norm{x_j}_{\Ebb}$ ), there holds
\begin{align*}
&\dist\paren{x,\paren{\cap_{j=1}^m\Omega_j, \cap_{j=1}^m\Omega_j, \dots, \cap_{j=1}^m\Omega_j}\cap \Lambda} = \dist\paren{u,\cap_{j=1}^m\Omega_j};\\
&\dist\paren{x,\Omega} = \max_{j=1,\dots, m}\dist\paren{u,\Omega_j}\quad \mbox{ for all }\; u \mbox{ and } x\; \mbox{ as above}.
\end{align*}
Then the two properties in Proposition \ref{t:metric characterization} are equivalent for the same gauge $\mu'=\mu$.
\bigskip

By \cite[Theorem 1]{KruLukNgu16}, Proposition \ref{t:metric characterization} shows that 
Definition \ref{d:(s)lf}\ref{d:lf} coincides with subtransversality defined in \cite[Definition~6]{KruLukNgu16} for points 
of intersection.   This notion was developed to bring many other definitions of 
regularities of collections of sets \cite{LewisMalick08, LewisLukeMalick09, BauLukePhanWang13a, BauLukePhanWang13b, 
DruIofLew15, KrugerNguyen15, Kruger06} 
under a common framework.  
% Transversality and subtransversality of collections of sets at points of intersection is defined 
% in \cite{Kruger05} and \cite{KrugerNguyen15} respectively under different names.  In 
% \cite[Definition 6]{KruLukNgu16} the same definitions appear with the names transversality and 
% subtransversality.  
The definition given in \cite{KruLukNgu16}, however, does not immediately lead to a characterization 
of the relation between sets at points 
that are not common to all sets. There is much to be done to align the many different characterizations
of (sub)transversality studied in \cite{KruLukNgu16} with Definition \ref{d:(s)lf} above, but this is not our main 
interest here.  
% to show that, in the case where $\xbar$ is a point of intersection, the definition 
% of (sub)transversality given above is equivalent to (sub)transversality defined in \cite{KruLukNgu16}.
% The identification follows from the equivalent characterization below.
% \begin{propn}[subtransversality at points of intersection]
% \label{t:localregular}
%  A collection of closed, nonempty sets $\Omega_j$ ($j=1,2,\dots m$) is 
% \emph{subtransversal} in the sense of \cite{KrugerNguyen15} $\xhat\in \cap_{j=1}^m \Omega_j$ 
% if and only if there exists a neighborhood $U$ of $\xhat$ together with a constant 
% $\kappa>0$ such that, 
% \begin{equation}\label{eq:locallinear}
% \dist\paren{x, \cap_{j=1}^m \Omega_j}\leq\kappa \max_{i=1,\dots,m}\dist\paren{x, \Omega_i} 
% \quad\forall ~x\in  U. 
% \end{equation}
% The infimum over all $\kappa$ such that \eqref{eq:locallinear} holds is called the \emph{modulus of subtransversality
% for $U$}.
% \end{propn}
% \cite[Theorem 3.1]{KrugerNguyen15}.

\subsubsection{Cyclic projections}\label{s:cp}
Having established the basic geometric language of set feasibility and its connection to the averaging and stability
properties of fixed point mappings,  we can now pursue our main goal for this section: new 
convergence results for cyclic projections between sets with possibly empty intersection, Theorem \ref{t:cp ncvx}
 and 
Corollary \ref{t:cp cvx}.  
The majority of the work, and the source of technical complications, lies in constructing an appropriate 
fixed point mapping in the right space in order to be able to apply Theorem \ref{t:metric subreg convergence}. 
As we have already said, establishing the extent of almost averaging is a straight-forward application of 
Theorem \ref{t:subreg proj-ref}. Thanks to Proposition \ref{t:av-comp av} this can be stated in terms of the 
more primitive property of elemental set regularity. The challenging part is to show that subtransversality 
as introduced above leads to metric subregularity of an appropriate fixed point surrogate for 
cyclic projections, Proposition \ref{t:msr cp}.  In the process we show in Proposition \ref{t:entaglement} that   
elemental regularity and subtransversality become entangled and it is not clear whether they can be 
completely separated when it comes to necessary conditions for convergence of cyclic projections. 

Given a collection of closed subsets of $\Ebb$,  $\{\Omega_1, \Omega_2, \dots, \Omega_m\}$ ($m\geq 2$), and an initial 
point $u^0$, the cyclic projections algorithm generates the sequence $(u^k)_{k\in\Nbb}$ by 
\begin{equation}\label{e:P_0}
 u^{k+1}\in P_0 u^k\quad P_0\equiv P_{\Omega_1}P_{\Omega_2}\cdots P_{\Omega_m}P_{\Omega_1}.
\end{equation}
Since projectors are idempotent, the initial $P_{\Omega_1}$ at the right end of the cycle has no 
real effect on the sequence, though we retain it for technical reasons.  We will assume throughout this 
section that $\Fix P_0\neq \emptyset$.  

Our analysis proceeds on an appropriate product space designed for the cycles associated with 
a given fixed point of $P_0$.
As above we will use $\Omega$ to denote 
the sets $\Omega_j$ on $\Ebb^m$:  
$\Omega\equiv \Omega_1, \times\Omega_2\times\cdots\times\Omega_m$.  Let $\ubar\in \Fix P_0$ and let 
$\zetabar\in \Zcal(\ubar)$ where 
\begin{equation}%%\label{e:Zcal}
   \Zcal(u)\equiv \set{\zeta\equiv z-\Pi z}{ z\in W_0\subset\Ebb^m, ~ z_1=u}
\end{equation}%
for 
\begin{equation}%%   \label{e:cyclic subsets}
   W_0\equiv\set{x\in\Ebb^m}%
   {x_m\in P_{\Omega_m} x_1,\, x_j\in P_{\Omega_{j}}x_{j+1},~j=1,2,\dots,m-1}.
\end{equation}%
Note that $\sum_{j=1}^{m}\zetabar_j=0$.
The vector $\zetabar$ is a 
{\em difference vector} which gives information regarding the intra-steps 
of the cyclic projection operator $P_0$ at the fixed point $\ubar$.
In the case of only two sets, a difference vector is frequently called a {\em gap} vector \cite{BauBorSVA, BCL3, Luke08, BauMou16b}.
This is unique in the convex case, but need not be in the nonconvex case 
(see Lemma \ref{l:unique cycles} below).  
In the more general setting 
we have here, this corresponds to nonuniqueness of cycles for cyclic projections.  This greatly complicates matters
since the fixed points associated with $P_0$ will not, in general, be associated with cycles that are the same 
length and orientation.  Consequently, the usual trick of looking at the zeros of $P_0-\Id$ is rather uninformative, and 
another mapping needs to be constructed which distinguishes fixed points associated with different cycles.  
The following development establishes some of the key properties of difference vectors and cycles which 
then motivates the mapping that we construct for this purpose. 
  
To analyze the cyclic projections algorithm we consider the sequence on the product space  on $\Ebb^m$, 
$\paren{x^k}_{k\in\Nbb}$ generated by $x^{k+1}\in T_{\zetabar}x^k$ with
\begin{subequations}\label{e:Tc prod}
\begin{equation}%%\label{e:Tcp}
 \mmap{T_{\zetabar}}{\Ebb^m}{\Ebb^m}:x\mapsto\set{\paren{x_1^+, x_1^+-\zetabar_1, \dots, x_1^+-\sum_{j=1}^{m-1}\zetabar_j}}{x_1^+\in P_0x_1}
\end{equation}%
for $\zetabar\in \Zcal(\ubar)$ where $\ubar\in\Fix P_0$. 
% \begin{equation}\label{e:Pj}
% P_{j}\equiv P_{\Omega_j}P_{\Omega_{j+1}}\cdots P_{\Omega_{m}}
% P_{\Omega_1}P_{\Omega_2}\cdots P_{\Omega_{j-1}}P_{\Omega_j}.
% \end{equation}
\end{subequations}
% The projector $P_j$ is just the permutated cyclic projection.  
In order to isolate cycles we restrict our attention to relevant subsets of $\Ebb^m$.  
These are
\begin{subequations}%\label{e:cyclic sets}
\begin{eqnarray}
 W(\zetabar)&\equiv&\set{x\in\Ebb^m}{x-\Pi x = \zetabar},
%\label{e:loop}
 \\
 L&\equiv&\mbox{ an affine subspace with } \mmap{T_{\zetabar}}{L}{L}\und
 %\label{e:affine hull}
 \\ 
 \Lambda&\equiv&L\cap W(\zetabar).
% \label{e:alltogether}
\end{eqnarray}
\end{subequations}
The set $W(\zetabar)$ is an affine
%linear 
transformation of the diagonal of the product space and thus 
an affine subspace: for $x, y\in W(\zetabar)$, $z=\lambda x+(1-\lambda)y$ satisfies $z-\Pi z =\zetabar$ for all 
$\lambda\in\Rbb$.  This affine subspace is used to characterize the local geometry of the sets in 
relation to each other at fixed points of the cyclic projection operator.  
% $\xbar_j\in P_j \xbar_j$ but $\xbar_j\notin P_{_\Omega_j}\xbar_{j+1}$.  

Points in $\Fix P_0$ can correspond to cycles of different lengths, 
hence an element $x\in \Fix T_{\zetabar}$ need not be in $W_0$ and vice verse, 
as the next example 
demonstrates.    
\begin{eg}[$\Fix T_{\zetabar}$ and $W_0$]\label{eg:Tzeta supset W0}
Consider the sets $\Omega_1=\{0, 1\}$ and $\Omega_2=\{0, 3/4\}$.   The 
cyclic projections operator $P_0$ has fixed points $\{0, 1\}$ and two 
corresponding cycles, $\Zcal(0)=\{(0,0)\}$ and $\Zcal(1)=\{(1/4,-1/4)\}$.
Let $\zetabar=(1/4,-1/4)$.   Then $(0,-1/4)\in\Fix T_{\zetabar}$ but 
$(0,-1/4)\notin W_0$.  Conversely, the vector $(0, 0)\in W_0$, but 
$(0,0)\notin\Fix T_{\zetabar}$.  The point $(1,3/4)$, however, belongs
to both $W_0$ and $\Fix T_{\zetabar}$.
\end{eg}

The example above shows that  what distinguishes elements in 
$\Fix T_{\zetabar}$ from each other is whether or not they also belong to $W_0$.  
% For instance, if the sets 
% $\Omega_j$ have nonempty intersection, then at $\ubar\in\cap_{j=1}^m\Omega_j$ the 
% corresponding fixed point of $T_{\zetabar}$, regardless of the displacement vector $\zetabar$, is 
% $\xbar\equiv \paren{\ubar, \ubar-\zetabar_1, \dots, \ubar-\sum_{j=1}^{m-1}\zetabar_j}.$
% Again, by construction of $T_{\zetabar}$, the 
% fixed point $\xbar\in W(\zetabar)$ but it need not hold that $\xbar\in W_0$ defined by 
% \eqref{e:cyclic subsets} for the given $\zetabar$.  If indeed $\zetabar=0$, then  
% clearly $\xbar=(\ubar, \ubar, \dots, \ubar)\in W(0)\cap W_0$.  
The next lemma establishes that, on appropriate subsets,  
a fixed point of $T_{\zetabar}$ can be identified meaningfully with a vector 
in the image of the mapping $\Psi$ in Definition \ref{d:(s)lf} which is used to characterize the 
alignment of the sets $\Omega_j$ to each other at points of interest (in particular, fixed points of the 
cyclic projections operator). 

\begin{lemma}\label{l:W relations}
Let $\ubar\in \Fix P_0$ and let 
$\zetabar\in \Zcal(\ubar)$.  Define 
$\Psi\equiv \paren{P_\Omega-\Id}\circ \Pi$ and  $\Phi_{\zetabar}\equiv T_{\zetabar}-\Id$.
\begin{enumerate}[(i)]
\item \label{l:W relations -1} $T_\zetabar$ maps $W(\zetabar)$ to itself. Moreover 
$x\in \Fix T_{\zetabar}$ if and only if 
$x\in W(\zetabar)$ with $x_1\in\Fix P_0$.  Indeed, 
\begin{equation}%%\label{e:Fix Tcbar}
   \Fix T_{\zetabar}=\set{x=(x_1, x_2, \dots, x_m)\in\Ebb^m}%
  {x_1\in\Fix P_0, ~x_j=x_1-\sum_{i=1}^{j-1}\zetabar_i,~ j=2,3,\dots,m}.
\end{equation}%
\item\label{l:W relations i} A point $\zbar\in\Fix T_{\zetabar}\cap W_0$ 
if and only if $~\zetabar\in \Psi(\zbar)$ if and only if  $~\zetabar\in\paren{\Phi_{\zetabar}\circ\Pi}(\zbar)$.
\item\label{l:W relations ii}  
$   \Psi^{-1}(\zetabar)\cap W(\zetabar)\subseteq \Phi_{\zetabar}^{-1}(0)\cap W(\zetabar).$
\item\label{l:W relations iii} 
If the distance is with respect to the Euclidean norm then 
$\dist\paren{0,\Phi_{\zetabar}(x)} = \sqrt{m}\dist\paren{x_1,P_0x_1}.$
%\item\label{l:W relations iv} Define 
%$\Ccal(x)\equiv \paren{P_0x_1, P_{\Omega_2}P_{\Omega_3}\cdots P_{\Omega_m}x_1, \dots, P_{\Omega_m}x_1}$. 
%If $x_1\in \Omega_1$ then 
%\begin{equation}\label{e:W iv} 
%P_0x_1 - x_1 \subseteq \set{\sum_{j=1}^m y_j}{y_j\in \paren{\Psi(x')}_j, ~x'\in \Ccal(x)}.
%\end{equation}
\end{enumerate}
\end{lemma}

\begin{proof}
\ref{l:W relations -1}:  This is immediate from the definitions of $W(\zetabar)$ and $T_{\zetabar}$. 

\ref{l:W relations i}: From the definition of $W_0$ it follows directly that $\zetabar\in \Psi(\zbar)$ if and 
only if $\zbar\in\Fix T_{\zetabar}\cap W_0$.  Moreover,  
$\zetabar\in \Phi_{\zetabar}(\Pi\zbar)=T_{\zetabar}\Pi\zbar-\Pi\zbar$
if and only if for each $j=1,2,\dots,m$ it holds that 
$\zbar_{j+1}+\zetabar_{j}\in \paren{T_{\zetabar}\Pi\zbar}_j = u - \sum_{i=1}^{j-1}\zetabar_i$
for some $u\in P_0\zbar_2$ and $\zbar_{m+1}\equiv \zbar_1$. 
Equivalently, for some $u\in P_0\zbar_2$ it holds that $\zbar_{j+1}+ \sum_{i=1}^{j}\zetabar_i=u$ 
for all $j=1, 2, \dots, m$.  Since $\sum_{i=1}^{m}\zetabar_i=0$, 
then $\zbar_1=u$, so $\zbar_{j+1}=\zbar_1- \sum_{i=1}^{j}\zetabar_i$ for all $j=1, 2, \dots, m$
and $\zbar_1\in P_0\zbar_2$ which, thanks to the redundancy of the first projector in the definition of 
$P_0$ \eqref{e:P_0} and the definition of $W_0$, is equivalent to $\zbar\in\Fix T_{\zetabar}\cap W_0$, 
as claimed. 

To establish \ref{l:W relations ii}, let $\zbar\in \Psi^{-1}(\zetabar)\cap W(\zetabar)$.  
Then $\zetabar\in\paren{\paren{P_\Omega-\Id}\circ\Pi}(\zbar)$, and,  since $\zbar\in W(\zetabar)$, also  
$\zetabar=\zbar-\Pi\zbar$.   Hence $\zetabar=\zbar-\Pi\zbar$ and 
$\zbar\in P_\Omega\paren{\Pi\zbar}$.  But this implies that $\zetabar=\zbar-\Pi\zbar$ and  
$\zbar_1\in P_0\zbar_1$, hence $\Phi_{\zetabar}(\zbar)=0$ and $\zbar\in W(\zetabar)$.  
That is, $\zbar\in\Phi_{\zetabar}^{-1}(0)\cap W(\zetabar)$ which 
verifies \ref{l:W relations ii}.

Relation \ref{l:W relations iii} is obvious from the definition of $\Phi_{\zetabar}$.
%
%To see \ref{l:W relations iii}, note that 
%for each $x^+\in T_{\zetabar}(x)$ and any $x\in W(\zetabar)$ one has
%\begin{eqnarray*}
%(\forall j=1, 2,3,\dots,m)\quad y_j &=& \paren{x_1^+ - \sum_{i=1}^{j-1}\zetabar_i} - x_j \\
%&=&\paren{x_1^+ - \sum_{i=1}^{j-1}\zetabar_i} - \paren{x_1 - \sum_{i=1}^{j-1}\zetabar_i}\\
%&=&x_1^+ - x_1.
%\end{eqnarray*}
%So for $x^+\in T_{\zetabar}(x)$ one has 
%$\Phi(\zetabar)\ni y\equiv x^+-x= \paren{x_1^+ - x_1, x_1^+ - x_1,\dots,  x_1^+ - x_1}$.
%The result follows immediately.  
%
%Relation \ref{l:W relations iv} follows from writing $P_0x_1-x_1$ as a telescoping sum:
%$r\in    P_0x_1-x_1 $ if and only if 
%\[
%   r = \sum_{j=1}^m x'_j - x'_{j+1}\subseteq \sum_{j=1}^m P_{\Omega_j}x'_{j+1} - x'_{j+1} 
%\]
%where $x'_{j}\in P_{\Omega_j}P_{\Omega_{j+1}}\cdots P_{\Omega_m} x_{1}$ ($x'_{m+1}=x_1 $).
%The result then follows by writing $y\in \Psi(x')$. 
\end{proof}

\begin{lemma}[difference vectors: cyclic projections]\label{l:unique cycles}
	Let $\Omega_j\subseteq\Ebb$ be nonempty and closed ($j=1,2,\dots,m$).
	Let $S_0\subset\Fix P_0$, let $U_0$ be a neighborhood of $S_0$ and define 
	$U\equiv \set{z = (z_1,z_2,\ldots,z_m)\in W_0}{z_1\in U_0}$.  Fix
	$\bar{u}\in S_0$ and  the difference vector $\zetabar\in\mathcal{Z}(\bar{u})$ with $\zetabar=\zbar-\Pi\zbar$ 
	for the point $\zbar = (\zbar_1,\zbar_2,\ldots,\zbar_m)\in W_0$ having $\zbar_1=\bar{u}$.		
	If $\Omega_j$ is elementally subregular at $\zbar_j$ for $(\zbar_j,0)\in\gph\pncone{\Omega_j}$
	% relative to $U_j:=p_j(U)$ for all $(x, 0)\in\gph\pncone{\Omega_j}\cap U_j$
% 	  \[
% 	  (x,v)\in V_j\equiv\set{(z, w)\in\gph\pncone{\Omega_j}}{ z+w\in U_j\und z\in P_{\Omega_j}(z+w)} 
% 	  \]
	  with constant $\varepsilonbar_j$ and neighborhood $U_j:=p_j(U)$ of $\zbar_j$ (where $p_j$ is the $j$th coordinate projection operator),
% 	Let $T_j$ be pointwise almost averaged at $\zbar_j$ with violation $\varepsilon_j$ and averaging 
% 	constant $\alpha_j$ on $U_j\equiv p_j(U)$ where $p_j:\Ebb^m\to\Ebb$ denotes the $j$th coordinate 
%	projection operator  ($j=1,2,\dots,m$).
	then
	\begin{equation}%% \label{eq:unique cycles}
	\|\zetabar-\zeta\|^2 \leq \sum_{j=1}^m\varepsilon_j \|\zbar_j-z_j\|^2 \quad (\varepsilon_j\equiv 2\varepsilonbar_j+2\varepsilonbar_j^2)
	\end{equation}%
	for the difference vector $\zeta\in\mathcal{Z}(u)$ with $u\in S_0$ and  
	$\zeta=z-\Pi z$ where $z= (z_1,z_2,\ldots,z_m)\in W_0$ with $z_1=u$. 
	If the sets $\Omega_j$ ($j=1,2,\ldots, m$) are in fact convex,  
% 	If the mapping $T_j$ is in fact pointwise averaged at $\zbar_j$ on $U_j$ ($j=1,2,\dots,m$), 
 	then the difference vector is unique and independent of the initial point $\ubar$, that is, 
	$\mathcal{Z}(u)=\{\zetabar\}$ for all $u\in S_0$.

% 	Let 
% 	$\bar{u}\in S_0\equiv \Fix P_0$ and $\zetabar\in\mathcal{Z}(\ubar)$ with $\zetabar=\zbar-\Pi\zbar$ 
% 	for a point $\zbar=(\zbar_1,\zbar_2,\dots,\zbar_m)\in W_0$ having $\zbar_1=\ubar$. If the projector 
% 	$P_{\Omega_j}$ is pointwise almost firmly nonexpansive at $\zbar_j$ with violation $\varepsilon_j$ 
% 	on $p_j(W_0)$ ($p_j$s given in Lemma~\ref{l:difference vector averaged operators}) ($j=1,2,\dots,m$) 
% 	then, for any $u\in S_0$ and $\zeta\in\mathcal{Z}(u)$ with $\zeta=z-\Pi z$ for 
% 	$z= (z_1,z_2,\ldots,z_m)\in W_0$ having $z_1=u$, it holds
% 	\begin{equation}% 
% 		\|\zetabar-\zeta\|^2 \leq \sum_{j=1}^m\varepsilon_j \|\zbar_j-z_j\|^2.
% 	\end{equation}%
% 	In particular, if the sets $\Omega_j$ ($j=1,2,\dots,m$) are convex, the difference vector is unique and 
% 	independent of the initial point $\ubar$. That is, there exists a unique $\zetabar\in\Ebb^m$ such that 
% 	$\mathcal{Z}(u)=\{\zetabar\}$ for all $u\in S_0$.
\end{lemma}
\begin{proof} 
	 Note that $U_0\subset\Omega_1$ and $U_j\subset\Omega_j$ ($j=1,2,\ldots, m$). 
	 By Theorem \ref{t:subreg proj-ref}\ref{t:subreg proj-ref2}, the projectors $P_{\Omega_j}$ are 
	  pointwise almost firmly nonexpansive at $\zbar_j$
	  %\in U_j\subset\Omega_j$ 
	  on $U_j$ with violation 
	  $\varepsilon_j\equiv 2\varepsilonbar_j+2\varepsilonbar_j^2$ (and averaging constant $\alpha_j=1/2$). 
	   If the sets $\Omega_j$ are convex, then the violation $\varepsilon_j=0$ and the projectors are 
	  firmly nonexpansive (globally). 
	 The result then follows by specializing Lemma~\ref{l:difference vector averaged operators} to pointwise 
	 almost firmly nonexpansive (respectively, firmly nonexpansive) projectors.
\end{proof}

\begin{propn}[metric subregularity of cyclic projections]\label{t:msr cp}%\label{t:proj coerciv}
Let $\ubar\in \Fix P_0$ and $\zetabar\in \Zcal(\ubar)$, let $\xbar=(\xbar_1, \xbar_2, \dots, \xbar_m)\in W_0$ 
satisfy $\zetabar= \xbar-\Pi\xbar$ with $\xbar_1=\ubar$,
 and let $L$ be an affine subspace containing $\xbar$ with $\mmap{T_{\zetabar}}{L}{L}$.
Suppose the following hold:
\begin{enumerate}[(a)]
\item\label{t:msr cp i} the collection of sets $\{\Omega_1, \Omega_2, \dots, \Omega_m\}$ is subtransversal at 
$\xbar$ for $\zetabar$ relative to 
$\Lambda\equiv L\cap W(\zetabar)$ with constant $\kappa$ and neighborhood $U$ of $\xbar$;
 \item\label{t:msr cp ii} 
there exists a positive constant $\sigma$ such that 
\[
   \dist\paren{\zetabar, \Psi(x)}\leq \sigma \dist(0, \Phi_{\zetabar}(x)),\quad \forall x\in \Lambda\cap U \mbox{ with } x_1\in\Omega_1.
\]
\end{enumerate}
Then the mapping $\Phi_{\zetabar}\equiv T_{\zetabar}-\Id$ is metrically subregular for $0$ on $U$ (metrically regular on $U\times \{0\}$) relative to 
$\Lambda$ with constant $\kappabar=\kappa\sigma$.
\end{propn}
\begin{proof}
A straightforward application of the assumptions and Lemma \ref{l:W relations}\ref{l:W relations ii} yields
\begin{eqnarray*}
(\forall x\in U\cap \Lambda \mbox{ with } x_1\in\Omega_1)\quad \dist\paren{x, \Phi_{\zetabar}^{-1}(0)\cap \Lambda}&\leq& 
\dist\paren{x, \Psi^{-1}(\zetabar)\cap \Lambda}\\
&\leq& 
\kappa\dist\paren{\zetabar, \Psi(x)}\\
&\leq& 
\kappa\sigma\dist\paren{0, \Phi_{\zetabar}(x)}.
\end{eqnarray*}
In other words, $\Phi_{\zetabar}$ is metrically subregular for $0$ on $U$ relative to 
$\Lambda$ with constant $\kappabar$, as claimed.  
\end{proof}

\begin{eg}[two intersecting sets]\label{eg:two sets}
To provide some insight into condition  \ref{t:msr cp ii} of Proposition \ref{t:msr cp} it is 
instructive to examine the case of two sets with nonempty intersection.  
Let $\xbar = (\ubar,\ubar)$ with $\ubar \in \Omega_1\cap\Omega_2$ and the difference vector $\zetabar=0\in \Zcal(\xbar)$.
To simplify the presentation, let us consider
% without loss of generality
 $L=\Ebb^2$ and $U=U'\times U'$, where $U'$ is a neighborhood of $\ubar$.
Then, one has
$\Lambda = W(0) =\{(u,u):u\in \Ebb\}$ and, hence,
$x\in \Lambda\cap U$  with $x_1\in\Omega_1$ is equivalent to $x=(u,u)\in U$ with $u\in\Omega_1\cap U'$.
For such a point $x=(u,u)$, one has
\begin{align*}
&\dist(0,\Psi(x))=\dist(u,\Omega_2),\\
&\dist(0,\Phi_{0}(x))=\sqrt{2}\dist\paren{u,P_{\Omega_1}P_{\Omega_2}(u)},
\end{align*}
where the last equality follows from the representation $\Phi_{0}(x)=\{(z-u,z-u)\in \Ebb^2:z\in P_{\Omega_1}P_{\Omega_2}(u)\}$.\\
 \ref{t:msr cp ii} of Proposition \ref{t:msr cp} becomes
\begin{equation}\label{Tha_03}
\dist(u,\Omega_2)\leq \gamma\dist(u,P_{\Omega_1}P_{\Omega_2}(u)),\quad \forall u\in \Omega_1\cap U'.
% \mbox{ with } (u,u)\in U,
\end{equation}
where $\gamma:=\sqrt{2}\sigma>0$.  In \cite[Remark 12]{KruLukNgu16} the phenomenon of entanglement of 
elemental subregularity and regularity of collections of sets is briefly discussed in the context of 
other notions of regularity in the literature.   Inequality \eqref{Tha_03} serves as a type of 
conduit for this entanglement of regularities  as   
Proposition \ref{t:entaglement} demonstrates. 
%\dots}
\end{eg}
\begin{proposition}[elemental subregularity and \eqref{Tha_03} imply subtransversality]\label{t:entaglement}
Let $\ubar\in \Omega_1\cap \Omega_2$ and $U'$ be the neighborhood of $\ubar$ as in Example \ref{eg:two sets}.
Suppose that condition \eqref{Tha_03} holds and that the set $\Omega_1$ is elementally subregular relative to 
$\Omega_2$ at $\ubar$ for all $(\ybar,0)$ with $\ybar\in \Omega_1\cap U'$ with constant 
$\varepsilon<1/(1+\gamma^2)$ and the neighborhood $U'$.
Then $\{\Omega_1,\Omega_2\}$ is subtransversal at $\ubar$.
\end{proposition}
\begin{proof}
Choose a number $\delta'>0$ such that $\Ball_{2\delta'}(\ubar)\subset U'$.
Take any $u\in \Omega_1\cap \Ball_{\delta'}(\ubar)$ and $u^+\in P_{\Omega_1}P_{\Omega_2}(u)$. Let $u'\in P_{\Omega_2}(u)$ 
such that $u^+\in P_{\Omega_1}(u')$.
Note that $u'\in \Omega_2 \cap U'$. Without loss of generality, we can assume $u'\notin \Omega_1$.
Then $\|u-u'\|\ge \|u'-u^+\|>0$.
The elemental regularity of $\Omega_1$ relative to $\Omega_2$ at $\ubar$ for $(u,0)$ with constant $\varepsilon$ and neighborhood $U'$ yields
\begin{align*}
\ip{u'-u^+}{u-u^+}\le \varepsilon\|u'-u^+\|\|u-u^+\|.
\end{align*}
This inequality and condition \eqref{Tha_03} (note that $\dist(u,\Omega_2)=\|u-u'\|$ and $\dist(u,P_{\Omega_1}P_{\Omega_2}(u))\le \|u-u^+\|$) yield
\begin{align*}
\|u-u'\|^2 &= \|u-u^+\|^2+\|u^+-u'\|^2 +2\ip{u-u^+}{u^+-u'}\\
&\ge \|u-u^+\|^2+\|u^+-u'\|^2 - 2\varepsilon\|u'-u^+\|\|u-u^+\|\\
&= (1-\varepsilon)\paren{\|u-u^+\|^2+\|u^+-u'\|^2} +\varepsilon\paren{\|u'-u^+\|-\|u-u^+\|}^2\\
&\ge (1-\varepsilon)\paren{\|u-u^+\|^2+\|u^+-u'\|^2}\\
&\ge (1-\varepsilon)\paren{\frac{1}{\gamma^2}\|u-u'\|^2+\|u^+-u'\|^2}.
\end{align*}
It is clear that $\frac{1}{1-\varepsilon}\ge \frac{1}{\gamma^2}$, and
hence
\begin{align}%\label{Tha_06}
\|u'-u^+\| \le c\|u-u'\|,
\end{align}
where $c:=\sqrt{\frac{1}{1-\varepsilon}-\frac{1}{\gamma^2}}\in [0,1)$ as $\varepsilon<1/(1+\gamma^2)$.\\
Choose a number $\delta>0$ such that $\frac{1+c}{1-c}\delta\le \delta'$.
%$\Ball_{\frac{\delta(1+c)}{1-c}}(\ubar)\subset U'$.
Employing the basic argument originated in \cite[Theorem 5.2]{LewisLukeMalick09}, one can derive that for any given point 
$u\in \Ball_{\delta}(\ubar)\cap\Omega_1$, there exists a point $\utilde\in \Omega_1\cap\Omega_2$ such that
\begin{align*}
\|u-\utilde\| \le \frac{2}{1-c}\|u-u'\|.
\end{align*}
In other words,
\begin{align*}
\frac{1-c}{2}\dist(u,\Omega_1\cap\Omega_2) \le \frac{1-c}{2}\|u-\utilde\| \le \|u-u'\| = \dist(u,\Omega_2),\quad \forall u\in \Ball_{\delta}(\ubar)\cap\Omega_1.
\end{align*}
The subtransversality of $\{\Omega_1,\Omega_2\}$ at $\ubar$ now follows from 
Proposition \ref{t:metric characterization} (or alternatively \cite[Theorem 1(iii)]{KruLukNgu16}).
\end{proof}

The main result of this section can now be presented.  This statement uses the full technology of regularities relativized 
to certain sets of points $S_j$ introduced in Definitions \ref{d:set regularity} and \ref{d:ane-aa} and used
in Proposition \ref{t:av-comp av}, as well as the expanded notion of subtransversality of sets at points of nonintersection
introduced in Definition \ref{d:(s)lf} and applied in Proposition \ref{t:msr cp}.  

\begin{thm}[convergence of cyclic projections]\label{t:cp ncvx}
% Let $\Omega_j$ ($j=1,2,\dots, m$) be closed subsets of $\Ebb$,  denote
%  $\Omega\equiv\Omega_1\times\Omega_2\times\cdots\times\Omega_m$
% and let 
% $P_0\equiv P_{\Omega_1}P_{\Omega_2}\cdots P_{\Omega_m} P_{\Omega_1}$. 
Let $S_0\subset \Fix P_0\neq \emptyset$ and  $Z\equiv \cup_{u\in S_0}\Zcal(u)$.
% \set{\zetabar\equiv \xbar-\Pi\xbar}{ \xbar\in W_0\subset\Ebb^m, ~ \xbar_1\in S_0}$  
Define 
\begin{equation}\label{e:Sj}
S_j\equiv \bigcup_{\zeta\in Z}\paren{S_0- \sum_{i=1}^{j-1}\zeta_j}
\quad (j=1,2\dots,m).
\end{equation}
Let $U\equiv U_1\times U_2, \times\cdots\times U_m$ 
be a neighborhood of $S\equiv S_1\times S_2\times\cdots\times S_m
%\cap W_0
$ 
and suppose that % , for $j=1,2\dots,m$, 
\begin{subequations}\label{e:neighborhood relations}
\begin{equation}
P_{\Omega_j}\paren{u- \sum_{i=1}^{j}\zeta_j}\subseteq S_0- \sum_{i=1}^{j-1}\zeta_j
\quad \forall~ u\in S_0, \forall~\zeta\in Z ~ 
\mbox{ for each } ~j=1,2\dots,m,
\label{e:S cond} 
\end{equation}
\begin{equation}\label{e:Uj}
P_{\Omega_{j}}U_{j+1}\subseteq U_j
%\und S_j\subset U_j 
\mbox{ for each } ~j=1,2\dots,m
\quad (U_{m+1}\equiv U_1).
\end{equation}
\end{subequations}
% and $\map{\Pi}{\Ebb^m}{\Ebb^m}:(x_1, x_2, \dots, x_m)\mapsto (x_2, \dots, x_m, x_1)$.
%  and all other points $\wbar\in \Fix P_0\cap U_0$ such that $\zbar\in W_0$
% where $\zbar_j=\wbar - \sum_{i=1}^{j-1}\zetabar_j$ for $j=1,2,\dots,m$.  
Let $\Lambda\equiv L\cap \aff\paren{\cup_{\zeta\in Z }W(\zeta)}\supset S$ such that $\mmap{T_\zeta}{\Lambda}{\Lambda}$
for all $\zeta\in  Z $ and some affine subspace $L$.   
% Define $\Phi_{\zetabar}\equiv T_{\zetabar}-\Id$ and 
% $\Psi\equiv (P_\Omega-\Id)\circ\Pi$. 
Suppose that the following hold:
\begin{enumerate}[(a)]
\item\label{t:cp ncvx i} the set $\Omega_j$ is elementally subregular at all $\xhat_j\in S_j$ relative to $S_j$ for each 
\[
(x_j, v_j)\in V_j\equiv\set{(z, w)\in\gph\pncone{\Omega_j}}{ z+w\in U_j\und z\in P_{\Omega_j}(z+w)} 
\]
 with constant $\varepsilon_j\in (0,1)$ on the neighborhood $U_j$ for $j=1, 2, \dots, m$;
\item\label{t:cp ncvx ii} for each $\xhat = (\xhat_1, \xhat_2, \dots, \xhat_m)\in S$, the collection of sets 
$\{\Omega_1, \Omega_2, \dots, \Omega_m\}$ is subtransversal at 
% the collection of points $\klam{\xhat_1, \xhat_2, \dots, \xhat_m}$ 
$\xhat$ for $\zetahat\equiv \xhat-\Pi\xhat$ relative to 
$\Lambda$ with constant $\kappa$ on the neighborhood $U$;
%  \item\label{e:inverse cond}  $\Psi^{-1}(\ybar)\cap \Lambda\cap U\subset\Phi^{-1}(\ybar)\cap \Lambda\cap U$;
 \item\label{t:cp ncvx iii} 
there exists a positive constant $\sigma$ such that for all $\zetahat\in Z $ 
% for all  
% $x'\in\Ccal(x)\equiv \paren{P_0x_1, P_{\Omega_2}P_{\Omega_3}\cdots P_{\Omega_m}x_1, \dots, P_{\Omega_m}x_{1}}$,
\[
   \dist\paren{\zetahat, \Psi(x)}\leq \sigma \dist(0, \Phi_{\zetahat}(x))
% \und 
% \dist\paren{\zetabar, \Psi(x')}\leq \rho\dist(0, \Phi_{\zetabar}(x))
\]
holds whenever $x\in \Lambda\cap U$ with $x_1\in\Omega_1$;
\item\label{t:cp ncvx iv} $\dist(x, S) \le \dist\paren{x, \Phi_{\zetahat}^{-1}(0)\cap \Lambda}$ for all $x\in U\cap \Lambda$, for all $\zetahat\in Z $.
\end{enumerate}
Then for $\zetabar\in Z $ fixed and $\xbar\in S$ with $\zetabar=\Pi\xbar-\xbar$, 
the sequence $\paren{x^k}_{k\in \Nbb}$ generated by $x^{k+1}\in T_{\zetabar}x^k$
% for $T_{\zetabar}$
% defined by \eqref{e:Tc prod}, 
seeded by a point $x^0\in W(\zetabar)\cap U$
% for $W(\zetabar)$ defined by \eqref{e:loop} 
with $x_1^0\in \Omega_1\cap U_1$ satisfies
\[
 \dist\paren{x^{k+1}, \Fix T_{\zetabar}\cap S}\leq c\dist(x^k, S)
\]
whenever $x^k\in U$
for 
\begin{equation}%%\label{e:c cp}
c\equiv\sqrt{1+\varepsilonbar-\frac{1-\alpha}{\alpha\kappabar^2}} 
\end{equation}%
with 
\begin{equation}\label{e:e a cp}
\varepsilonbar\equiv \prod_{j=1}^m\paren{1+\varepsilontilde_j}-1,\quad 
 \varepsilontilde_j\equiv 
  4\varepsilon_j\frac{1+\varepsilon_j}{\paren{1-\varepsilon_j}^2}, 
 \quad
\alpha\equiv \frac{m}{m+1}
\end{equation}
and $\kappabar=\kappa\sigma$.
If, in addition, 
\begin{equation}%%\label{e:kappa-epsilon cp}
\kappabar < \sqrt{\frac{1-\alpha}{\varepsilonbar\alpha}},
\end{equation}%
then $\dist\paren{x^k, \Fix T_{\zetabar}\cap S}\to 0$, and hence $\dist\paren{x_1^k, \Fix P_0\cap S_1}\to 0$,  at least linearly with rate $c<1$. 
\end{thm}

\begin{proof}
The neighborhood $U$ can be replaced by an enlargement of $S$, hence the result follows from Theorem \ref{t:metric subreg convergence} once 
it can be shown that the assumptions are satisfied for the mapping $T_{\zetabar}$
on the product space $\Ebb^m$ restricted to $\Lambda$.  
% Before beginning, note that 
% by Lemma \ref{l:unique cycles}, assumption \eqref{e:cond unique cycles} implies that $ Z $ is a singleton, so that 
% \begin{equation}\label{e:Sj2}
% S_j\equiv \paren{S_0- \sum_{i=1}^{j-1}\zetabar_j}\cup\paren{S_0- \sum_{i=1}^{j}\zetabar_j}
% \quad (j=1,2\dots,m).
% \end{equation}
% 
To see that Assumption~\ref{t:metric subreg convergence a} of Theorem \ref{t:metric subreg convergence}
is satisfied, 
note first that, by condition \eqref{e:S cond} and definition \eqref{e:Sj}, 
$P_{\Omega_j}S_{j+1}\subset S_j$.  
This together with condition \eqref{e:Uj} and Assumption~\ref{t:cp ncvx i} allow one to 
conclude from Theorem \ref{t:subreg proj-ref}\ref{t:subreg proj-ref2b} that  the projector $P_{\Omega_j}$ is 
pointwise almost firmly nonexpansive at each $y_j\in S_j$ with violation 
$
 \varepsilontilde_j
$
on $U_j$ given by \eqref{e:e a cp}.  
Then by Proposition \ref{t:av-comp av}\ref{t:av-comp av iii} the cyclic projections mapping $P_0$ is 
pointwise almost averaged at each $y_1\in S_1$ with violation $\varepsilonbar$ and averaging constant 
$\alpha$ given by \eqref{e:e a cp} on $U_1$.  Since $T_{\zetabar}$ is just $P_0$
shifted by $\zetabar$ on the product space, it follows that 
$T_{\zetabar}$ is pointwise almost averaged at each
$y\in S\equiv S_1\times S_2\times\cdots\times S_m$ with the same violation $\varepsilonbar$
and averaging constant $\alpha$ on $U$. 

Assumption~\ref{t:metric subreg convergence b} of Theorem \ref{t:metric subreg convergence}
for $\Phi_{\zetabar}$
follows from Assumptions~\ref{t:cp ncvx ii}-\ref{t:cp ncvx iv} and Proposition~\ref{t:msr cp}. 
This completes the proof.
\end{proof}

\begin{cor}[global $R$-linear convergence of convex cyclic projections]\label{t:cp cvx}
Let the sets $\Omega_j$ ($j=1,2,\dots,m$) be nonempty, closed and convex, let 
$S_0=\Fix P_0\neq \emptyset$ and let $S=S_1\times S_2\times \cdots\times S_m$ for $S_j$ 
defined by \eqref{e:Sj}.  
Let $\Lambda\equiv W(\zetabar)$ for $\zetabar\in \Zcal(u)$ and any $u\in S_0$.   
% Define $\Phi_{\zetabar}\equiv T_{\zetabar}-\Id$ and 
% $\Psi\equiv (P_\Omega-\Id)\circ\Pi$. 
Suppose, in addition, that 
\begin{enumerate}
\item[(b$'$)]\label{t:cp cvx ii} for each $\xhat = (\xhat_1, \xhat_2, \dots, \xhat_m)\in S$, the collection of sets 
$\{\Omega_1, \Omega_2, \dots, \Omega_m\}$ is subtransversal at $\xhat$ for $\zetabar= \xhat-\Pi\xhat$
relative to $\Lambda$ with neighborhood $U\supset S$;
 \item[(c$'$)]\label{t:cp cvx iii} 
there exists a positive constant $\sigma$ such that 
\[
   \dist\paren{\zetabar, \Psi(x)}\leq \sigma \dist(0, \Phi_{\zetabar}(x))
\]
holds whenever $x\in \Lambda\cap U$ with $x_1\in\Omega_1$.
\end{enumerate}
Then the sequence $\paren{x^k}_{k\in \Nbb}$ generated by $x^{k+1}\in T_{\zetabar}x^k$
seeded by any point $x^0\in W(\zetabar)$ 
with $x_1^0\in \Omega_1$ satisfies
\[
 \dist\paren{x^{k+1}, \Fix T_{\zetabar}\cap S}\leq c\dist(x^k, S)
\]
for all $k$ large enough where 
\begin{equation}%%\label{e:c cp cvx}
c\equiv\sqrt{1-\frac{1-\alpha}{\alpha\kappabar^2}} <1
\end{equation}%
with 
 $\kappabar=\kappa\sigma$ for $\kappa$ a constant of metric subregularity of 
 $\Psi$ for $\zetabar$ on $U$ relative to $\Lambda$ and $\alpha$ given by \eqref{e:e a cp}.  In other words,  
$\dist\paren{x^{k}, \Fix T_{\zetabar}\cap S}\to 0$, and hence 
$ \dist\paren{x_1^{k}, \Fix P_0\cap S_0}\to 0$,  at least R-linearly with rate $c<1$. 
\end{cor}

\begin{proof}
	By Lemma \ref{l:unique cycles},  $ Z =\Zcal(u)=\{\zetabar\}$ for any $u\in S_0$.  Moreover, 
	since $\Omega_j$ is convex, the projector is single-valued and firmly nonexpansive, and further the
	conditions \eqref{e:neighborhood relations} are satisfied with $U_j=\Ebb$ 
	($j=1,2,\dots,m$)
	since $S_0=S_1$ and
\begin{equation}%
P_{\Omega_j}\paren{S_1 - \sum_{i=1}^{j}\zetabar_j} = P_{\Omega_j}\paren{S_{j+1}} = S_j =
S_1- \sum_{i=1}^{j-1}\zetabar_j
~
\mbox{ for each } ~j=1,2\dots,m.
\end{equation}%	
	Also by convexity, $\Omega_j$ ($j=1,2,\dots,m$) 
	is elementally regular with constant $\varepsilon_j=0$ globally ($U_j=\Ebb$), so 
	Assumption \ref{t:cp ncvx i} of Theorem~\ref{t:cp ncvx} is satisfied.  Moreover, 
	$\Phi_\zetabar^{-1}(0)=S_0$ so condition \ref{t:cp ncvx iv} holds trivially.  
	The result then follows immediately from Theorem \ref{t:cp ncvx}.
\end{proof}

When the sets $\Omega_j$ are affine, then it is easy to see that 
the sets are subtransversal to each other at 
collections of nearest points corresponding to the gap between the sets.  If the cyclic 
projection algorithm does not converge in one step (which it will in the case of either parallel or 
orthogonally arranged sets) the above corollary shows that cyclic projections converge linearly 
with rate $\sqrt{1-\kappa}$ where $\kappa$ 
is the constant of metric subregularity, reflecting the angle between the affine subspaces.  This much 
for the affine case has already been shown in \cite[Theorem~5.7.8]{BauBorLew97}.  
\begin{remark}[global convergence for nonconvex alternating projections]\label{r:ap saf}
 Convexity is not necessary for global linear convergence of alternating projections.  This 
 has been demonstrated using earlier versions of the theory presented here 
 for sparse affine feasibility in 
 \cite[Corollary III.13 and Theorem III.15]{HesseLukeNeumann14}.  
 A sufficient property for global results in sparse affine feasibility is 
 a common {\em restricted isometry property}  \cite[Eq. (32)]{HesseLukeNeumann14}
 familiar to experts in signal processing with sparsity constraints.  
 The restricted isometry 
 property was shown in  \cite[Proposition III.14]{HesseLukeNeumann14} to imply  {\em transversality} of the affine subspace with 
 all subspaces of a certain dimension.
\end{remark}

\begin{eg}[an equilateral triangle -- three affine subspaces with a hole]\label{eg:triangle}
Consider the problem specified by the following three sets in $\mathbb{R}^2$
\begin{alignat*}{2}
  \Omega_1 &= \mathbb{R}(1,0) &&= \{x\in\mathbb{R}^2\mid \langle (0,1),x\rangle =0\},  \\
  \Omega_2 &= (0,-1) + \mathbb{R}(-\sqrt{3},1) &&= \{x\in\mathbb{R}^2\mid \langle (-\sqrt{3},1),x\rangle =\sqrt{3}\}, \\
  \Omega_3 &= (0,1) + \mathbb{R}(\sqrt{3},1) &&= \{x\in\mathbb{R}^2\mid \langle (\sqrt{3},1),x\rangle =1\}. 
\end{alignat*}
% In this example we focus only on the local behavior near the point $\ubar = (-1/3,0)$ which is the unique fixed point of $P_0$.
% The corresponding unique fixed loop and unique displacement vector are, respectively, given by
% \begin{align*}
% & \xbar = (\xbar_1,\xbar_2,\xbar_3) = \paren{\paren{-1/3,0},\paren{-1/3,2/\sqrt{3}},\paren{2/3,1/\sqrt{3}}},\\
% & \zetabar = (\zetabar_1,\zetabar_2,\zetabar_3) = \left( (-1,-1/\sqrt{3}), (1,-1/\sqrt{3}), (0,-2/\sqrt{3})\right).
% \end{align*}
% Take any triple $x=(x_1,x_2,x_3)\in W(\zetabar)$ with $x_1\in\Omega_1$. We may represent $x_1=(a-1/3,0)$ 
% where $a\in\mathbb{R}$ is a parameter. Since the affine spaces have everywhere unique projections, the unique 
% element of $\mathcal{C}(x)$ is
% \[
% x' = \paren{ \paren{-a/8 - 1/3,0} , \paren{-a/8 - 1/3,-\sqrt{3}a/8 + 2/\sqrt{3}}, \paren{a/4 + 2/3, -\sqrt{3}a/4+1/\sqrt{3}}}.
%  \]
The following statements regarding the assumptions of Corollary~\ref{t:cp cvx} are easily verified.
%, where $L=W(\zetabar)$.
\begin{enumerate}[(i)]
\item The set $S_0=\Fix P_0=\{ (-1/3,0) \}$.
\item There is a unique fixed point $\xbar = (\xbar_1,\xbar_2,\xbar_3) = \paren{\paren{-1/3,0},\paren{-1/3,2/\sqrt{3}},\paren{2/3,1/\sqrt{3}}}$.
\item The set of difference vectors is a singleton: 
\[ 
 Z  =\klam{\zetabar}= \klam{(\zetabar_1, \zetabar_2, \zetabar_3)} = \klam{  \left( (0,-2/\sqrt{3}), (-1,1/\sqrt{3}), (1,1/\sqrt{3})\right)}.
\]
\item The sets $S_1,\,S_2$ and $S_3$ are given by
   \begin{alignat*}{2}
    S_1 &= S_0-\zetabar_1 &&= \klam{(-1/3, 2/\sqrt{3})}\\
    S_2 &= S_0-\zetabar_1-\zetabar_2 &&= \klam{(2/3, 1/\sqrt{3})}\\
    S_3 &= S_0 &&= \klam{(-1/3,0)}.
    \end{alignat*}
\item Condition \eqref{e:S cond} is satisfied and condition \eqref{e:Uj} is satisfied  with $U_j=\Rtw$ ($j=1,2,3$).
%\item The affine subspace $L$ in Theorem \ref{t:cp ncvx} is $L=\Rtw\times \Rtw\times\Rtw$.
\item  For $j\in\{1,2,3\}$, $\Omega_j$ is convex and hence elementally regular at $\xbar_j$ with constant $\varepsilon_j=0$ 
    \cite[Proposition~4]{KruLukNgu16}.
%\item By Proposition~\ref{t:av-comp av}\ref{t:av-comp av iii}, $P_0$ is pointwise almost averaged with averaging constant $\alpha=3/4$ 
% on $\Rtw$ (violation $\varepsilonbar=0$), and thus $T_\zetabar$ is pointwise almost averaged with averaging constant $\alpha=3/4$ 	  
% on $\paren{\Rtw}^3$.
%\item The collection $\{\Omega_1,\Omega_2,\Omega_3\}$ is subtransversal  at 
%	% the collection of points $\klam{\xbar_1, \xbar_2, \xbar_3}$ 
%	$\xbar$
%	relative to $W(\zetabar)$ for $\zetabar$.
\item The mapping $\Psi$ is metrically subregular for $\zetabar$ on $\paren{\Rtw}^3$ with constant $\kappa=\sqrt{2}$ relative to $W(\zetabar)$: 
	\[
	    \dist\paren{x,\Psi^{-1}(\zetabar)\cap W(\zetabar)} \leq \sqrt{2}\, \dist\paren{\zetabar,\Psi(x)} \quad\forall x\in \paren{\Rtw}^3.
	\]
%\item The infimum of possible numbers $\rho$ and $\sigma$ are given by, respectively,
%\begin{align*}
%\frac{\dist(\zetabar,\Psi(x))}{\dist(0,\Phi_{\zetabar}(x))} = \frac{4}{9}\sqrt{2},\; \mbox{ and }\;
%\frac{\dist(\zetabar,\Psi(x'))}{\dist(0,\Phi_{\zetabar}(x))} = \frac{\sqrt{21}}{18}.
%\end{align*}
\item For all $x\in W(\zetabar)$, % and $x'\in\mathcal{C}(x)$, 
	the inequality
 $\dist(\zetabar,\Psi(x)) \leq \sigma\, \dist(0,\Phi_{\zetabar}(x))$ % ,\qquad \dist(\zetabar,\Psi(x'))\leq \rho\,\dist(0,\Phi_{\zetabar}(x)),\]
  holds with $\sigma = 4\sqrt{2}/9$.%  and $\rho=\sqrt{21}/18$.
\end{enumerate}
 The assumptions of Corollary~\ref{t:cp cvx} are satisfied. Furthermore, Proposition~\ref{t:msr cp} shows that the 
mapping $\Phi_{\zetabar}$ is metrically subregular for $0$ on $\paren{\Rtw}^3$ relative to $W(\zetabar)$ with constant 
$\kappabar=\kappa\sigma=\sqrt{2}\times 4\sqrt{2}/9=8/9.$  Altogether, 
Corollary~\ref{t:cp cvx} yields that, from any starting point,  the cyclic projection method converges linearly to $\ubar$ 
with rate at most  $c=\sqrt{37}/8$.
\end{eg}

The next example is new and rather unexpected. 
\begin{eg}[two non-intersecting circles]\label{eg:circles}
 Fix $r>0$ and consider the problem specified by the following two sets in $\mathbb{R}^2$
  \begin{equation}%
  	\begin{split}
     \Omega_1 &= \{x\in\mathbb{R}^2\mid \|x\|=1\}, \\
     \Omega_2 &= \{x\in\mathbb{R}^2\mid \|x+(0,1/2+r)\|=2+r\}.
   \end{split}
  \end{equation}%  
 In this example we focus on (local) behavior around the point $\ubar=(0,1)$. For $U_1$, 
a sufficiently small neighborhood of $\ubar$, the following statements regarding the assumptions of Theorem~\ref{t:cp ncvx} can be verified.
 \begin{enumerate}[(i)]
 	\item $S_0 = \Fix P_0\cap U_1 = \{\ubar\}=\{(0,1)\}$;
 	\item $\xbar=\paren{\xbar_1,\xbar_2}=\paren{\ubar,(0,3/2)}=\paren{(0,1),(0,3/2)}$;
 	\item $\mathcal{Z}= \{\zetabar\} = \{(\zetabar_1,\zetabar_2)\} = \{\left( (0,-1/2), (0,1/2) \right)\}$;
 	\item the sets $S_1$ and $S_2$ are given by
 	        \begin{alignat*}{2}
		 		S_1&=S_0-\zetabar_1 &&= \klam{(0,1/2)}\\
		 		S_2&=S_0-\zetabar_1-\zetabar_2 &&= \klam{(0,1)};
		 	\end{alignat*}
 	\item  \eqref{e:S cond} is satisfied, and \eqref{e:Uj} holds with $U_1$ already given and $U_2$ equal to a 
 	scaled-translate of $U_1$-- more precisely,  $U_1$ and $U_2$ are related by
 	 $$ U_2 = \frac{2+r}{\dist\paren{\ubar,(0,-\frac{1}{2}-r)}}\, U_1 + (0,1/2); $$
 	 
 	\item $L=\mathbb{R}^2\times\mathbb{R}^2$;
 	\item for $j\in\{1,2\}$, $\Omega_j$ is uniformly elementally regular at $\xbar_j$ for any $\varepsilon_j\in(0,1)$ \cite[Example~2(b)]{KruLukNgu16};
 	\end{enumerate}
 	   In order to verify the remaining conditions of Theorem~\ref{t:cp ncvx}, we use the following parametrization: any double 
	    $x=(x_1,x_2)\in W(\zetabar)$ with $x_1\in\Omega_1$ may be expressed in the form $x_1=(b,\sqrt{1-b^2})\in\Omega_1$ 
	    where $b\in\mathbb{R}$ is a parameter.
 	\begin{enumerate}[(i),resume]
 	\item \label{ex2:subtrans} $\{\Omega_1,\Omega_2\}$ is subtransversal at $\bar{x}$ relative to $W(\bar\zeta)$, {\em i.e.,} $\Psi$ 
 	is metrically subregular at $\xbar$ for $\zetabar$ on $U$ (metrically regular at $(\bar{x}, \zetabar)$ on $U\times\{\zetabar\}$) 
 	relative to $W(\bar\zeta)$  with constant 
 	\[
 	\kappa\lim_{b\to 0}\frac{\dist\paren{x,\Psi^{-1}(\zetabar)\cap W(\zetabar)}}{\dist\paren{\zetabar,\Psi(x)}} = \frac{3(2r+3)}{\sqrt{2r^2+6r+9}}.
 	\]
 	  \item For any $\rho>0$ such that
 	   	  \begin{align*}
 	   	  	\rho > \lim_{b\to 0}\frac{\dist(\zetabar,\Psi(x))}{\dist(0,\Phi_{\zetabar}(x))} =
 	   	  	\frac{\sqrt{2} \sqrt{2 \, r^{2} + 6 \, r + 9}\, {\left(2 \, r +
 	   	  			3\right)}}{2 \, \sqrt{4 \, r^{2} + 12 \, r + 13}\, {\left(r + 2\right)}},
 	   	  \end{align*}
 	  the following inequality holds
 	  $$\dist(\zetabar,\Psi(x)) \leq \rho\, \dist(0,\Phi_{\zetabar}(x))$$  
 	  for all $x\in W(\zetabar)$ sufficiently close to $\xbar$.

 \end{enumerate}
  The assumptions of Theorem~\ref{t:cp ncvx} are satisfied. Furthermore, the proof of Proposition~\ref{t:msr cp} shows that 
  the mapping $\Phi_{\zetabar}$ is metrically subregular at $\xbar$ for $0$ relative to $W(\zetabar)$ on $U$ with the constant 
  $\kappabar$ equal to the product of constant of subtransversality $\kappa$ in \ref{ex2:subtrans} and $\rho$. That is,
  $$\kappabar = \frac{3 \sqrt{2} {\left(2 \, r +
  		3\right)^2}}{2 \, \sqrt{4 \, r^{2} + 12 \, r + 13}\, {\left(r + 2\right)}} . $$ 
  Altogether, Theorem~\ref{t:cp ncvx} yields that, for any $c$ with
  $$1>c>\sqrt{1 -        	
  	\frac{{\left(4 \, r^{2} + 12 \, r + 13\right)} {\left(r +
  			2\right)}^{2}}{9 \, {\left(2 \, r + 3\right)}^{4}} },$$
  there exists a neighborhood of $\bar{u}$ such that the cyclic projection method converges linearly to $\ubar $ with rate~$c$. 
  
%
%  \noindent\hrulefill
% 
% 
%  The corresponding unique fixed loop and unique difference vector are, respectively, given by
%   \begin{align*}
%    \bar{x} &= (\bar{x}_1,\bar{x}_2) = \left( (0,1), (0,3/2) \right), \\
%    \bar\zeta &= (\bar\zeta_1,\bar\zeta_2) = \left( (0,-1/2), (0,1/2) \right).
%   \end{align*}
%  Take any double $x=(x_1,x_2)\in W(\bar\zeta)$ with $x_1\in\Omega_1$. We may represent $x_1=(b,\sqrt{1-b^2})\in\Omega_1$ 
%where $b\in\mathbb{R}$ is a parameter. Since the projection onto a sphere is unique except at the circle's center, we see that 
%$\mathcal{C}(x)$ is a singleton, say $\{x'\}=\mathcal{C}(x)$. As in the previous example, an expression for $x'$ can be given 
%in terms of the parameter $b$, however, in this case, the expression is significantly more complex and requires the assistance 
%of a computer algebra system.  
%\texttt{Sage} and can be found in Figure~\ref{fig:x dash}.\todo{i dont think we want to include this}
\end{eg} 
\begin{remark}[non-intersecting circle and line]
 A similar analysis to Example \ref{eg:circles} can be performed for the case in which the second circle $\Omega_2$ is replaced with the line 
$(0,3/2)+\mathbb{R}(1,0)$. 
Formally, this corresponds to setting the parameter $r=+\infty$ in Example \ref{eg:circles}. Although there are some technicalities involved 
in order to make such an argument fully rigorous, a separate computation has verified the constants obtained in this way agree 
with those obtained from a direct computation.  When the circle and line are tangent, then Example \ref{eg:4.4} shows how 
sublinear convergence of alternating projections can be quantified.
\end{remark}

\begin{eg}[phase retrieval]\label{eg:pr cp}
 In the discrete version of the phase retrieval problem 
\cite{Luke02a, BCL1, BurkeLuke03, Luke05a, Luke12, HLST, Luke17}, the constraint sets are of the form
 \begin{equation}\label{e:pr}
\Omega_j=\set{x\in\Cn}{|(A_jx)_k|^2=b_{jk}, ~k=1,2,\dots, n},
 \end{equation}
where $\map{A_j}{\Cn}{\Cn}$ is a unitary linear operator (a Fresnel or Fourier transform depending on whether the 
data is in the near field or far field respectively)  for $j=1,2,\ldots, m$, 
possibly together with an additional support/support-nonnegativity constraint, $\Omega_0$.  It is elementary to show that 
the sets $\Omega_j$ are elementally {\em regular} (indeed, they are semi-algebraic \cite[Proposition 3.5]{HLST} and 
prox-regular \cite[Proposition 4]{KruLukNgu16}, and $\Omega_0$ is convex) so condition \ref{t:cp ncvx i}
of Theorem \ref{t:cp ncvx} is satisfied for each $\Omega_j$ with some violation $\varepsilon_j$ 
on  local neighborhoods.  Subtransversality of the collection of sets at 
a fixed point $\xbar$ of $P_0$ can only be violated when the sets are locally 
parallel at $\xbar$ for the corresponding 
difference vector.  It is beyond the focus of this paper to show that this cannot happen in almost all 
practical instances, establishing that condition \ref{t:cp ncvx ii} of Theorem \ref{t:cp ncvx} holds.  
The remaining conditions \ref{t:cp ncvx iii}-\ref{t:cp ncvx iv} are technical and Example \ref{eg:circles} --
which essentially captures the geometry of the sets in the phase retrieval problem -- shows that 
these assumptions are satisfied.  Theorem \ref{t:cp ncvx} then shows that 
near {\em stable} fixed points (defined as those which correspond to local best approximation points 
\cite[Definition 3.3]{Luke08}) the method of alternating projections {\em must} converge linearly.  
In particular, the cyclic projections algorithm can be expected to converge linearly on neighborhoods of 
stable fixed points {\em regardless of whether or not the phase sets intersect}.
This improves, in several ways, the local linear convergence result obtained in 
\cite[Theorem 5.1]{Luke12} which established local linear convergence of {\em approximate} 
alternating projections to local solutions with more general gauges for the case of two sets:  
first, the present theory handles more than two sets, which is relevant for wavefront sensing \cite{Luke02a,HLST};  
secondly, it does not require that the intersection of the constraint sets (which are expressed in terms of noisy, incomplete 
measurement data) be nonempty.  This is in contrast to recent studies of the phase retrieval 
problem (of which there are too many to cite here) which require the assumption of feasibility, 
despite evidence, both numerical and experimental, to the contrary.  
Indeed,  according to elementary noncrystallographic diffraction theory, since the experimental measurements -- 
the constants $b_{jk}$ in the sets $\Omega_j$ 
defined in \eqref{e:pr} -- are finite samples of 
the continuous Fourier/Fresnel transform, 
there can be no overlap between the set of points satisfying the measurements and the set 
of compactly supported objects specified by the constraint $\Omega_0$.  
Adding another layer to this fundamental inconsistency is the fact that the measurements are 
noisy and inexact.  The presumption that these sets have nonempty intersection is neither reasonable nor 
necessary.  Regarding  approximate/inexact evaluation
of the projectors studied in \cite{Luke12}, we see no obvious impediment to such an extension 
and this would indeed be a valuable endeavor, again, beyond the 
scope of this work.  Toward global convergence results, Theorem~\ref{t:cp ncvx} indicates that the focus
of any such effort should be on determining when the set of {\em difference vectors} is unique rather than 
focusing on uniqueness of the intersection as proposed in \cite{CandesEldarStrohmerVononinski}
and \cite{HesseLukeNeumann14}.  
\end{eg}

\subsection{Structured (nonconvex) optimization}\label{s:str opt}
We consider next the problem 
\begin{equation*}\tag{$\Pcal$}\label{e:smooth-nonsmooth}
   \ucmin{f(x)+g(x)}{x\in\Ebb}
\end{equation*}
under different assumptions on the functions $f$ and $g$.  At the very least, we will assume that these
functions are proper, lower semi-continuous (l.s.c.) functions.   
\subsubsection{Forward-backward}\label{s:fb}
We begin with the ubiquitous {\em forward-backward} algorithm: given $x^0\in \Ebb$, 
generate the sequence $\paren{x^k}_{k\in\Nbb}$ via 
\begin{equation}\label{e:fb}
   x^{k+1}\in T_{\rm FB}(x^k)\equiv \prox_{1,g}\paren{x^k-t\nabla f(x^k)}.
\end{equation}
We keep the step-length fixed for simplicity.  
This is a reasonable strategy, obviously, when $f$ is continuously differentiable with 
Lipschitz continuous gradient and when $g$ is convex (not necessarily smooth), 
which we will assume throughout this subsection. 
For the case that $g$ is the indicator function of a set $C$, that is $g=\iota_C$, then \eqref{e:fb} is 
just the projected gradient algorithm for constrained optimization with a smooth objective.  
For simplicity, we will take the proximal parameter $\lambda=1$ and use the notation $\prox_{g}$
instead of $\prox_{1,g}$.  The following discussion uses the property of hypomonotonicity 
(Definition \ref{d:sub/hypomonotone}\ref{d:p-hypomonotone}).
%
%	\begin{proposition}[almost averaged: steepest descent] \label{t:gradient averaged}
%		Let $U$ be a nonempty subset of $\Rbb^{n}$ and let $\xbar \in U$. Let $\map{f}{\Rbb^{n}}{\Rbb}$ 
%		be continuously differentiable function with Lipschitz continuous gradient 
%		(with modulus $L$) which is also hypomonotone at $\xbar$ with violation constant 
%		$\tau$ on $U$. Then, for fixed $\beta > 0$ and $t = \alpha\beta$, the mapping $T_{t f} :=
%		\paren{\Id - t\nabla f}$ is almost averaged at $\xbar$ with the violation constant 
%		$\varepsilon =  2\beta\tau + \beta^{2}L^{2}$ and the averaging constant $\alpha = t/\beta 
%		\in \left(0 , 1\right)$ on $U$ whenever $t < \beta$. In particular, if $\nabla  f$ is 
%		strongly monotone with constant $|\tau| > 0$ and Lipschitz continuous with modulus $L$, 
%		then $T_{t f}$ is averaged (that is, violation constant $\varepsilon = 0$) with the averaging 
%		constant $\alpha = tL^{2}/\left(2|\tau|\right) \in \left(0 , 1\right]$ for all $t \leq 
%		2|\tau|/L^{2}$.   
%	\end{proposition}
	\begin{proposition}[almost averaged: steepest descent] \label{t:gradient averaged}
		Let $U$ be a nonempty open subset of $\Rbb^{n}$.
		%and let $\xbar \in U$. 
		Let $\map{f}{\Rbb^{n}}{\Rbb}$ 
		be a continuously differentiable function with calm gradient at $\xbar$ and
		calmness modulus $L$ on the neighborhood $U$ of $\xbar$.  In addition, let $\nabla f$ 
		be pointwise hypomonotone at $\xbar$
		with violation constant $\tau$ 
		on $U$. Choose $\beta>0$ and let $t\in(0,\beta)$. Then the mapping $T_{t,f} :=
		\Id - t\nabla f$ is pointwise almost averaged at $\xbar$
		with averaging constant $\alpha = t/\beta 
		\in \left(0 , 1\right)$ and  violation constant 
		$\varepsilon =  \alpha(2\beta\tau + \beta^{2}L^{2})$ on $U$. 
		If $\nabla  f$ is pointwise 
		strongly monotone at $\xbar$
		with modulus $|\tau| > 0$ (that is, pointwise hypomonotone with constant $\tau<0$) 
		and calm with modulus $L$ on $U$ 
		and $t < 2|\tau|/L^{2}$, then $T_{t,f}$ is pointwise averaged at $\xbar$ 
		with averaging constant $\alpha = tL^{2}/\left(2|\tau|\right) \in (0,1)$ on $U$.
	\end{proposition}
	\begin{proof}
	    Noting that 
  		\begin{equation}%
   			\Id - \paren{\alpha\beta}\nabla  f = \paren{1 - \alpha}\Id + \alpha\paren{\Id - \beta
   			\nabla  f},
		\end{equation}%
   		by definition, $T_{t,f} = \Id - \paren{\alpha\beta}\nabla f$ is pointwise almost averaged
   		at $\xbar$
		with violation $\varepsilon=\alpha\left(2\beta 
		\tau + \beta^{2}L^{2}\right)$ and averaging constant $\alpha \in \left(0 , 
   		1\right)$ on $U$ if and only if $\Id - \beta\nabla  f$ is pointwise almost nonexpansive
   		at $\xbar$ with violation constant $\varepsilon/\alpha=2\beta 
		\tau + \beta^{2}L^{2}$ on $U$.

 		Define $T_{\beta,f} := \Id - \beta\nabla  f$. Then, since $f$ is 
		continuously differentiable with calm gradient at $\xbar$ and calmness modulus $L$  on $U$, and 
		the gradient $\nabla f$ is pointwise hypomonotone at $\xbar$ 
		%a neighborhood of $\xbar$ 
		with violation $\tau$ on $U$, 
		\begin{align}
   			\norm{T_{\beta,f}\paren{x} - T_{\beta,f}\paren{\xbar}}^{2} & = \norm{x - \xbar}
   			^{2} - 2\beta\ip{x - \xbar}{ \nabla  f\left(x\right) - \nabla  f\left(\xbar
   			\right)} + \beta^{2}\norm{\nabla  f\left(x\right) - \nabla  f\left(\xbar\right)}
   			^{2} \nonumber \\
			& \leq \paren{1 + 2\beta\tau + \beta^{2}L^{2}}\norm{x - \xbar}^{2}, \quad \forall x\in U. 
			\label{e:interim1}
		\end{align}
		This proves the first statement.
		
        In addition, if $\nabla  f$ is pointwise strongly monotone (pointwise hypomonotone with $\tau < 0$) at $\xbar$, then from
		\eqref{e:interim1}, $2\beta\tau + \beta^{2}L^{2} \leq 0$ 
		whenever $\beta \leq 2|\tau|/L^{2}$ -- that is, $T_{\beta,f}$ is 
		{\em nonexpansive} -- on $U$ where equality holds when 
		$\beta = 2|\tau|/L^{2}$.
		Choose $\beta=2|\tau|/L^{2}$ and set $\alpha = t/\beta = tL^{2}/\left(2|\tau|\right) \in (0,1)$ 
		since $t < 2|\tau|/L^{2}$.
		The first statement 		
		then yields the result for this case 
		and completes the proof.  
	\end{proof}

	Note the trade-off between the step-length and the averaging property: the smaller the step, the 
	smaller the averaging constant. In the case that $\nabla  f$ is not monotone, the violation 
	constant of nonexpansivity can also be chosen to be arbitrarily small by choosing $\beta$ 
	arbitrarily small, regardless of the size of the hypomonotonicity constant $\tau$ or the Lipschitz 
	constant $L$. This will be exploited in Theorem \ref{t:linear convergence of fb} 
	below. If $\nabla  f$ is strongly monotone, the theorem establishes an upper limit on the stepsize 
	for which nonexpansivity holds, but this does not rule out the possibility that, even for 
	nonexpansive mappings, it might be more efficient to take a larger step that technically renders 
	the mapping only {\em almost nonexpansive}. As we have seen in Theorem 
	\ref{t:metric subreg convergence}, if the fixed point set is {\em attractive} enough, then linear 
	convergence of the iteration can still be guaranteed, even with this larger stepsize.  This 
	yields a local justification of {\em extrapolation}, or excessively 
	large stepsizes.
	\begin{proposition}[almost averaged: nonconvex forward-backward] \label{t:fb averaged}
		Let $\map{g}{\Rbb^{n}}{\extre}$ be proper and l.s.c. with nonempty, pointwise submonotone subdifferential
		at all points on $S_g'\subset U_g'$ with violation $\tau_g$ on $U_g'$ in the sense of \eqref{e:submonotone}, 
		that is, at each $w\in\sd g(v)$ and $v\in S_g'$ the inequality  
		\begin{equation}%%\label{e:sd g submon}
		  -\tau_g\norm{(u+z)-(v+w)}^2\leq \ip{z-w}{u-v}
		\end{equation}%
		holds whenever $z\in\sd g(u)$ for $u\in U_g'$. 
		Let $\map{f}{\Rbb^{n}}{\Rbb}$ be a continuously differentiable function with calm  
		gradient (modulus $L$) which is also pointwise hypomonotone at all $\xbar\in S_{f}\subset U_f$ with 
		violation constant $\tau_f$ on $U_{f}$. 
		Suppose that $T_{t,f}U_f\subset U_g$ where $U_g\equiv \set{u+z}{u\in U_g', z\in \sd g(u)}$ and that
		$T_{t,f}S_{f}\subset S_g$ where $S_g\equiv \set{v+w}{v\in S_g', w\in \sd g(v)}$.  
		Choose $\beta > 0$ and $t \in(0,\beta)$. Then the 
		\textit{forward-backward} mapping $T_{\rm FB} := \prox_{g}\paren{\Id - t\nabla f}$ is pointwise 
		almost averaged at all	$\xbar\in S_f$ with violation constant 
		$\varepsilon =  \paren{1+2\tau_g}\paren{1+t\paren{2\tau_f + \beta L^{2}}}-1$
		and averaging constant $\alpha$ on $U_f$ where
		\begin{equation} \label{e:alphabar}
    			\alpha = 
    			\begin{cases}
    				\frac{2}{3}, & \mbox{ for all }\alpha_{0} \leq \frac{1}{2}, \\
                \frac{2\alpha_{0}}{\alpha_{0} + 1}, & \mbox{ for all }\alpha_{0} > \frac{1}{2},
               \end{cases}
  			\quad \text{and} \quad 
  			\alpha_{0} = \frac{t}{\beta}.
		\end{equation} 
	\end{proposition}	
	\begin{proof}
   		The proof follows from Propositions \ref{t:av-comp av} and \ref{t:gradient averaged}. Indeed, 
		by Proposition \ref{t:gradient averaged}, the mapping $T_{t,f} := \Id - t\nabla f$ is 
		pointwise almost averaged at $\xbar$ with the violation constant $\varepsilon_f =
		\alpha_0\paren{2\beta\tau_f + \beta^{2} L^{2}}$  
		and the averaging constant
		 $\alpha_0 = t/\beta \in \left(0 , 1\right)$ on $U_f$ for $t< \beta$.
		 It is more convenient to write the violation in terms of $t$ as 
		$\varepsilon_f =  t\paren{2\tau_f + \beta L^{2}}$. 
		By Proposition \ref{t:firmlynonexpansive} and Definition \ref{d:sub/hypomonotone}\ref{d:submonotone}, $\prox_g$ is 
		pointwise almost firmly nonexpansive at points $\ybar\in S_g$ with violation 
		$\varepsilon_g=2\tau_g$ on $U_g$, since $\prox_g$
		is the resolvent of $\sd g$ which, by assumption, is pointwise submonotone (see \eqref{e:submonotone}) 
		at points in $S_g'$ with constant $\tau_g$ on 
		$U_g'$.  Also by assumption, $T_{t,f}U_f \subset U_g$ and $T_{t,f}S_f\subset S_g$, so we can apply 
		Proposition \ref{t:av-comp av}\ref{t:av-comp av iii} to conclude that $T_{\rm FB}$ is pointwise 
		averaged at $\xbar\in S_f$ 
		with the violation constant $\paren{1+2\tau_g}\paren{1+t\paren{2\tau_f + \beta L^{2}}}-1$ and the averaging 
		constant $\alpha$ which is given by \eqref{e:alphabar} on $U_f$ whenever $t < \beta$, as 
		claimed. 
	\end{proof}

		\begin{corollary}[almost averaged: semi-convex forward-backward] \label{t:scvx fb averaged}
		Let $\map{g}{\Rbb^{n}}{\extre}$ be proper, l.s.c. and convex.
		Let $\map{f}{\Rbb^{n}}{\Rbb}$ be a continuously differentiable function with calm  
		gradient (calmness modulus $L$) which is also pointwise hypomonotone at all 
		$\xbar\in S_{f}\subset U_f$ with 
		violation constant $\tau_f$ on $U_{f}$. 
		Choose $\beta > 0$ and $t\in(0,\beta)$. Then the 
		\textit{forward-backward} mapping $T_{\rm FB} := \prox_{g}\paren{\Id - t\nabla f}$ is 
		pointwise almost averaged at all  
		$\xbar\in S_f$ with violation constant $\varepsilon =  t\paren{2\tau_f + \beta L^{2}}$
		and averaging constant  $\alpha$ given by \eqref{e:alphabar} on $U_f$. 
	\end{corollary}	
	\begin{proof}
   		This is a specialization of Proposition~\ref{t:fb averaged} to the case where $g$ is convex.  
   		In this setting, $\partial g$ is a maximally monotone
   		mapping \cite{Minty62, Moreau65}, and hence  submonotone at all points in $\Rbb^n$ 
   		with no violation ({\em i.e.,} $\tau_g=0$). The assumptions
		$T_{t,f}U_f\subset U_g$ where $U_g\equiv \Rbb^n$ and  
		$T_{t,f}S_{f}\subset S_g$ where $S_g\equiv \Rbb^n$ of Proposition \ref{t:fb averaged} 
		are obviously automatically satisfied.  
	\end{proof}

	As the above proposition shows, the almost averaging property comes relatively naturally.  
	A little more challenging is to show that Assumption \ref{t:metric subreg convergence b}  of Theorem   
	\ref{t:metric subreg convergence} holds for a given application.  The next 
	theorem is formulated in terms of metric subregularity, but for the forward-backward iteration, 
	the graphical derivative characterization given in Proposition \ref{t:strmsr diff}
	 can allow for a direct verification of the 
	regularity assumptions. 
	\begin{thm}[local linear convergence: forward-backward]
	\label{t:linear convergence of fb}
		Let $\map{f}{\Rbb^{n}}{\Rbb}$ be a continuously differentiable function with calm  
		gradient (modulus $L$) which is also pointwise hypomonotone at all 
		$\xbar\in \Fix T_{\rm FB}\subset U_f$ with 
		violation constant $\tau_f$ on $U_{f}$. 
		Let $\map{g}{\Rbb^{n}}{\extre}$ be proper and l.s.c. with nonempty subdifferential 
		that is pointwise submonotone (Definition \ref{d:sub/hypomonotone}\ref{d:submonotone}) 
		at all $v\in S_g'\subset U_g'$,
		with violation $\tau_g$ on $U_g'$ whenever $z\in\sd g(u)$ for $u\in U_g'$. 
		Let $T_{t,f}U_f\subset U_g$ where $U_g\equiv \set{u+z}{u\in U_g', z\in \sd g(u)}$ and let 
		$T_{t,f}\Fix T_{\rm FB}\subset S_g$ where $S_g\equiv \set{v+w}{v\in S_g', w\in \sd g(v)}$.
		If, for all $t\geq 0$ small 
		enough, $\Phi_{\rm FB}\equiv T_{\rm FB}-\Id$ is metrically subregular for $0$ 
		on $U_f$ with modulus $\kappa \leq \kappabar <1/\paren{2\sqrt{\tau_g}}$,
		then for all $t$ small enough, the forward-backward iteration $x^{k + 1} \in T_{\rm FB}x^{k}$ 
		satisfies $\dist\paren{x^{k}, \Fix T_{\rm FB}}\to 0$ at least linearly for all $x^{0}$ close enough to 
		$\Fix T_{\rm FB}$.  In particular, if $g$ is convex, and $\kappabar$ is finite, then the 
		distance of the iterates to $\Fix T_{\rm FB}$ converges linearly to zero from any initial point $x^0$ close enough 
		provided that the stepsize $t$ is sufficiently small. 
	\end{thm}
	\begin{proof}  
		Denote the averaging constant of the inner forward mapping $T_{t,f}
		\equiv \Id-t\nabla f$ by  $\alpha_0$.  Since, by Proposition \ref{t:gradient averaged}  
		the stepsize $t$, $\alpha_0$ and $\beta$ are all relative, for convenience we fix $\alpha_0=1/2$ so that 
		$t = \beta/2$. From Proposition \ref{t:fb averaged} it then holds that the forward-backward 
		mapping $T_{\rm FB}$ is pointwise almost averaged at all $\xbar\in \Fix T_{\rm FB}$ with the violation constant 
		$\varepsilon = \paren{1+2\tau_g}\paren{1+\beta/2\paren{2\tau_f + \beta L^{2}}}-1$ and the 
		averaging constant $\alpha=2/3$ 
		(given by \eqref{e:alphabar}) on $U_f$. 
		Hence Assumption \ref{t:metric subreg convergence a} of Theorem
		\ref{t:metric subreg convergence} is satisfied with $S=\Fix T_{\rm FB}$.  
		By assumption, for all $t$ (hence $\beta$) small enough, $\Phi_{\rm FB}$ is 
	        metrically subregular for $0$ on $U_f$ with modulus at most $\kappabar$, so 
		by Corollary \ref{t:str metric subreg convergence}, for all $x$ close enough to $\Fix T_{\rm FB}$
		\begin{equation}%%\label{e:metric subreg linear conv fb}
		    \dist\paren{x^+, \Fix T_{\rm FB}}
		      % \leq \dist\paren{x^{j+1}, \Fix T\cap S}
		    \leq c\dist\paren{x, \Fix T_{\rm FB}}
		\end{equation}%
	      where $x^+\in T_{\rm FB}x$ and 
	      $c\equiv \sqrt{1+\varepsilon-\tfrac{1}{2\kappabar^2}}$.   
	      By assumption, the constant $\kappabar$ is suitable
	      % constant of metric subregularity 
	    for all $t$ small enough, but the violation 
	      $\varepsilon= 2\tau_g +o(t)$ can be 
	      made arbitrarily close to $2\tau_g$ simply by taking the stepsize $t=\beta/2$ small enough.  Hence, 
	      $c < 1$ for all $\beta > 0$ with $2\tau_g+\beta/2\left(2\tau_f + \beta L^2\right)+o(\beta^2) < 1/\kappabar^{2}$.    
	      In other words,  for all $x^{0}$ close enough to $\Fix T_{\rm FB}$, and all $t$ (or $\beta$) small enough, 
		convergence of the forward-backward iteration is at 
		least linear with rate at most 
		$c\equiv\sqrt{1+\varepsilon-\tfrac{1-\alpha}{\kappa^2\alpha}}<1$.
		
		If $g$ is convex, then as in Corollary \ref{t:scvx fb averaged}, $\tau_g=0$, so it suffices simply 
		to have $\kappabar$ bounded.  
	\end{proof}
% 	\begin{remark}
%    		The constant $c$ in Theorem \ref{t:linear convergence of fb} is given by $c = \sqrt{1 + 
%    		\varepsilon - \left(1 - \alpha\right)/\left(\alpha\kappa^{2}\right)}$ where $\varepsilon = 4\beta
%    		\left(1 + \beta\right)$. It follows that $c < 1$ for all $\beta > 0$ for which $4\beta
%    		\left(1 + \beta\right) < \left(1 - \alpha\right)/\left(\alpha\kappa^{2}\right)$ where $\alpha$ 
%    		is given by \eqref{e:alphabar} for $t < \beta$.    
% 	\end{remark}

	\begin{cor}[global linear convergence: convex forward-backward]
	\label{t:linear convergence cvx fb}
%		Let  $\xbar\in S_f=\Fix T_{\rm FB}$.  
		Let $\map{f}{\Rbb^{n}}{\Rbb}$ be a continuously differentiable function with calm  
		gradient (modulus $L$) which is also pointwise strongly monotone at all 
		$\xbar\in \Fix T_{\rm FB}$ on $\Rbb^n$. 
		Let $\map{g}{\Rbb^{n}}{\extre}$ be proper, convex and l.s.c. 
		Let $T_{t,f}\Fix T_{\rm FB}\subset S_g$ where $S_g\equiv \set{v+w}{v\in S_g', w\in \sd g(v)}$.
		If, for all $t\geq 0$ small 
		enough, $\Phi_{\rm FB}\equiv T_{\rm FB}-\Id$ is metrically subregular for $0$
% at all points in $\Fix T_{\rm FB}$
		on $\Rbb^n$ with modulus $\kappa \leq \kappabar<+\infty$,
		then for all fixed step-length $t$ small enough, the forward-backward iteration $x^{k + 1} = T_{\rm FB}x^{k}$ 
		satisfies $\dist\paren{x^{k}, \Fix T_{\rm FB}}\to 0$ at least linearly for all $x^{0}\in \Rn$.  
	\end{cor}
	\begin{proof}  
	   Note that $\nabla f$ being pointwise strongly monotone is equivalent to $\nabla f$ being pointwise hypomonotone with violation $\tau_f<0$. 
           Proposition \ref{t:fb averaged} then establishes that the forward-backward 
	   mapping $T_{\rm FB}$ is pointwise almost averaged at all $\xbar\in \Fix T_{\rm FB}$ with the violation constant 
	   $\varepsilon = \beta/2\paren{2\tau_f + \beta L^{2}}$ and the 
	   averaging constant $\alpha=2/3$ 	(given by \eqref{e:alphabar}) on $\Rn$.  For all stepsizes small enough, 
	  or equivalently for all $\beta$ small enough it holds that $2\tau_f + \beta L^{2}<0$ and $T_{FB}$ is 
	  in fact pointwise averaged.  Additionally, for all $t$ (hence $\beta$) small enough, $\Phi_{\rm FB}$ is 
	        metrically subregular for $0$ on $\Rn$ with modulus at most $\kappabar<\infty$, so 
		by Corollary \ref{t:str metric subreg convergence}, for all $x$ 
		\begin{equation}%%\label{e:metric subreg linear conv cvx fb}
		    \dist\paren{T_{\rm FB}x, \Fix T_{\rm FB}}
		      % \leq \dist\paren{x^{j+1}, \Fix T\cap S}
		    \leq c\dist\paren{x, \Fix T_{\rm FB}}
		\end{equation}%
	      where  $c\equiv \sqrt{1-\tfrac{1}{2\kappabar^2}}<1$.  This completes the proof.   
	\end{proof}
	
\begin{remark}[extrapolation] 
In the proof of Corollary \ref{t:linear convergence cvx fb} it is not necessary to choose the stepsize small enough 
that  $T_{\rm FB}$ is pointwise averaged.   It suffices to choose the stepsize $t$ small enough that   
 $c\equiv \sqrt{1+\varepsilon-\tfrac{1}{2\kappabar^2}}<1$ where 
$\varepsilon = \beta/2\paren{2\tau_f + \beta L^{2}}$.  In this case, $T_{\rm FB}$ is only 
almost pointwise averaged with violation $\varepsilon$ on $\Rn$.  
\end{remark}

\begin{remark}
   Optimization problems involving the sum of a smooth function and a nonsmooth function 
are commonly found in applications and accelerations to forward-backward algorithms have 
been a subject of intense study \cite{BeckTeboulle09, AttouchPeypouquet15, Nesterov07, 
ChambolleDossal}.  
To this point the theory on quantitative convergence of the iterates is limited to the convex setting 
under the additional assumption of strong convexity/strong monotonicity.  
Theorem \ref{t:linear convergence of fb} shows that locally, convexity of the 
smooth function plays no role  in the convergence of the iterates or the order of convergence,
and strong convexity, much less convexity, of the function $g$ is also not crucial - it is primarily the regularity of the 
fixed points that matters locally.  This agrees nicely with recent global linear convergence results of 
a primal-dual method for saddle point problems that uses {\em pointwise quadratic supportability} in 
place of the much stronger strong convexity assumption \cite{LukShe17}.  
Moreover, local linear convergence
is guaranteed by metric subregularity on an appropriate set without any fine-tuning of the only algorithm parameter $t$, 
other than assuring that this parameter is small enough.  
When the nonsmooth term is the indicator function of some constraint set, then the regularity assumption can be 
replaced by the characterization in terms of the graphical 
derivative \eqref{e:gd strmr} to yield a familiar constraint qualification at fixed points.
\end{remark}

\noindent If the functions in \eqref{e:smooth-nonsmooth} are piecewise linear-quadratic, then 
the forward-backward mapping has polyhedral structure (Proposition \ref{t:polyhedral subdiff}), 
which, following Proposition \ref{t:polyhedral convergence}, allows for easy verification of the conditions for 
linear convergence (Proposition \ref{t:convergence polyhedral fb}).
\begin{defn}[piecewise linear-quadratic functions]\label{d:plq}
 A function $f:\Rn \rightarrow [-\infty, +\infty]$ is called \emph{piecewise linear-quadratic} 
 if $\dom f$ can be represented as the union
of finitely many polyhedral sets, relative to each of which $f(x)$ is given by an expression of the form 
$\frac{1}{2}\langle x, Ax\rangle +\langle a, x \rangle+\alpha $ for some scalar $\alpha\in \mathbb{R}$ vector
$a\in \Rn$, and symmetric matrix $A\in \mathbb{R}^{n\times n} $. 
If $f$ can be represented by a single linear-quadratic equation on $\Rn$, then $f$ is said to be linear-quadratic.
\end{defn}
For instance, if $f$ is piecewise linear-quadratic, then the subdifferential of $f$ and its proximal mapping $\prox_f$ 
are polyhedral \cite[Proposition~12.30]{VA}.  

\begin{propn}[polyhedral forward-backward]\label{t:polyhedral subdiff}
Let $\map{f}{\Ebb}{\Rbb}$ be quadratic and let $\map{g}{\Ebb}{\extre}$ 
be proper, l.s.c. and piecewise linear-quadratic convex.  
The mapping $T_{\rm FB}$ defined by \eqref{e:fb} is single-valued and polyhedral.
\end{propn}
\begin{proof}
Since the functions $f$ and $ g$ are piecewise linear-quadratic, the mappings 
$\Id-\nabla f$  and $\partial g$ are polyhedral. Moreover, since $g$ is convex, the 
mapping $\prox_g$ (that is, the resolvent of $\partial g$) is single-valued and polyhedral \cite[Proposition 12.30]{VA}. 
The mapping $\Id-\nabla f$ is clearly single-valued, so $T_{\rm FB}=\prox_g\paren{\Id-\nabla f}$ is 
also single-valued and polyhedral as the composition of single-valued polyhedral maps.
\end{proof}

\begin{propn}[linear convergence of polyhedral forward-backward]\label{t:convergence polyhedral fb}
Let $\map{f}{\Ebb}{\Rbb}$ be quadratic and let $\map{g}{\Ebb}{\extre}$ 
be proper, l.s.c. and piecewise linear-quadratic convex.  
Suppose $\Fix T_{\rm FB}$ is an isolated point $\{\xbar\}$, where 
$T_{\rm FB}\equiv \prox_g\paren{\Id-t\nabla f}$.  Suppose also that 
the modulus of metric subregularity $\kappa$ of $\Phi\equiv T_{\rm FB}-\Id$ at 
$\xbar$ for $0$ is bounded above by some constant $\kappabar$ for all $t>0$ small enough.  
Then, for all $t$ small enough, the 
forward-backward iteration $x^{k + 1} = T_{\rm FB}\left(x^{k}\right)$ 
converges at least linearly to $\xbar$ whenever $x^{0}$ is close enough to $\xbar$.
\end{propn}
\begin{proof}
By Corollary \ref{t:scvx fb averaged} the mapping $T_{\rm FB}$ is 
pointwise almost averaged with violation $\varepsilon$ proportional to the stepsize $t$. 
By Proposition \ref{t:polyhedral subdiff}
$T_{\rm FB}$ is polyhedral and by Proposition \ref{t:polyhedrality-strong msr} metrically 
subregular at $\xbar$ for $0$ with constant $\kappa$ on some neighborhood $U$ of $\xbar$.   
Since the violation $\varepsilon$ can be made arbitrarily small by taking $t$ arbitrarily small, and since 
the modulus of metric subregularity $\kappa\leq \kappabar<\infty$ for all $t$ small enough, 
the result follows by Proposition~\ref{t:polyhedral convergence}.
\end{proof}

\begin{eg}[iterative soft-thresholding for $\ell_1$-regularized quadratic minimization]
   Let $f(x)=x^TAx + x^Tb$ and $g(x)=\alpha\|Bx\|_1$ for $A\in\Rnn$ symmetric
and $B\in \Rmn$ full rank.  The forward-backward algorithm applied to the 
problem {\em minimize }$ f(x)+g(x)$ is the iterative soft-thresholding algorithm 
\cite{Daubechies04} with fixed step-length $t$ in the forward step $x-t\nabla f(x)= x-t(2Ax +b)$.  The function $g$
is piecewise linear, so $\prox_g$ is polyhedral hence  
the forward-backward fixed point mapping $T_{\rm FB}$ is single-valued and polyhedral.  
As long as $\Fix T_{\rm FB}$ is an isolated point relative to the affine hull of the
iterates $x^{k+1}= T_{\rm FB}x^k$, and the modulus of metric subregularity is independent of 
the stepsize $t$ for all $t$ small enough, then, by Proposition \ref{t:convergence polyhedral fb}
 for small enough stepsize $t$ the iterates
$x^k$ converge linearly to $\Fix T_{\rm FB}$ for all starting points close enough to $\Fix T_{\rm FB}$.  
If $A$ is positive definite (i.e., $f$ is convex) then the set of fixed points is a singleton and 
convergence is linear from any starting point $x^0$. 
\end{eg}

\subsubsection{Douglas--Rachford and relaxations}
The Douglas--Rachford algorithm is commonly encountered in one form or another for solving both feasibility
problems and structured optimization.  In the context of problem \eqref{e:smooth-nonsmooth}  the iteration 
takes the form 
\begin{equation}\label{e:dr}
   x^{k+1}\in T_{\rm DR}(x^k)\equiv \tfrac12\paren{R_f R_g+\Id}(x^k).
\end{equation}
where $R_f\equiv 2\prox_f-\Id$ ({\em i.e.,} the proximal reflector) and $R_g$ is similarly given.

Revisiting the setting of \cite{Luke08}, we use the 
tools developed in the present paper to show when one can expect local linear convergence of the 
Douglas--Rachford iteration.
For simplicity, as in \cite{Luke08}, we will assume that $f$ is convex in order to 
arrive at a clean final statement, though convexity is not needed for local linear convergence.  

\begin{proposition}\label{l:DR1}
 Let $g=\iota_\Omega$ for $\Omega\subset\Rn$ a manifold, 
and let  $\map{f}{\Rn}{\Rbb}$ be convex and linear-quadratic. Fix $\xbar\in\Fix T_{\rm DR}$.  Then 
for any $\varepsilon>0$ small enough, there exists $\delta>0$ such that $T_{\rm DR}$ is single-valued 
and almost firmly nonexpansive 
with violation $\varepsilon_g=4\varepsilon+ 4\varepsilon^2$ on $\Ball_\delta(\xbar)$.  
\end{proposition}
\begin{proof}
 Suppose that $g=\iota_\Omega$ for $\Omega\subset\Rn$ a manifold.  In the language of 
Definition~\ref{d:set regularity}\ref{d:geom reg}, at each point $\xbar\in\Omega$, for any $\varepsilon>0$ there is a $\delta$ such that 
$\Omega$ is elementally regular at $\xbar$ for all
$(a,v)\in\gph{\ncone{\Omega}}$ where $a\in \Ball_\delta(\xbar)$ with constant $\varepsilon$ and neighborhood $\Ball_\delta(\xbar)$.
In other words, $\Omega$ is {\em prox-regular} \cite[Proposition~4(vi)]{KruLukNgu16}.  By 
Theorem~\ref{t:subreg proj-ref}\ref{t:subreg proj-ref3} the reflector $R_g$ is then almost firmly nonexpansive with 
violation $\varepsilon_g\equiv 4\varepsilon+ 4\varepsilon^2$ on $\Ball_\delta(\xbar)$.  
Another characterization of prox-regular sets is that the projector $P_\Omega$ is locally single-valued \cite{PolRockThib00}. 
We can furthermore conclude that $R_g$ is single-valued on $\Ball_\delta(\xbar)$.  
Next, the function $\map{f}{\Rn}{\Rbb}$ is quadratic convex, so  $R_f$ is firmly nonexpansive and single-valued
 as the reflected resolvent of the (maximal monotone) subdifferential of $f$.   By Proposition \ref{t:av-comp av}\ref{t:av-comp av ii}, 
the composition of reflectors $R_fR_g$ is therefore 
almost nonexpansive with violation $\varepsilon_g$ on $\Ball_\delta(\xbar)$.  Then 
by the definition of averaged mappings, the Douglas--Rachford mapping $T_{\rm DR}$ is 
almost firmly nonexpansive with violation $\varepsilon_g$ on $\Ball_\delta(\xbar)$.
\end{proof}

\begin{thm}\label{t:RAAR pr}
Let  $g=\iota_\Omega$ for $\Omega\subset\Rn$ a manifold
and let $\map{f}{\Rn}{\Rbb}$ be linear-quadratic convex.  Let $(x^k)_{k\in \Nbb}$ 
be iterates of the Douglas--Rachford~\eqref{e:dr} algorithm and let $\Lambda=\aff (x^k)$.  
If $T_{DR}-\Id$ is metrically subregular at all points $\xbar\in\Fix T_{DR}\cap \Lambda\neq \emptyset$ relative to 
$\Lambda$ then for all $x^0$ close enough to $\Fix T_{DR}\cap \Lambda$, 
the sequence $x^k$ converges linearly to a point in $\Fix T\cap \Lambda$ with constant at most 
$c=\sqrt{1+\varepsilon - 1/{\kappa^2}}<1$ where $\kappa$ is 
the constant of metric subregularity for $\Phi\equiv T_{\rm DR}-\Id$ on some neighborhood $U$ containing the sequence
and $\varepsilon$ is the violation of almost firm nonexpansiveness on the neighborhood $U$.
\end{thm}
\begin{proof}
 $T_{DR}-\Id$ is metrically subregular at all points in  $\Fix T_{DR}\cap \Lambda$ with constant $\kappa$ on some neighborhood
 $U'$.  By Proposition \ref{l:DR1} there exists a neighborhood $U\subset U'$ on which $T_{DR}$ is single-valued and almost 
 firmly nonexpansive with violation $\varepsilon$ satisfying 
 $\varepsilon < 1/{\kappa^2}$.  By Corollary \ref{t:str metric subreg convergence} the sequence $x^{k+1}=T_{DR}x^k$ then 
 converges linearly to a point in $\Fix T_{DR}\cap \Lambda$ with rate at most $c=\sqrt{1+\varepsilon - 1/{\kappa^2}}<1$.
\end{proof}

\begin{remark}
Assuming that the fixed points, restricted to the affine hull of 
the iterates, are isolated points, polyhedrality was used in \cite{ACL15} to verify that the Douglas--Rachford mapping is 
indeed metrically subregular at the fixed points.  While in principle the graphical derivative formulas 
(see Proposition \ref{t:strmsr diff}) could be used for more general situations, it is not easy to compute the 
graphical derivative of the Douglas--Rachford operator, even in the simple setting above.  This is a theoretical 
bottleneck for the practical applicability of metric subregularity for more general algorithms. 
\end{remark}

\begin{eg}[Relaxed Alternating Averaged Reflections (RAAR) for phase retrieval]\label{eg:RAAR pr}
 Applied to feasibility problems, the Douglas--Rachford algorithm is also described as 
averaged alternating reflections \cite{BCL3}.  Here,  both $f=\iota_A$ and $g=\iota_B$, 
the indicator functions of individual constraint sets.  When the sets $A$ and $B$ are sufficiently regular, 
as they certainly are in the phase retrieval problem, and intersect transversally,
local linear convergence of the Douglas--Rachford algorithm in this instance can be deduced from \cite{Phan16}.  
As discussed in Example \ref{eg:4.4}, however, for any phase retrieval problem arising from a 
physical noncrystallographic diffraction experiment, the constraint sets cannot intersect when finite support
is required of the reconstructed object.    
This fact, seldom acknowledged in the phase retrieval literature, is borne out in the observed instability 
of the Douglas--Rachford algorithm applied to phase retrieval \cite{Luke05a}:  it cannot converge when the constraint
sets do not intersect \cite[Theorem 3.13]{BCL3}.  

To address this issue, a relaxation for nonconvex feasibility was studied in \cite{Luke05a, Luke08} 
that amounts to \eqref{e:dr} where $f$ is the {\em Moreau envelope} of a nonsmooth function 
and $g$ is the indicator function of a sufficiently regular set.  Optimization problems with this structure  
are guaranteed to have solutions. In particular, when $f$ is the Moreau envelope to 
$\iota_A$ with parameter $\lambda$, the corresponding iteration given by \eqref{e:dr} 
can be expressed as a 
convex combination of the underlying basic Douglas--Rachford operator and the projector of the 
constraint set encoded by $g$ \cite[Proposition 2.5]{Luke08}:
\begin{equation}\label{e:RAAR}
   x^{k+1}\in T_{DR\lambda}x^k\equiv \tfrac{\lambda}{2(\lambda+1)}\paren{R_A R_B+\Id}(x^k) +\frac{\lambda}{\lambda+1}P_Bx^k
%  =  \tfrac{\beta}{2}\paren{R_A R_B+\Id}(x^k) +(1-\beta)P_Bx^k
\end{equation}
where $R_A=2P_A-\Id$ and $R_B=2P_B-\Id$.  
In \cite{Luke05a} and the physics literature this is known as 
\emph{relaxed alternating averaged reflections} or RAAR.
As noted in Example \ref{eg:pr cp},  the phase retrieval problem in its many different manifestations 
in photonic imaging has exactly the structure 
of the functions in Theorem \ref{t:RAAR pr}.   If, in addition, the fixed point operator $T_{DR\lambda}$ 
is metrically subregular at its fixed points relative to the affine hull of the iterates, then according to 
Theorem \ref{t:RAAR pr}, for $\lambda$ large enough and for all starting points close enough to the set of 
fixed points,  the Algorithm \eqref{e:RAAR} applied to the phase retrieval problem converges locally linearly 
to a fixed point.   In contrast to the usual Douglas--Rachford 
algorithm and its variants \cite{BCL2}, the RAAR method does not require that the constraint sets 
intersect.  Still, it is an open problem to determine whether $T_{DR\lambda}$ is usually (in some appropriate sense) 
metrically subregular for phase retrieval. 
\end{eg}
% \bibliographystyle{plain}
% \bibliography{master_citations}

\end{document}